\documentclass[reqno]{amsart} \usepackage{graphicx,
  amsmath,amsfonts,amsthm,amssymb, paralist, fullpage}
\newtheorem{theorem}{Theorem}[section]
\newtheorem{corollary}[theorem]{Corollary}
\newtheorem{lemma}[theorem]{Lemma}

\theoremstyle{definition} \newtheorem{definition}[theorem]{Definition}
\theoremstyle{remark} 
\numberwithin{equation}{section}

\newcommand{\g}{\geqslant} 
  \newcommand{\RR}{\mathbb{R}}
\newcommand{\ZZ}{\mathbb{Z}} \newcommand{\CC}{\mathbb{C}}
 
\newcommand{\sph}{\mathbb{S}} \newcommand{\NN}{\mathbb{N}}
\newcommand{\p}{\partial} 
\newcommand{\les}{\leqslant} \newcommand{\lesa}{\lesssim}

\newcommand{\mc}[1]{\mathcal{#1}} \newcommand{\mb}[1]{\mathbf{#1}}
 \newcommand{\lr}[1]{ \langle #1
  \rangle} \newcommand{\ind}{\mathbbold{1}}

\renewcommand{\Re}{\mathrm{Re}}

\DeclareSymbolFont{bbold}{U}{bbold}{m}{n}
\DeclareSymbolFontAlphabet{\mathbbold}{bbold}

 \DeclareMathOperator*{\supp}{supp}

\setdefaultenum{(i)}{(a)}{1}{A}

\begin{document}

\title[Conditional large data scattering for DKG]{Conditional large data
  scattering results for the Dirac-Klein-Gordon system}%
\author[T.~Candy]{Timothy Candy}%
\address[T.~Candy]{Universit\"at Bielefeld, Fakult\"at f\"ur
  Mathematik, Postfach 100131, 33501 Bielefeld, Germany}
\email{tcandy@math.uni-bielefeld.de}%
\author[S.~Herr]{Sebastian Herr}%
\address[S.~Herr]{Universit\"at Bielefeld, Fakult\"at f\"ur
  Mathematik, Postfach 100131, 33501 Bielefeld, Germany}
\email{herr@math.uni-bielefeld.de} \thanks{Financial support by the
  DFG through the CRC ``Taming uncertainty and profiting from
  randomness and low regularity in analysis, stochastics and their
  applications'' is acknowledged.}%

\begin{abstract}
  We obtain conditional results on the global existence and scattering for
  large solutions of the Dirac-Klein-Gordon system in critical spaces in dimension $1+3$. In particular, for bounded solutions
  we identify a space-time Lebesgue norm controlling the global behaviour. The proof relies on
  refined nonlinear estimates involving the controlling norm.
\end{abstract}
\maketitle
\section{Introduction}\label{sec:intro}
The Dirac-Klein-Gordon system for a spinor $\psi:\RR^{1+3}\to \CC^4$ and a scalar field $\phi:\RR^{1+3}\to \RR$ is given as
\begin{equation}\label{eqn:DKG}
\begin{split}
    -i \gamma^\mu \p_\mu \psi + M \psi &= \phi \psi \\
    \Box \phi + m^2 \phi &= \overline{\psi}\psi,
\end{split}
\end{equation}
for the Dirac matrices $\gamma^\mu\in \CC^{4\times 4}$, using the summation convention with respect to $\mu=0,\ldots, 3$, where $\partial_0=\partial_t$ and $\partial_j=\partial_{x_j}$ for $j=1,2,3$. Further, $m,M>0$ are mass parameters and $\overline{\psi}=\psi^\dagger \gamma^0$, where $\psi^\dagger$ is the conjugate transpose. The system \eqref{eqn:DKG} arises as a model for the description of particle interactions in relativistic quantum mechanics, see \cite{Bjorken1964} for more details, and we also refer to \cite{Thaller} for a thorough introduction to Dirac equations.
The aim of the present paper is to initiate the study of large dispersive solutions to \eqref{eqn:DKG}, building on our previous results \cite{Bejenaru2015,Candy2016} on the initial value problem \eqref{eqn:DKG} with small initial data.

Recently, global well-posedness and scattering results have been proven for initial data $(\psi(0),\phi(0),\partial_t\phi(0))$ which are small in spaces close to the critical Sobolev space, which is
\[
L^2(\RR^3)\times \dot{H}^{\frac12}(\RR^3)\times \dot{H}^{-\frac12}(\RR^3).
\]
More precisely, in the non-resonant case $2M>m>0$ \cite{Wang2015} proved a small data result in a critical Besov space with angular regularity, in \cite{Bejenaru2015} we treated the subcritical Sobolev spaces, and  for arbitrary $M,m>0$ we proved this for small initial data in the critical Sobolev space with some small amount of additional angular regularity in \cite{Candy2016}. We refer to the introductions of \cite{Candy2016,Bejenaru2015} for more references to earlier work.

The key in \cite{Wang2015,Bejenaru2015,Candy2016} is the use of the null-structure in \eqref{eqn:DKG} discovered in \cite{D'Ancona2007b} and the construction of custom-made function spaces which allow for global-in-time nonlinear estimates. Here, as a first step towards a better understanding of large solutions to \eqref{eqn:DKG}, we aim at identifying space-time Lebesgue norms which control the global behaviour of dispersive solutions.

We always assume that the Sobolev regularity is $s_0 \g 0$, the angular regularity is $\sigma \g 0$, and the masses $M, m >0$ satisfy
     \begin{align}  \label{eqn:non resonant regime}\text{ either }\;\;& 0<s_0 \ll 1,\; \sigma=0,\;\text{ and  }2M>m>0,\\
 \label{eqn:resonant regime} \text{or }\;\; &s_0=0,\; \sigma>0,\quad \quad \; \text{ and }M, m >0.
    \end{align}
 We take data in the Sobolev space \[H^{s_0}_\sigma(\RR^3) = (1-\Delta_{\sph^2})^{-\sigma} H^{s_0}(\RR^3), \text{ with norm }\|f\|_{H^{s_0}_\sigma} =\| ( 1- \Delta_{\sph^2})^\sigma f \|_{H^{s_0}}.\] Thus in the non-resonant regime,  \eqref{eqn:non resonant regime}, we consider data in the standard Sobolev spaces $H^{s_0} \times H^{s_0 + \frac{1}{2}} \times H^{s_0 - \frac{1}{2}}$, while in \eqref{eqn:resonant regime}, which includes the resonant regime $0<2M<m$, we work in the critical spaces with a small amount of angular regularity $H^0_\sigma \times H^{\frac{1}{2}}_\sigma \times H^{-\frac{1}{2}}_\sigma$, $\sigma>0$.  Given an interval $I \subset \RR$ and $s \g 0$, we define the dispersive type norm $\| u \|_{\mb{D}^s_0(I)} = \| \lr{\nabla}^s u \|_{L^4_{t,x}(I\times \RR^3)}$, and for $s \g 0$,  $ \sigma > 0$, we define
    \begin{equation}\label{eqn:defn of D}  \| u \|_{\mb{D}^s_\sigma(I)} =  \bigg( \sum_{N \in 2^\NN} N^{2\sigma} \|\lr{\nabla}^s H_N u \|_{L^4_{t,x}(I\times \RR^3)}^2 \bigg)^\frac{1}{2},
    \end{equation} where $H_N$ denotes the projection on angular frequencies of size $N$, see \eqref{eq:HN} below.
Our main result is the following.
  \begin{theorem}\label{thm:cond-dkg}
 Let $s_0, \sigma \g 0$ and $M, m >0$ satisfy either \eqref{eqn:non resonant regime} or \eqref{eqn:resonant regime}. Consider any maximal $H^{s_0}_\sigma$-solution
    \[\psi\in C_{\mathrm{loc}}(I^*,H^{s_0}_{\sigma}(\RR^3,\CC^4))\text{ and }(\phi,\partial_t\phi )\in C_{\mathrm{loc}}(I^*,H^{
      s_0+\frac12}_{\sigma} (\RR^3,\RR) \times H^{s_0-\frac12}_{\sigma}(\RR^3,\RR))
    \]
 of \eqref{eqn:DKG} which is bounded, i.e.
    $$ \sup_{t\in I^*}\Big( \| \psi(t) \|_{H^{s_0}_\sigma(\RR^3)} + \| (\phi, \p_t \phi)(t) \|_{H^{s_0+\frac{1}{2}}_\sigma \times H^{s_0 - \frac{1}{2}}_\sigma(\RR^3)}\Big) <+\infty.$$
If $\| \psi \|_{\mb{D}^{-\frac{1}{2}}_\sigma(I^*)}<+\infty$, then we have $I^*=\RR$ and $(\psi,\phi)$ scatters to a free solution as $t\rightarrow \pm \infty$.
  \end{theorem}
The norm $\|\cdot \|_{\mb{D}^{-\frac{1}{2}}_\sigma}$ is scaling-critical for \eqref{eqn:DKG}, and in particular, if $\psi$ is a free solution to the Dirac equation, we have the Strichartz bound $\| \psi \|_{\mb{D}^{-\frac{1}{2}}_\sigma(\RR)} \lesa \| \psi(0) \|_{H^0_\sigma}$. It is important to note that a condition of the form $\| \psi \|_{\mb{D}^{-\frac{1}{2}}_\sigma(I^*)}<\infty$ is necessary to ensure scattering. This follows from the fact that \eqref{eqn:DKG} admits global stationary solutions of the form
    \begin{equation}\label{eqn:solitons}\psi(t)=e^{it\omega}\psi_\ast, \qquad \phi(t)=(m^2-\Delta)^{-1}(\overline{\psi_\ast}\psi_\ast), \quad \omega\in(0,M),\end{equation}
where $\psi_\ast:\RR^3\to \CC^4$ is smooth and exponentially decreasing, see \cite{Esteban1996}. In particular, there exist global solutions to \eqref{eqn:DKG} which do not scatter to free solutions as $t\rightarrow \pm \infty$.

In recent years the notion of type-I and type-II blow-up has played an important role in the study of nonlinear wave equations, see the survey \cite{Kenig2015} for more details and references. Roughly speaking, a maximal solution is of type-I if its spatial Sobolev norm goes to infinity in finite time, and it is of type-II if the spatial Sobolev norm stays finite, but it is ceases to exist for all times. Thus an alternative phrasing of Theorem \ref{thm:cond-dkg} is that any type-II blow-up solution $(\psi,\phi)$ of \eqref{eqn:DKG} with maximal interval $I^\ast$ must satisfy \begin{equation}\label{eq:infinite-d}\| \psi \|_{\mb{D}^{-\frac{1}{2}}_\sigma(I^*)}=+\infty.
\end{equation}

The main technical result behind Theorem \ref{thm:cond-dkg} is Theorem \ref{thm:main-det} which gives good control over any (strong) solutions with small $\mb{D}^{-\frac{1}{2}}_{\sigma}$ norm. We remark that it is possible to replace the hypothesis $\| \psi \|_{\mb{D}^{-\frac{1}{2}}_\sigma(I^*)}<+\infty$ by $\|(\phi,\lr{\nabla}^{-1} \p_t \phi) \|_{\mb{D}^{0}_\sigma(I^*)}<+\infty$, which follows immediately from the statement of Theorem \ref{thm:main-det} and the proof of Theorem \ref{thm:cond-dkg} presented in Section \ref{sec:proof of thm cond-dkg}.

In the case of radial data, Theorem \ref{thm:cond-dkg}, covers the critical regularity case $\psi \in L^2$, $(\phi, \p_t \phi) \in H^{\frac{1}{2}}\times H^{-\frac{1}{2}}$. Strictly speaking however, as the linear Dirac operator does preserve radial solutions, it is better to consider the partial wave subspace of the lowest degree. More precisely, we let $\mc{H}$ be the collection of spinors $\psi_0 \in L^2(\RR^3, \CC^4)$ of the form
    $$ \psi_0(x) = \begin{pmatrix} f(|x|) \begin{pmatrix} 0 \\ 1 \end{pmatrix} \\ g(|x|) \begin{pmatrix} \omega_1 + i \omega_2 \\ \omega_3 \end{pmatrix} \end{pmatrix}$$
with $\omega = \frac{x}{|x|}$, and $f, g \in L^2([0,\infty),r^2dr)$. The subspace $\mc{H}$ is preserved under the linear Dirac operator \cite{Thaller}. Moreover, a computation shows that the subspace $\mc{H}$ is preserved under the nonlinear evolution \eqref{eqn:DKG}, provided that $\phi(0)$ and $\p_t \phi(0)$ are radial. It is worth noting that there are other partial wave subspaces which remain invariant under the evolution of \eqref{eqn:DKG} \cite{Thaller}, however the class $\mc{H}$ is used frequently in the literature, for instance the stationary solutions \eqref{eqn:solitons} belong to $\mc{H}$.  

Applying Theorem \ref{thm:cond-dkg} to data $\psi(0) \in \mc{H}$, and exploiting the conservation of charge, $\| \psi(t) \|_{L^2_x}$, we can then drop the assumption
  $\|\psi\|_{L^\infty(I^*, H^{0}_{\sigma}(\RR^3))}<+\infty$.
  
  \begin{corollary}\label{cor:rad}
Let $m,M>0$. Suppose that
    \[(\psi(0), \phi(0),\partial_t \phi(0))\in L^2(\RR^3) \times H^{\frac12}(\RR^3)\times
    H^{-\frac12}(\RR^3) \]
with $\psi(0) \in \mc{H}$, $\phi(0)$, $\p_t\phi(0)$ radial, and that the corresponding $L^2$-maximal solution \[\psi\in C_{\mathrm{loc}}(I^*,L^2(\RR^3,\CC^4))\text{ and }(\phi,\partial_t\phi )\in C_{\mathrm{loc}}(I^*,H^{
      \frac12} (\RR^3,\RR) \times H^{-\frac12}(\RR^3,\RR))
    \]
    of \eqref{eqn:DKG} satisfies
\[ \sup_{t\in I^*} \| (\phi, \p_t \phi)(t) \|_{H^{\frac{1}{2}}\times H^{- \frac{1}{2}}(\RR^3)} <+\infty \text{ and } \|\lr{\nabla}^{-\frac12} \psi\|_{L^4(I^*\times \RR^3)}<+\infty.\]
Then we have $I^*=\RR$ and $(\phi,\psi)$ scatters to a free
    solution as $t\rightarrow \pm \infty$.
  \end{corollary}

The main novelty in this paper is a certain refinement of the multilinear estimates from \cite{Bejenaru2015,Candy2016} in the sense that we allow for small positive powers of suitable $L^4_{t,x}$ norms on the right hand side, see Theorem \ref{thm:main bilinear}. This is achieved by exploiting the recent progress on bilinear adjoint Fourier restriction estimates \cite{Candy2017}. The particular result from \cite{Candy2017} we will exploit here is summarised in Theorem \ref{thm:bilinear small scale KG}. These $L^4_{t,x}$ norms have the elementary yet crucial property that they can be made arbitrarily small by shrinking the time interval. This has been used successfully to prove global well-posedness and scattering results for Wave and Schr\"odinger equations with polynomial nonlinearities, see e.g. \cite{Bourgain1999,Colliander2006,Killip2013} and the references therein. However, for equations with derivative nonlinearities, such as Wave maps or the Maxwell-Klein-Gordon system, this is more difficult to exploit due to the presence of more involved norms. Recently, there has been significant progress, such as \cite{Sterbenz2010,Sterbenz2010b,Oh2016a,Oh2016b}.
 Our contribution here is closer in spirit to the controlling norm result in \cite{Dodson2015} in the context of Schr\"odinger maps.

The estimates proved in this paper have further applications. In particular, they are applied in \cite{Candy2017b} to prove scattering results for solutions which approximately satisfy a so-called Majorana condition, which defines an open set of large initial data yielding global solutions which scatter.

 The paper is organised as follows. In Section \ref{sec:not} we introduce notation, which is consistent with our previous work \cite{Candy2016}. Further, we define the relevant function spaces and provide some auxiliary results. In Section \ref{sec:dkg} we prove our main local results, namely Theorem \ref{thm:lwp} and Theorem \ref{thm:main-det}, based on nonlinear estimates in Theorem \ref{thm:main bilinear}, whose proof relies on the results of the last two sections. In Section \ref{sec:proof of thm cond-dkg} we prove our main result of this paper, Theorems \ref{thm:cond-dkg} and Corollary \ref{cor:rad}. In Section \ref{sec:prelim} we introduce further notation and preliminary results which are important in the remaining sections. Section \ref{sec:multi-sub} is devoted to the proof of the crucial multilinear estimates in the subcritical regime, while in Section \ref{sec:multi-crit} critical regime is considered, hence completing the proof of Theorem \ref{thm:main bilinear}.

\section{Notation and Function Spaces}\label{sec:not}
Let $\ZZ$ denote the integers and $\NN$ the non-negative integers.
Given a function $f\in L^1_x(\RR^3)$, we let $\widehat{f}(\xi)= \int_{\RR^3} f(x) e^{-ix\cdot \xi} dx$ denote the spatial Fourier transform of $f$. Similarly, for $u \in L^1_{t,x}(\RR^{1+3})$, we take $\widetilde{u}(\tau, \xi) = \int_{\RR^{1+3}} u(t,x) e^{-i(t,x)\cdot(\tau,  \xi)} dx dt$ to be the space-time Fourier transform of $u$. We extend these transforms to tempered distributions in the usual manner. Let $\rho \in C^\infty_0(\RR)$ be a smooth bump function satisfying $\supp \rho \subset \{\frac{1}{2}<t<2\}$ and $\sum_{ \lambda \in 2^{\ZZ}} \rho\big( \tfrac{t}{\lambda}\big) =1$ for $t\ne 0$, and take
        $\rho_{\les 1} = \sum_{\lambda \les 1} \rho(\frac{t}{\lambda})$
for $t \not = 0$, and $\rho_{\les 1}(0)=1$.  Set $\rho_{\les \lambda}(t)=\rho_{\les 1}(\frac{t}{\lambda})$ for $\lambda\in 2^{\ZZ}$. For each $\lambda \in 2^\NN$ and $d \in 2^\ZZ$ we define the Littlewood-Paley multipliers $P_\lambda$, and the modulation localisation operators $C^{\pm, m}_d$ as

    $$ \widehat{P_\lambda f}(\xi) = \rho(\tfrac{|\xi|}{\lambda})\widehat{f}(\xi), \qquad \widehat{P_1 f}(\xi) = \rho_{\les 1}(|\xi|)\widehat{f}(\xi), \qquad \widetilde{C^{\pm, m}_d u}(\tau, \xi)  = \rho( \tfrac{| \tau \pm \lr{\xi}_m|}{d})\widetilde{u}(\tau, \xi)$$
where $\lr{\xi}_m = ( m^2 + |\xi|^2)^\frac{1}{2}$. Thus $P_\lambda$ localises to frequencies of size $\lambda$, and $C^{\pm, m}_d$ localises to space-time frequencies at distance $d$ from the surface $\tau \pm \lr{\xi}_m = 0$. We also define
        $$ C^{\pm, m}_{\les d} = \sum_{ d' \les d} C^{\pm, m}_{d'}$$
which localises to space-time frequencies within $d$ of $\tau \pm \lr{\xi}_m=0$.
We define the localisation operators to angular frequencies of size $N\in 2^{\NN}$ by
\begin{equation}\label{eq:HN}
\begin{split}
(H_N f)(x) ={}& \sum_{\ell \in \NN} \sum_{n=0}^{2\ell}  \rho\big( \tfrac{\ell}{N} \big)  \lr{f(|x|\cdot), y_{\ell, n}}_{L^2(\sph^2)} y_{\ell, n}\big(\tfrac{x}{|x|}\big) \; (N>1),\\  H_1f(x) ={}& \sum_{\ell \in \NN} \sum_{n=0}^{2\ell}  \rho_{\les 1}( \ell) \lr{f(|x|\cdot), y_{\ell, n}}_{L^2(\sph^2)} \, y_{\ell, n}\big(\tfrac{x}{|x|}\big),
\end{split}
\end{equation}
where $(y_{\ell, n})_{n=0, ..., 2\ell}$ denotes an orthonormal basis for the space of homogeneous harmonic polynomials of degree $\ell$, as in \cite[Subsection 7.2]{Candy2016}.

To simplify notation somewhat, in Section \ref{sec:multi-sub} and Section \ref{sec:multi-crit}, we use the shorthand
    $$ C^\pm_d = C^{\pm, 1}_d, \qquad \mc{C}^\pm_d = \Pi_\pm C^{\pm, M}_d$$
where the projections $\Pi_{\pm}$ are defined as
     \begin{equation}\label{eqn:dirac proj}
         \Pi_\pm = \tfrac{1}{2} \Big( I \pm \frac{1}{\lr{\nabla}_M} (   - i \gamma^0 \gamma^j \p_j + M \gamma^0) \Big).
     \end{equation}
The projections $\Pi_\pm$ diagonalise the Dirac operator, for instance we have the identity
    $$  \big( - i \gamma^\mu \p_\mu + M\big) \Pi_\pm \psi = \gamma^0 ( - i \p_t \pm \lr{\nabla}_M) \Pi_\pm \psi, $$
and are also used to uncover the null structure hidden in the product $\overline{\psi} \psi$. See Section \ref{sec:prelim} for further details. \\

Define the propagator for the homogeneous half wave equation as $\mathcal{U}^{\pm}_{ m}(t) = e^{\mp i t \lr{\nabla}_m}$ and let $\mc{U}_{M}$ denote the free Dirac propagator. Explicitly we have
    $$ \mc{U}_M(t) = \mc{U}^{+}_{ M}(t) \Pi_{+} + \mc{U}^{-}_{ M}(t) \Pi_-.$$
Given $t_0 \in I \subset \RR$ and $F \in L^\infty_t L^2_x(I\times \RR^3)$, for $t\in I$ we let $\mc{I}^{\pm, m}_{t_0}[F](t)$ denote the inhomogeneous solution operator for the half-wave equation
    $$ \mc{I}^{\pm, m}_{t_0}[F](t) = i \int_{t_0}^t \mc{U}^\pm_m(t-t') F(t') dt' $$
and $\mc{I}^{M}_{t_0}[F] (t)= \mc{I}^{+, M}_{t_0}[\Pi_+F](t) + \mc{I}^{-, M}_{t_0}[\Pi_-F](t)$ denote the inhomogeneous solution operator for the Dirac equation. Thus, if $- i \p_t u + \lr{\nabla}_m u = F$, then we can write the Duhamel formula as
        $$ u(t) = \mc{U}^+_m(t-t_0) u(t_0) + \mc{I}^{+, m}_{t_0}[F](t).$$
Similarly, for the Dirac equation, if $ -i \gamma^\mu \p_\mu \psi + M \psi = G$, then we have
    $$ \psi(t) = \mc{U}_M(t-t_0)\psi(t_0) + \mc{I}^M_{t_0}[\gamma^0 G](t).$$

  We now define the main function spaces in which we construct solutions. The basic building blocks are the $V^2$ spaces introduced by Koch and Tataru \cite{KochTataru2005}, and studied systematically in \cite{Hadac2009,Koch2014}. Let $\mc{Z} = \{ (t_j)_{j \in \ZZ} \mid t_j \in \RR \text{ and } t_j<t_{j+1} \}$ and $1\leq p<\infty$. For a function $u: \RR \rightarrow L^2_x$, the $p$-variation of $u$ is defined as
        \[| u |_{V^p} = \sup_{(t_j) \in \mc{Z}} \Big( \sum_{j \in \ZZ} \| u(t_j) - u(t_{j-1}) \|_{L^2_x}^p \Big)^\frac{1}{p}.\]
 The normed space $V^p$ is then defined to be all right continuous functions $u: \RR \rightarrow L^2_x(\RR^3)$ satisfying
     \[ \| u \|_{V^p} = \| u\|_{L^\infty_t L^2_x} + | u |_{V^p} < \infty.\]
The space $V^p$ is complete and functions in  $V^p$ have one-sided limits at each point, including $\pm\infty$. Define $V^p_{\pm, m}=\mc{U}^\pm_m(t)V^p$ with norm
\[ \| u \|_{V^p_{\pm, m}} = \| \mc{U}^\pm_m(-t) u \|_{V^p}.\]

For $m \g 0$ and $2\les q \les \infty$, and any
          $d \in 2^\ZZ$ we have
        \[\| C_d^{\pm, m}  u \|_{L^q_t L^2_x} \lesa d^{-\frac{1}{q}} \| u \|_{V^2_{\pm,m}},\]
see \cite[Corollary 2.18]{Hadac2009}.
 We also require an additional auxiliary norm, which is used to obtain a high-low frequency gain in a particular case of the bilinear estimates appearing in sections \ref{sec:multi-sub} and \ref{sec:multi-crit}. Given $1<a<2$ and $b>0$ we define
	$$ \| u \|_{Y^{\pm, m}_{\lambda}} = \sup_{d \in 2^\ZZ} d^{\frac{1}{a}} \Big( \frac{\min\{d, \lambda\}}{\lambda} \Big)^{b}  \| C^{\pm, m}_d P_\lambda u \|_{L^a_t L^2_x}. $$
The parameters $a$ and $b$ are fixed later in sections \ref{sec:multi-sub} and \ref{sec:multi-crit},  but roughly we  take $\frac{1}{a}  - \frac{1}{2} \approx \sigma + s_0$, and $b \approx \frac{1}{a} - \frac{1}{2}$ where $s_0$ and $\sigma$ are as in \eqref{eqn:non resonant regime} or \eqref{eqn:resonant regime}. For  $s \in \RR$ and $\sigma > 0$, we  define our main function norms as
     \begin{align*} \| u \|_{\mb{V}^{s, \sigma}_{\pm, m}} &= \bigg( \sum_{\lambda, N \in 2^\NN} \lambda^{2s} N^{2\sigma} \| P_\lambda H_N u \|_{V^2_{\pm, m}}^2 \bigg)^\frac{1}{2}  \\
     \| u \|_{\mb{Y}^{s, \sigma}_{\pm, m}} &=\bigg( \sum_{\lambda, N \in 2^\NN} \lambda^{2s} N^{2\sigma} \|  H_N u \|_{Y^{\pm, m}_\lambda}^2 \bigg)^\frac{1}{2}
     \end{align*}
while if  $\sigma=0$ we take
   \begin{align*} \| u \|_{\mb{V}^{s, 0}_{\pm, m}} &= \bigg( \sum_{\lambda \in 2^\NN} \lambda^{2s}  \| P_\lambda u \|_{V^2_{\pm, m}}^2 \bigg)^\frac{1}{2}  \\
     \| u \|_{\mb{Y}^{s, 0}_{\pm, m}} &=\bigg( \sum_{\lambda \in 2^\NN} \lambda^{2s} \|  u \|_{Y^{\pm, m}_\lambda}^2 \bigg)^\frac{1}{2}.
     \end{align*}
Note that strictly speaking $\| \cdot \|_{\mb{Y}^{s, \sigma}_{\pm, m}}$ is not a norm, as $ \| \mc{U}^{\pm}_m(t) f \|_{\mb{Y}^{s, \sigma}_{\pm, m}} = 0$ for all $f \in H^s_{\sigma}$. We finally define
    $$ \| u \|_{\mb{F}^{s, \sigma}_{\pm, m} } = \| u \|_{\mb{V}^{s, \sigma}_{\pm, m}} + \| u \|_{\mb{Y}^{s,\sigma}_{\pm, m}}.$$
The wave component of the DKG equation is estimated in $\mb{V}^{s, \sigma}_{\pm, m}$. On the other hand, to control solutions to the Dirac equation, we define
    $$ \| \psi \|_{\mb{V}^{s,\sigma}_M} = \| \Pi_{+} \psi \|_{\mb{V}^{s,\sigma}_{+, M}} + \| \Pi_- \psi \|_{\mb{V}^{s, \sigma}_{-, M}}, \qquad \| \psi \|_{\mb{Y}^{s,\sigma}_M} = \| \Pi_{+} \psi \|_{\mb{Y}^{s,\sigma}_{+, M}} + \| \Pi_- \psi \|_{\mb{Y}^{s, \sigma}_{-, M}}$$
and
    $$ \| \psi \|_{\mb{F}^{s, \sigma}_M} = \| \psi \|_{\mb{V}^{s, \sigma}_M}  + \| \psi \|_{\mb{V}^{s, \sigma}_M}$$
The Banach space $\mb{V}^{s, \sigma}_{\pm, m}$ is then defined as the collection of all right continuous functions $u \in L^\infty_t H^s_\sigma$ such that $\| u \|_{\mb{V}^{s, \sigma}_{\pm, m}}<\infty$. The Banach spaces $\mb{V}^{s, \sigma}_{M}$ and $\mb{F}^{s, \sigma}_{M}$ are defined similarly. \\

We next localise the spaces constructed above to intervals. Let $I \subset \RR$ be a left-closed interval and take $\ind_I(t)$ to be the corresponding indicator function. Given a function $u$ on $I$, we make a harmless abuse of notation and think of $\ind_I(t) u(t)$ as a function defined on $\RR$. In other words $\ind_I u$ is the extension of $u$ by zero to a function on $\RR$. We then define $\mb{V}^{s,\sigma}_{\pm, m}(I)$ as the set of all right-continuous functions $u \in L^\infty_t H^s_\sigma(I\times \RR^3)$ such that $\ind_I u \in \mb{V}^{s, \sigma}_{\pm, m}$ with the obvious norm
    $$ \| u \|_{\mb{F}^{s,\sigma}_{\pm, m}(I)} = \| \ind_I u \|_{\mb{V}^{s, \sigma}_{\pm, m}}.$$
The Banach space $\mb{F}^{s, \sigma}_M(I)$, and the norms $\| \cdot \|_{\mb{F}^{s,\sigma}_{M}(I)}$ and $\| \cdot \|_{\mb{Y}^{s,\sigma}_{M}(I)}$ are defined analogously. Note that the existence of left-sided limits in $V^2$ immediately implies that if $u \in \mb{V}^{s, \sigma}_{\pm, m}(I)$ then there exists $f\in H^s_\sigma$ such that
      $\| u(t) - \mc{U}^\pm_m(t) f \|_{H^s_\sigma} \rightarrow 0$
as $t \rightarrow \sup I$. In particular, if $I=[t_0, t_1)$, $t_1<\infty$, and $\| u \|_{\mb{V}^{s, \sigma}_{\pm, m}(I)}<\infty$ then $u(t_1) \in H^s_{\sigma}$ is well-defined.

The following lemma shows that we may freely restrict the spaces to smaller intervals.
\begin{lemma}\label{lem:intervals disposable}
Let $M>0$, $s, \sigma \g 0$, and $I$ and $I'$ be left closed intervals with $I \subset I'$. If $\psi \in \mb{F}^{s, \sigma}_M(I')$ and $\phi \in \mb{V}^{s, \sigma}_{+, M}(I')$, then
        $$ \| \psi \|_{\mb{F}^{s, \sigma}_M(I)} \lesa \| \psi \|_{\mb{F}^{s, \sigma}_M(I')}, \qquad \qquad \| \phi \|_{\mb{V}^{s, \sigma}_{+, M}(I)} \lesa \| \phi \|_{\mb{V}^{s, \sigma}_{+, M}(I')}.$$
\end{lemma}
\begin{proof}
  By taking differences, exploiting translation invariance, and unpacking the definitions of the spaces $\mb{F}^{s, \sigma}_M(I)$ and $\mb{V}^{s, \sigma}_{+, M}(I)$, it suffices to show that
        \begin{equation}\label{eqn:proof of lem intervals disp:V2 bound}
            \| \ind_I u \|_{V^2} \les 2 \| u \|_{V^2}
        \end{equation}
  and for any $\lambda \in 2^\NN$
        \begin{equation}\label{eqn:proof of lem intervals disp:Y bound}
            \| \ind_I u \|_{Y^{\pm, M}_\lambda} \lesa \| P_\lambda u \|_{L^\infty_t L^2_x} + \| u \|_{Y^{\pm, M}_\lambda}
        \end{equation}
  with $I=[0, \infty)$. The first inequality follows by noting that for any increasing sequence $(t_j)_{ j \in \ZZ}$ we have
    $$ \sum_{ j \in \ZZ} \| (\ind_I u)(t_j) - (\ind_I u)(t_{j-1})\|_{L^2_x(\RR^3)}^2 \les \| u \|_{L^\infty_t L^2_x}^2 + | u |_{V^2_{\pm, M}}^2= \| u\|_{V^2_{\pm, M}}^2.$$
  On the other hand, to prove \eqref{eqn:proof of lem intervals disp:Y bound}, we start by defining the time frequency localisation operators $P^{(t)}_{d} = \rho(\frac{|-i\p_t|}{d})$ and $P^{(t)}_{\les d} = \rho_{\les 1}(\frac{|-i\p_t|}{d})$, where $\rho$ is as in the definition of the $C_d^{M, \pm}$ and $P_\lambda$ multipliers. Suppose for the moment that we have the bounds
        \begin{equation}\label{eqn:proof of lem intervals disp:ind bounds}
            \| \ind_I - P^{(t)}_{\ll d} \ind_I \|_{L^a_t}\lesa d^{-\frac{1}{a}}, \qquad \qquad \| P^{(t)}_{\ll d} \ind_I \|_{L^\infty_t} \lesa 1.
        \end{equation}
  Let $F(t) = \mc{U}^M_\pm(-t) u$. The identity $C^{\pm, M}_d = \mc{U}^{M}_\pm(t) P^{(t)}_d \mc{U}^M_{\pm}(-t)$, together with
  \eqref{eqn:proof of lem intervals disp:ind bounds} and  the fact that the free solution propagators are unitary on $L^2_x(\RR^3)$, implies that
     \begin{align*} \| C^{\pm, M}_d ( \ind_I u) \|_{L^a_t L^2_x} &= \| P^{(t)}_{d}( \ind_I F) \|_{L^a_t L^2_x} \\
            &\les \| P^{(t)}_{d}( [\ind_I - P^{(t)}_{\ll d} \ind_I] F) \|_{L^a_t L^2_x} + \| P^{(t)}_{d}( P_{\ll d}^{(t)} \ind_I F) \|_{L^a_t L^2_x}.\\
            &\lesa \| [\ind_I - P^{(t)}_{\ll d} \ind_I] F \|_{L^a_t L^2_x} + \sup_{d'\approx d} \|  P_{\ll d}^{(t)} \ind_I P^{(t)}_{d'} F \|_{L^a_t L^2_x}\\
            &\lesa d^{-\frac{1}{a}} \| u \|_{L^\infty_t L^2_x} +  \sup_{d'\approx d} \| C^{\pm, M}_{d'} u \|_{L^a_t L^2_x}.
     \end{align*}
  Thus, by definition of the $Y^{\pm, M}_{\lambda}$ norm, the required bound \eqref{eqn:proof of lem intervals disp:Y bound} follows. Consequently it only remains to prove the bounds \eqref{eqn:proof of lem intervals disp:ind bounds}. To this end, the definition of the $P^{(t)}_{\ll d}$ multipliers implies that there exists a rapidly decreasing function $\sigma \in C^\infty(\RR)$ with $\int_\RR \sigma(t) dt  = 1$ such that
        $$ P^{(t)}_{\ll d} \ind_I(t) = \int_0^\infty \sigma( t d -s) ds .$$
  Hence the second estimate in \eqref{eqn:proof of lem intervals disp:ind bounds} is immediate. On the other hand, for the first term, we observe that the rapid decay of $\sigma$ gives
    \begin{align*}
      \Big| \ind_I(t)  - P^{(t)}_{\ll d} \ind_I (t) \Big| &= \Big| \ind_I(t) \int_\RR \sigma( t d - s) ds   - \int_0^\infty\sigma(td - s) \Big|\\
                                    &\les \Big| \int_\RR \sigma( |t| d + |s|) ds \Big| + \Big| \int_\RR \sigma( - |t| d -|s|) ds \Big| \\
                                    &\lesa \lr{td}^{-10}
    \end{align*}
  and hence \eqref{eqn:proof of lem intervals disp:ind bounds} follows.
\end{proof}

A straight forward computation implies that linear solutions belong to the spaces $\mb{F}^{s, \sigma}_M$ and $\mb{V}^{s, \sigma}_{+, m}$. Lemma \ref{lem:intervals disposable} implies that for any $s, \sigma \g 0$ and $t_0 \in I \subset \RR$ we have
    \begin{equation}\label{eqn:free solns bounded}
        \| \mc{U}_M(t-t_0) \psi_0 \|_{\mb{F}^{s, \sigma}_{M}(I)} \lesa \| \psi_0 \|_{H^s_\sigma}, \qquad \| \mc{U}^+_m(t-t_0) \phi_0 \|_{\mb{V}^{s, \sigma}_{+, m}(I)} \lesa \|\phi_0\|_{H^s_\sigma}.
    \end{equation}
The Strichartz type spaces are also controlled by the $\mb{F}^{s, \sigma}_M$ and $\mb{V}^{s, \sigma}_{+, m}$ norms. We give a more precise version of this statement in Section \ref{sec:prelim}, and for the moment we simply recall that we have the bounds
    \begin{equation}\label{eqn:norms control L4}
        \begin{split}
            \| \psi \|_{L^\infty_t H^s_\sigma(I\times \RR^3)} + \| \psi \|_{\mb{D}^{s-\frac{1}{2}}_\sigma(I)} &\lesa \|\psi \|_{\mb{V}^{s, \sigma}_M(I)} \\
            \| \phi \|_{L^\infty_t H^{s+\frac{1}{2}}_\sigma(I\times \RR^3)} + \| \phi \|_{\mb{D}^s_\sigma(I)} &\lesa \| \phi \|_{\mb{V}^{s+\frac{1}{2}, \sigma}_{+, m}(I)}
        \end{split}
    \end{equation}
These estimates follow directly from the fact that the estimates for the free solutions immediately imply bounds in the corresponding $V^2$ space, the details can be found in, for instance, \cite{Hadac2009}. We also need to understand how the norm $\| \cdot \|_{\mb{F}^{s, \sigma}_M(I)}$ depends on the interval $I$. Clearly, if $I$ is left-closed, and we can write $I=I_1\cup I_2$ with $I_1$, $I_2$ disjoint left-closed intervals, then by the triangle inequality we have the bound
    \begin{equation}\label{eqn:sum intervals bound V}
        \| u \|_{\mb{V}^{s, \sigma}_{\pm, m}(I)} = \| \ind_I(t) u \|_{\mb{V}^{s, \sigma}_{\pm, m}} \les \|\ind_{I_1}(t) u \|_{\mb{V}^{s,\sigma}_{\pm, m}} +  \|\ind_{I_2}(t) u \|_{\mb{V}^{s,\sigma}_{\pm, m}} =  \| u \|_{\mb{V}^{s, \sigma}_{\pm, m}(I_1)}+ \| u \|_{\mb{V}^{s, \sigma}_{\pm, m}(I_2)}.
    \end{equation}
An identical argument gives
    \begin{equation}\label{eqn:sum intervals bound F}
        \| u \|_{\mb{F}^{s, \sigma}_{M}(I)} \les   \| u \|_{\mb{F}^{s, \sigma}_{M}(I_1)}+ \| u \|_{\mb{F}^{s, \sigma}_{M}(I_2)}.
    \end{equation}

\section{Local theory for the Dirac-Klein-Gordon system}\label{sec:dkg}
  In this section, we derive two key consequences of the bilinear estimates obtained in Section \ref{sec:multi-sub} and Section \ref{sec:multi-crit}, namely Theorem \ref{thm:lwp} and Theorem \ref{thm:main-det}. These theorems show that the time of existence of solutions to the DKG system can be controlled by the $L^4_{t,x}$ norm, or, more precisely, by the $\| \cdot \|_{\mb{D}^0_\sigma}$ norms. In particular,  we refine the previous local and small data global results obtained in \cite{Candy2016,Bejenaru2015}. The proof of Theorem \ref{thm:cond-dkg} is then an easy consequence.

We now start with the local theory. By time reversibility, it is enough to consider the forward in time problem, thus we always work on left-closed intervals $I=[t_0, t_1)$ where potentially we may have $t_1 = \infty$. In addition, instead to the second order system \eqref{eqn:DKG}, we prefer to work with a first order system. More precisely, we construct $\psi: I\times \RR^3 \rightarrow \CC^4$ and $\phi_+: I\times \RR^3\rightarrow \CC$ solving the first order system
    \begin{equation}\label{eqn:dkg reduced}
        \begin{split}
          -i \gamma^\mu \p_\mu \psi + M\psi &= \Re(\phi_+) \psi \\
           - i \p_t \phi_+ + \lr{\nabla}_m \phi_+ &= \lr{\nabla}_m^{-1}\big(\overline{\psi}\psi\big).
        \end{split}
    \end{equation}
If we let $\phi = \Re(\phi_+)$, a standard computation using the fact that $\overline{\psi}\psi \in \RR$ then shows that, at least for classical solutions, $(\psi, \phi)$ is a solution to the original Dirac-Klein-Gordon system (\ref{eqn:DKG}).

We next clarify precisely what we mean by solutions, and maximal solutions.
\begin{definition}\label{def:sol}
Let $s,\sigma\in \RR$.
\begin{enumerate}
\item We say $(\psi,\phi_+):I\times \RR^3\to \CC^4\times \RR $ is a $H^s_\sigma$-strong solution on an interval $I\subset \RR$, if
\[(\psi,\phi_+)\in C\big(I,H^s_\sigma(\RR^3,\CC^4)\times H^{s+\frac{1}{2}}_\sigma(\RR^3,\CC)\big)\]
and there exists a sequence $(\psi_n,\phi_n)\in C^{2}(I,H^{m}(\RR^3,\CC^4\times \CC ))$, $m=\max\{10,s\}$, of classical solutions to \eqref{eqn:dkg reduced} such that, for any compact $I'\subseteq I$,
\[
\sup_{t \in I'}\Big(\|\psi(t)-\psi_n(t)\|_{H^s_\sigma(\RR^3,\CC^4)}+\|\phi_+(t)-\phi_n(t)\|_{H^{\frac12+s}_\sigma(\RR^3,\CC)}\Big)\to 0.
\]
\item We say $(\psi,\phi_+):[t_0,t^\ast)\times \RR^3\to \CC^4\times \CC$ is a (forward) maximal $H^s_\sigma$-solution if the following two properties hold:
\begin{enumerate}
\item For any $t_1\in (t_0,t^\ast)$,  $(\psi, \phi_+)$ is a strong $H^s_\sigma$-solution on $[t_0,t_1)$,
\item If $(\psi',\phi'_+):I\times \RR^3\to \CC^4\times \CC$ is a strong $H^s_\sigma$-solution on an interval $I$ satisfying $I\cap [t_0,t^\ast)\not=\varnothing$ and $( \psi',\phi'_+)=(\psi,\phi_+)$ on $I\cap [t_0,t^\ast)$, then $I\cap[t_0,\infty)\subseteq [t_0,t^\ast)$.\\
\end{enumerate}
\end{enumerate}
\end{definition}

We remark that $H^s_\sigma$-strong solutions are unique, owing to the fact classical solutions are unique. Moreover, the local theory we develop below implies that an $H^{s_0}_\sigma$-strong solution to \eqref{eqn:dkg reduced} on an interval $I$, is also \emph{locally} in $\mb{F}^{s_0, \sigma}_M \times \mb{V}^{s_0+\frac{1}{2}, \sigma}_{+, m}$. Furthermore, we also show that if we have control over both the $L^4_{t,x}$ type norm $\| \psi \|_{\mb{D}^0_{\sigma}(I)}$ \emph{and} the data norm $\| (\psi, \phi_+)(t_0) \|_{H^{s_0}_\sigma\times H^{s_0+\frac{1}{2}}_\sigma}$, then in fact $(\psi, \phi_+) \in \mb{F}^{s_0, \sigma}_M(I) \times \mb{V}^{s_0 + \frac{1}{2}, \sigma}_{+, m}(I)$. In view of the existence of right and left limits of elements of $V^2$, this second fact is, roughly speaking, simply a restatement of Theorem \ref{thm:cond-dkg}. We give a more detailed description of this argument in Section \ref{sec:proof of thm cond-dkg}.\\

We state the bilinear estimates that we exploit in the following.
  \begin{theorem}\label{thm:main bilinear}
Let $s_0, \sigma \g 0$ and $M, m>0$ satisfy either (\ref{eqn:resonant regime}) or (\ref{eqn:non resonant regime}). There exists $1<a<2$, $b>0$, and $C>0$ such that if $I\subset \RR$ is a left-closed interval,  $t_0 \in I$, and $\phi_+ \in \mb{V}^{s_0+\frac{1}{2},\sigma}_{m, +}(I)$, $\psi, \varphi \in \mb{F}^{s_0,\sigma}_M(I)$, then we have the bounds
    $$ \big\| \mc{I}^M_{t_0}\big( \Re(\phi_+) \gamma^0 \psi \big) \big\|_{\mb{V}^{s_0, \sigma}_{M}(I)} \les C \Big( \| \phi_+ \|_{\mb{D}^0_\sigma(I)} \|\psi\|_{\mb{D}^{-\frac{1}{2}}_\sigma(I)} \Big)^\theta \Big( \| \phi_+ \|_{\mb{V}^{s_0+\frac{1}{2}, \sigma}_{+, m}(I)} \| \psi \|_{\mb{F}^{s_0, \sigma}_{M}(I)}\Big)^{1-\theta}$$
and
     $$ \big\| \mc{I}^M_{t_0}\big( \Re(\phi_+) \gamma^0 \psi \big) \big\|_{\mb{Y}^{s_0, \sigma}_{M}(I)} \les C\Big( \| \phi_+ \|_{\mb{D}^0_\sigma(I)} \|\psi\|_{\mb{D}^{-\frac{1}{2}}_\sigma(I)} \Big)^\theta \Big( \| \phi_+ \|_{\mb{V}^{s_0+\frac{1}{2}, \sigma}_{+, m}(I)} \| \psi \|_{\mb{V}^{s_0, \sigma}_{M}(I)}\Big)^{1-\theta}$$
and
     $$ \big\| \mc{I}^{m, +}_{t_0}\big( \lr{\nabla}_m^{-1} ( \overline{\psi} \varphi) \big) \big\|_{\mb{V}^{s_0+\frac{1}{2}, \sigma}_{+, m}(I)} \les C \Big( \| \psi \|_{\mb{D}^{-\frac{1}{2}}_\sigma(I)} \|\varphi\|_{\mb{D}^{-\frac{1}{2}}_\sigma(I)} \Big)^\theta \Big( \| \psi \|_{\mb{F}^{s_0, \sigma}_{M}(I)} \| \varphi \|_{\mb{F}^{s_0, \sigma}_{M}(I)}\Big)^{1-\theta}.$$
Moreover, for any $s\g s_0$,  we have the fractional Leibniz type bounds
    $$ \big\| \mc{I}^M_{t_0}\big( \Re(\phi_+) \gamma^0\psi \big) \big\|_{\mb{F}^{s, \sigma}_{M}(I)} \les 2^s C \Big( \| \phi_+ \|_{\mb{D}^0_\sigma(I)}^\theta \| \phi_+ \|_{\mb{V}^{s_0+\frac{1}{2}, \sigma}_{+, m}(I)}^{1-\theta}  \| \psi \|_{\mb{F}^{s, \sigma}_{M}(I)} + \| \phi_+ \|_{\mb{V}^{s+\frac{1}{2}, \sigma}_{+, m}(I)} \|\psi\|_{\mb{D}^{-\frac{1}{2}}_\sigma(I)}^\theta  \| \psi \|_{\mb{F}^{s_0, \sigma}_{M}(I)}^{1-\theta}\Big)$$
and
    $$ \big\| \mc{I}^{m, +}_{t_0}\big( \lr{\nabla}_m^{-1} ( \overline{\psi} \varphi) \big) \big\|_{\mb{V}^{s+\frac{1}{2}, \sigma}_{+, m}(I)} \les 2^s C  \Big(\| \psi \|_{\mb{D}^{-\frac{1}{2}}_\sigma(I)}^\theta  \| \psi \|_{\mb{F}^{s_0, \sigma}_{M}(I)}^{1-\theta} \| \varphi \|_{\mb{F}^{s, \sigma}_{M}(I)} +\| \psi \|_{\mb{F}^{s, \sigma}_{M}(I)} \|\varphi\|_{\mb{D}^{-\frac{1}{2}}_\sigma(I)}^\theta \| \varphi \|_{\mb{F}^{s_0, \sigma}_{M}(I)}^{1-\theta}\Big).$$
\end{theorem}
\begin{proof}
First, suppose that $I=\RR$. By rescaling, we may assume $m=1$. The bounds are immediate consequences of Theorem \ref{thm:duhamel-sub} (in the case \eqref{eqn:non resonant regime}) and Theorem \ref{thm:duhamel-crit} (in the case \eqref{eqn:resonant regime}), because of $2\Re(\phi_+)=\phi_++\phi_+^\ast$, as well as the fact that for any $s'$ and $m_1\g 0$ we have $\| \cdot \|_{\mb{D}^{s'}_\sigma}\lesa \| \cdot \|_{\mb{V}^{s'+\frac{1}{2}, \sigma}_{\pm, m_1}}$.

Finally, we remark that Lemma \ref{lem:intervals disposable} implies the claims for arbitrary left-closed intervals $I$ with $t_0\in I$, since we have
\[
\ind_I\mc{I}^M_{t_0}\big( \Re(\phi_+) \gamma^0 \psi \big)=\ind_I\mc{I}^M_{t_0}\big( \Re(\ind_I \phi_+) \gamma^0 \ind_I\psi \big),
\]
and similarly for $\mc{I}^{m, +}_{t_0}\big( \lr{\nabla}_m^{-1} ( \overline{\psi} \varphi) \big)$.
\end{proof}

Note that we have elected to separate the statement of the bounds in the $\mb{F}^{s_0, \sigma}_{M}$ for spinor components into a $\mb{V}^{s_0, \sigma}_{M}$ bound, and a $\mb{Y}^{s_0, \sigma}_M$ bound. The motivation is that although the bound in the $\mb{V}^{s_0, \sigma}_{M}$ requires $\psi \in \mb{F}^{s_0, \sigma}_{M}$, the control the of the $\mb{Y}^{s_0, \sigma}_M$ norm only requires $\psi \in \mb{Y}^{s_0, \sigma}_M$. Consequently, if required, by using a two step iteration argument, it is possible to develop a local theory using just the $V^2$ type norms, without explicitly using the $\mb{Y}^{s_0, \sigma}_M$ spaces. Another way of stating this, is that for \emph{solutions} to the DKG system, if $\psi \in \mb{V}^{s_0, \sigma}_M$, then we immediately have $\psi \in \mb{F}^{s_0, \sigma}_{M}$. \\

We now give a precise version of the local well-posedness (and small data global well-posedness) theory that follows from the bilinear estimates in Theorem \ref{thm:main bilinear}.

\begin{theorem}[Local well-posedness]\label{thm:lwp}
Let $s_0, \sigma \g 0$ and $M, m>0$ satisfy either (\ref{eqn:resonant regime}) or (\ref{eqn:non resonant regime}). There exist $\theta \in (0,1)$ and $C>1$, such that if
    $$ A, B >0, \qquad 0<\alpha \les (C A^{1-\theta})^{-\frac{1}{\theta}}, \qquad 0<\beta \les (C B^{1-\theta})^{-\frac{1}{\theta}},$$
and $I\subset \RR$ is a left closed interval, then for any initial time $t_0\in I$, and any data $(\psi_0, \phi_0) \in H^{s_0}_\sigma \times H^{s_0+\frac{1}{2}}_\sigma$ satisfying
$$ \|\psi_0\|_{H^{s_0}_\sigma}< A, \qquad \|\mathcal{U}_{M}(t-t_0)\psi_0\|_{\mb{D}^{-\frac{1}{2}}_\sigma(I)}< \alpha, $$
and
$$\|\phi_0\|_{H^{s_0+\frac12}_\sigma }< B, \qquad \|\mc{U}_{m}^+(t-t_0) \phi_0\|_{\mb{D}^0_\sigma(I)}< \beta,$$
there exists a unique $H^{s_0}_\sigma$-strong solution $(\psi,\phi_+)$ of \eqref{eqn:dkg reduced} on $I$ with $(\psi, \phi_+)(t_0)=(\psi_0, \phi_0)$. Moreover,  the data-to-solution map is Lipschitz-continuous into $\mb{F}^{s_0, \sigma}_M(I) \times \mb{V}^{s_0, \sigma}_{+, m}(I)$ and we have the bounds
\begin{align*}
\|\psi - \mathcal{U}_M(t-t_0)\psi_0\|_{\mathbf{F}^{s_0, \sigma}_M(I)} &\les C\alpha^\theta A^{1-\theta}(\alpha^\theta A^{1-\theta}+\beta^\theta B^{1-\theta}) \\
\|\phi_+ - \mathcal{U}^+_m(t-t_0)\phi_0\|_{\mathbf{V}^{s_0+\frac{1}{2}, \sigma}_{+, m}} &\les C\big(\alpha^\theta A^{1-\theta}\big)^2.
\end{align*}
Finally, if we have additional regularity $(\psi_0,\phi_0)\in H^{s}_\sigma \times H^{s+\frac12}_\sigma$ for some $s>s_0$, then $(\psi,\phi_+) \in F^{s, \sigma}_M(I) \times \mb{V}^{s, \sigma}_{+, m}(I)$ is also a $H^s_\sigma$-strong solution.
\end{theorem}

\begin{proof}
Let $\epsilon_0 = \alpha^\theta A^{1-\theta} + \beta^\theta B^{1-\theta}$ and $\epsilon_s = \| \psi_0 \|_{H^s_\sigma} + \| \phi_0 \|_{H^{s+\frac{1}{2}}_\sigma}$. For ease of notation, we take
    $$ \psi_L = \mc{U}_M(t-t_0) \psi_0, \qquad \phi_L = \mc{U}^+_m(t-t_0) \phi_0, \qquad \psi_N=\psi- \psi_L, \qquad \phi_N = \phi - \phi_L.$$
Define the set $\mc{X}_s \subset \mb{F}^{s, \sigma}_M(I) \times \mb{V}^{s, \sigma}_{+, m}(I)$ as the  collection of all $(\psi, \phi) \in \mb{F}^{s, \sigma}_M(I) \times \mb{V}^{s, \sigma}_{+, m}(I)$ such that
    $$ \| \psi_N \|_{\mb{F}^{s_0, \sigma}_{M}(I) } + \| \phi_N \|_{\mb{V}^{s_0, \sigma}_{+, m}(I)} \les \epsilon_0, \qquad \| \psi_N \|_{\mb{F}^{s_0, \sigma}_M(I)} + \| \phi_N \|_{\mb{V}^{s, \sigma}_M(I)} \les \epsilon_s. $$
Our goal is construct a fixed point of the map $\mc{T}=(\mc{T}_1, \mc{T}_2):\mc{X}_s \to \mc{X}_s$ defined as
    $$ \mc{T}_1(\psi, \phi) = \psi_L + \mc{I}^M_{t_0}[\Re(\phi) \gamma^0 \psi], \qquad \mc{T}_2(\psi, \phi) = \phi_L + \mc{I}^{+, m}_{t_0}[\lr{\nabla}_m^{-1}(\overline{\psi}\psi)]. $$
To this end, if $(\psi, \phi) \in \mc{X}_s$, then after decomposing the product $\phi \psi = \phi_L \psi_L + \phi_N \psi_L + \phi_L \psi_N + \psi_N \phi_N$, an application of Theorem \ref{thm:main bilinear} together with the bounds \eqref{eqn:free solns bounded} and \eqref{eqn:norms control L4} implies there exists $C^*>0$ such that
    \begin{equation}\label{eqn:proof of thm lwp:psi bound}
        \begin{split}
            \| \mc{I}^M_{t_0}(\Re(\phi) \gamma^0 \psi) \|_{\mb{F}^{s_0, \sigma}_M(I)} &\les C^* \Big( \beta^\theta B^{1-\theta} + \epsilon_0 \Big) \epsilon_0\\
            \| \mc{I}^M_{t_0}(\Re(\phi) \gamma^0  \psi) \|_{\mb{F}^{s, \sigma}_M(I)} &\les  C^* \Big( \beta^\theta B^{1-\theta} + \epsilon_0\Big) \epsilon_s.
        \end{split}
    \end{equation}
Similarly, decomposing $\overline{\psi} \psi =\overline{\psi}_L \psi_L+\overline{\psi}_N \psi_L+\overline{\psi}_L \psi_N+\overline{\psi}_N \psi_N$, we have $C^*>0$ such that
    \begin{equation}\label{eqn:proof of thm lwp:phi bound}
        \begin{split}
            \| \mc{I}^{+, m}_{t_0}(\lr{\nabla}_m^{-1} \overline{\psi} \psi) \|_{\mb{V}^{s_0, \sigma}_{+,m}(I)} &\les C^*  \epsilon_0^2\\
            \| \mc{I}^{+, m}_{t_0}(\lr{\nabla}_m^{-1} \overline{\psi} \psi) \|_{\mb{V}^{s, \sigma}_{+,m}(I)}  &\les  C^*  \epsilon_0 \epsilon_s.
        \end{split}
    \end{equation}
To show that $\mc{T}$ is a contraction on $\mc{X}_s$, we observe that for $(\psi, \phi), (\psi', \phi') \in \mc{X}$ after decomposing the difference of the products as
    $$  \phi\psi - \phi' \psi' = (\phi - \phi') (\psi_L + \psi_N) + (\phi_L + \phi'_N) (\psi' - \psi) $$
another application of Theorem \ref{thm:main bilinear},  \eqref{eqn:free solns bounded},  and \eqref{eqn:norms control L4} implies that there exists $C^*=C^*(s)>0$ such that
    \begin{equation}\label{eqn:proof of thm lwp:psi diff bound low}
        \begin{split}
             \| \mc{I}^M_{t_0}(\Re(\phi) \gamma^0 \psi) - \mc{I}^M_{t_0}(\Re(\phi')& \gamma^0 \psi') \|_{\mb{F}^{s_0, \sigma}_M(I)}\\
                &\les C^* \Big( \beta^\theta B^{1-\theta} + \epsilon_0 \Big) \Big( \| \psi - \psi'\|_{\mb{F}^{s_0, \sigma}_{M}(I)} + \| \phi - \phi' \|_{\mb{V}^{s_0, \sigma}_{+, m}(I)}\Big)
        \end{split}
    \end{equation}
and
    \begin{equation}\label{eqn:proof of thm lwp:psi diff bound high}
        \begin{split}
            \| \mc{I}^M_{t_0}(\Re(\phi)& \gamma^0 \psi) - \mc{I}^M_{t_0}(\Re(\phi') \gamma^0 \psi') \|_{\mb{F}^{s, \sigma}_M(I)} \\
            &\les C^* \Big( \beta^\theta B^{1-\theta} + \epsilon_0 \Big) \Big( \| \psi - \psi'\|_{\mb{F}^{s, \sigma}_{M}(I)} + \| \phi - \phi' \|_{\mb{V}^{s, \sigma}_{+, m}(I)}\Big) \\
            &\qquad \qquad \qquad\qquad \qquad + C^*\epsilon_s \Big( \| \psi - \psi'\|_{\mb{F}^{s_0, \sigma}_{M}(I)} + \| \phi - \phi' \|_{\mb{V}^{s_0, \sigma}_{+, m}(I)}\Big).
        \end{split}
    \end{equation}
Similarly, we have $C^*=C^*(s)>0$ such that
     \begin{equation}\label{eqn:proof of thm lwp:phi diff bound low}
        \begin{split}
             \| \mc{I}^{+, m}_{t_0}(\lr{\nabla}_m^{-1} \overline{\psi} \psi) - \mc{I}^{+, m}_{t_0}(\lr{\nabla}_m^{-1} \overline{\psi'} \psi')\|_{\mb{V}^{s_0, \sigma}_{+, m}(I)} \les C^* \epsilon_0 \Big( \| \psi - \psi'\|_{\mb{F}^{s_0, \sigma}_{M}(I)} + \| \phi - \phi' \|_{\mb{V}^{s_0, \sigma}_{+, m}(I)}\Big)
        \end{split}
    \end{equation}
and
    \begin{equation}\label{eqn:proof of thm lwp:phi diff bound high}
        \begin{split}
            \|\mc{I}^{+, m}_{t_0}(\lr{\nabla}_m^{-1} \overline{\psi} \psi)& - \mc{I}^{+, m}_{t_0}(\lr{\nabla}_m^{-1} \overline{\psi'} \psi') \|_{\mb{V}^{s, \sigma}_{+, m}(I)} \\
            &\les C^* \epsilon_0 \Big( \| \psi - \psi'\|_{\mb{F}^{s, \sigma}_{M}(I)} + \| \phi - \phi' \|_{\mb{V}^{s, \sigma}_{+, m}(I)}\Big) \\
            &\qquad \qquad \qquad\qquad \qquad + C^*\epsilon_s \Big( \| \psi - \psi'\|_{\mb{F}^{s_0, \sigma}_{M}(I)} + \| \phi - \phi' \|_{\mb{V}^{s_0, \sigma}_{+, m}(I)}\Big).
        \end{split}
    \end{equation}
Consequently, taking $C^*=C^*(s)>0$ to be the largest of the constants appearing in (\ref{eqn:proof of thm lwp:psi bound})-(\ref{eqn:proof of thm lwp:phi diff bound high}), we see that provided
    $$ \beta^\theta B^{1-\theta} + \alpha^\theta A^{1-\theta} \les (2C^*)^{-1}$$
we have $\mc{T} :\mc{X}_s\to \mc{X}_s$ and the difference bounds
     \begin{align*}
       \| \mc{T}_1(\psi, \phi) - \mc{T}(\psi', \phi') \|_{\mb{F}^{s_0, \sigma}_{M}(I)} + \| \mc{T}_2(\psi, \phi)& - \mc{T}_2(\psi', \phi')\|_{\mb{V}^{s_0, \sigma}_{+, m}(I)} \\
       &\les \frac{1}{2} \Big(\| \psi - \psi'\|_{\mb{F}^{s_0, \sigma}_{M}(I)} + \| \phi - \phi' \|_{\mb{V}^{s_0, \sigma}_{+, m}(I)} \Big)
     \end{align*}
and
    \begin{align*}
       \| \mc{T}_1(\psi, \phi)& - \mc{T}(\psi', \phi') \|_{\mb{F}^{s, \sigma}_{M}(I)} + \| \mc{T}_2(\psi, \phi) - \mc{T}_2(\psi', \phi')\|_{\mb{V}^{s, \sigma}_{+, m}(I)} \\
       &\les \frac{1}{2} \Big(\| \psi - \psi'\|_{\mb{F}^{s, \sigma}_{M}(I)} + \| \phi - \phi' \|_{\mb{V}^{s, \sigma}_{+, m}(I)} \Big) + 2C^* \epsilon_s \Big(\| \psi - \psi'\|_{\mb{F}^{s_0, \sigma}_{M}(I)} + \| \phi - \phi' \|_{\mb{V}^{s_0, \sigma}_{+, m}(I)} \Big).
     \end{align*}
Therefore a standard argument implies that there exists a unique fixed point in $\mc{X}_s$, and that the resulting solution map depends continuously on the initial data.

Fix $C = \sup_{s_0\les s \les 10} 2C^*(s)$ and suppose that $0<\alpha \les ( C A^{1-\theta})^{-\frac{1}{\theta}}$ and $0<\beta \les ( C B^{1-\theta})^{-\frac{1}{\theta}}$. Then running the above argument with $s=s_0$ shows that for data $(\psi_0, \phi_0) \in H^{s_0}_\sigma H^{s_0+\frac{1}{2}}_\sigma$ we get a unique solution  $(\psi, \phi_+) \in \mc{X}_{s_0}$ which depends continuously on the data. Approximating the data with functions in $H^{10}_\sigma \times H^{10+\frac{1}{2}}_\sigma$, and applying the previous argument with $s=10$, we obtain a sequence of solutions in $\mc{X}_{10}$ which converge to $(\psi, \phi_+)$. In particular, $(\psi, \phi_+)$ is a $H^{s_0}_{\sigma}$-strong solution to \eqref{eqn:dkg reduced}. The claimed bounds on the norms of $(\psi, \phi_+)$ follow directly from the fact that $(\psi, \phi_+) \in \mc{X}$ together with \eqref{eqn:proof of thm lwp:psi bound} and \eqref{eqn:proof of thm lwp:phi bound}. Finally, if in addition we have additional regularity $(\psi_0, \phi_0) \in H^s_\sigma$ with $s>s_0$, then running the fixed point argument as above gives a solution in $\mb{F}^{s, \sigma}_{M}(I') \times \mb{V}^{s, \sigma}_{+, m}(I')$ provided that the interval $I'\subset I$ is chosen sufficiently small such that
    $$ \| \psi_L \|_{\mb{D}^{-\frac{1}{2}}_\sigma(I')}^\theta \| \psi_0 \|_{H^{s_0}_\sigma}^{1-\theta} + \| \phi_L \|_{\mb{D}^{0}_\sigma(I')}^\theta \| \phi_0 \|_{H^{s_0+\frac{1}{2}}_\sigma}^{1-\theta} \les ( 2 C^*(s) )^{-1}.$$
Since we can cover the full interval $I$ by $\mc{O}_s(1)$ smaller intervals $I'$, we deduce that $(\psi, \phi_+) \in \mb{F}^{s, \sigma}_M(I) \times \mb{V}^{s, \sigma}_M(I)$ as claimed.
\end{proof}

Note that we may take $I=[0, \infty)$ in the previous theorem. In particular, if $(\psi_0, \phi_0) \in H^{s_0}_\sigma \times H^{s_0+\frac{1}{2}}_\sigma$ then provided $\| \mc{U}_M(t) \psi_0 \|_{\mb{D}^{-\frac{1}{2}}_\sigma(I)}$ and $\| \mc{U}_m^+(t) \phi_0 \|_{\mb{D}^0_\sigma(I)}$ are sufficiently small, we have global existence and scattering. \\

The next result we give implies that $H^{s_0}_\sigma$-strong solutions belong to $\mb{F}^{s_0, \sigma}_M(I)\times \mb{V}^{s_0+\frac{1}{2}, \sigma}_{+, m}(I)$, provided that the $L^4_{t,x}$ type norm is sufficiently small relative to $\|(\psi, \phi_+)(t_0)\|_{H^{s_0}_\sigma \times H^{s_0+\frac{1}{2}}_\sigma}$.

\begin{theorem}\label{thm:main-det}
Let $s_0, \sigma \g 0$ and $M, m>0$ satisfy either (\ref{eqn:resonant regime}) or (\ref{eqn:non resonant regime}). There exists $C>1$ and $0<\theta<1$ such that, if $A>0$ and $I\subset \RR$ is a left-close interval, $t_0\in I$, and $(\psi, \phi_+)$ is a $H^{s_0}_\sigma$-strong solution on $I$ satisfying
            $$ \| \psi(t_0) \|_{H^{s_0}_\sigma} + \| \phi_+(t_0)\|_{H^{s_0+\frac{1}{2}}_\sigma} \les A, \qquad \min\Big\{ \| \psi \|_{\mb{D}^{-\frac{1}{2}}_\sigma(I)}, \| \phi_+ \|_{\mb{D}^0_\sigma(I)} \Big\}\les \big( C (1+A)^{2-\theta}A^{1-\theta}\big)^{-\frac{1}{\theta}}$$
then $(\psi, \phi_+) \in \mb{F}^{s_0, \sigma}_M (I) \times \mb{V}^{s_0 + \frac{1}{2}, \sigma}_{+, m}(I)$ and we have the bound
        $$ \| \psi \|_{\mb{F}^{s_0, \sigma}_M(I)} + \| \phi_+ \|_{\mb{V}^{s_0 + \frac{1}{2}, \sigma}_{+, m}(I)} \les C A (1+A) . $$
\end{theorem}
\begin{proof}
Let $C^*$ be the largest of the constants appearing in Lemma \ref{lem:intervals disposable}, Theorem \ref{thm:main bilinear}, Theorem \ref{thm:lwp}, and \eqref{eqn:free solns bounded}, \eqref{eqn:norms control L4}. We first consider the case $I=[t_0, t_1)$ with $t_1 \les \infty$. Let $(\psi, \phi)$ be a $H^{s_0}_\sigma$-strong solution on $I$, and define
    $$ \delta = \min\Big\{ \| \psi \|_{\mb{D}^{-\frac{1}{2}}_\sigma(I)}, \| \phi_+ \|_{\mb{D}^0_\sigma(I)} \Big\} $$
and
    $$ \mc{T} = \Big\{ t_0<T\les t_1 \, \Big| \, \sup_{t_0<T'\les T}  \| \psi \|_{\mb{F}^{s_0, \sigma}_M([t_0, T'))}\les 2C^*A, \sup_{t_0<T'\les T} \| \phi_+ \|_{\mb{V}^{s_0+\frac{1}{2}, \sigma}_{+, m}([t_0,T'))} \les (2 C^*)^3 A(1+A)\Big\}.$$
An application of the local well-posedness result in Theorem \ref{thm:lwp} implies that $T \in \mc{T}$ provided that $T-t_0$ is sufficiently small, in particular, $\mc{T}$ is nonempty. If we let $T_{sup} = \sup \mc{T}$, then our goal is to show that $T_{sup} = t_0$. Suppose that $T_{sup}<t_0$ and let $T_n \in \mc{T}$ be a sequence of times converging to $T_{sup}$. The continuity of the solution $(\psi, \phi_+)$ at $T_{sup}$, together with \eqref{eqn:norms control L4} and the definition of $\mc{T}$ implies that
    \begin{align*} \| \psi(T_{sup})\|_{H^{s_0}_\sigma} + \| \phi_+(T_{sup}) \|_{H^{s_0+\frac{1}{2}}_\sigma} &\les C^*\sup_{t_0<T< T_{sup}} \Big( \|\psi\|_{\mb{F}^{s_0, \sigma}_{M}([t_0, T))} + \| \phi_+ \|_{\mb{V}^{s_0+\frac{1}{2}, \sigma}_{+, m}([t_0, T))}\Big) \\
                    &\les (2 C^*)^4 A (1+A)
    \end{align*}
Hence, again applying Theorem \ref{thm:lwp}, there exists $n$ and $\epsilon_0>0$ such that for all $0<\epsilon<\epsilon_0$ we have on the interval $[T_n, T_{sup}+\epsilon)$ the bound
     \begin{align*} \| \psi \|_{\mb{F}^{s_0, \sigma}_M([T_n, T_{sup}+\epsilon))} + \| \phi \|_{\mb{V}^{s_0+\frac{1}{2}, \sigma}_{+, m}([T_n,T_{sup}+\epsilon))}&\les 2 C^* \Big( \| \psi(T_{sup}) \|_{H^{s_0}_{\sigma}} + \| \phi_+(T_{sup}) \|_{H^{s_0}_{\sigma}}\Big) \\
     &\les (2C^*)^5 A ( 1+A) .
     \end{align*}
We now exploit the smallness assumption on the $L^4_{t,x}$ norm. An application of \eqref{eqn:norms control L4}, \eqref{eqn:sum intervals bound F}, and \eqref{eqn:sum intervals bound V}, together with Theorem \ref{thm:main bilinear} and the fact that $(\psi, \phi_+)$ is a $H^{s_0}_{\sigma}$-strong solution on $[t_0, t_1)$ implies that
    \begin{align*}
      \| &\psi \|_{\mb{F}^{s_0, \sigma}_M([t_0, T_{sup}+\epsilon))} \\
      &\les C^* \| \psi(t_0) \|_{H^{s_0}_\sigma} + C^* \Big( \| \psi \|_{\mb{D}^{-\frac{1}{2}}_\sigma([t_0, T_{sup} + \epsilon))}  \| \phi_+ \|_{\mb{D}^{0}_\sigma([t_0, T_{sup} + \epsilon))} \Big)^\theta\\
       &\qquad \times \Big( \big( \| \psi \|_{\mb{F}^{s_0, \sigma}_M([t_0, T_n))} + \| \psi \|_{\mb{F}^{s_0, \sigma}_M([T_n, T_{sup} + \epsilon))}\big) \big( \| \phi_+ \|_{\mb{V}^{s_0+\frac{1}{2}, \sigma}_{+, m}([t_0, T_n))} + \| \phi_+ \|_{\mb{V}^{s_0+\frac{1}{2}, \sigma}_{+, m}([T_n, T_{sup} + \epsilon))}\big) \Big)^{1-\theta} \\
      &\les C^* A + \delta^\theta (2C^*)^{12} (1+A)^{2-\theta} A^{2-\theta}.
    \end{align*}
Consequently, provided $\delta \les [ (2C^*)^{12}(1+A)^{2-\theta} A^{1-\theta} ]^{-\frac{1}{\theta}}$, we see that
    $$ \| \psi \|_{\mb{F}^{s_0, \sigma}_M([t_0, T_{sup} + \epsilon))} \les 2 C^* A.$$
To bound $\phi$, we simply observe that another application of Theorem \ref{thm:main bilinear} implies that
    \begin{align*}
      \| \phi_+ \|_{\mb{V}^{s_0+\frac{1}{2}}_{+, m}([t_0, T_{sup} + \epsilon))} &\les C^* \| \psi_+(t_0) \|_{H^{s_0+\frac{1}{2}}_\sigma} + C^* \| \psi \|_{\mb{F}^{s_0, \sigma}_M([t_0, T_{sup}+\epsilon))}^2 \\
        &\les (2 C^*)^3 A (1+A).
    \end{align*}
Therefore $T_{sup}+\epsilon \in \mc{T}$ which contradicts the assumption $T_{sup} < t_1$. Consequently we must have $T_{sup} = t_1$ and hence $\psi \in \mb{F}^{s_0, \sigma}_{M}(I)$, $\phi_+ \in \mb{V}^{s_0+\frac{1}{2}}_{+, m}(I)$, and the claimed bounds hold. In the general case, when $t_0$ is not the left end point of $I$, we simply need to run the above argument for times smaller than $t_0$.\end{proof}

\section{Proof of Theorem \ref{thm:cond-dkg} and Corollary \ref{cor:rad}}\label{sec:proof of thm cond-dkg}
Here we give the proof of Theorem \ref{thm:cond-dkg} and Corollary \ref{cor:rad}.
\subsection{Proof of Theorem \ref{thm:cond-dkg}} We first observe that since \eqref{eqn:DKG} and \eqref{eqn:dkg reduced} are equivalent in the class of $H^{s_0}_\sigma$-strong solutions, and the Dirac-Klein-Gordon system is time reversible, it is enough to consider the forward in time problem for the reduced system \eqref{eqn:dkg reduced}. Thus let $(\psi, \phi_+):[t_0, t^*) \times \RR^3 \rightarrow \CC^4 \times \CC$ be a forward maximal $H^{s_0}_\sigma$-solution to \eqref{eqn:dkg reduced} such that
    $$ \sup_{t \in [t_0, t^*)} \Big( \| \psi(t) \|_{H^{s_0}_\sigma} + \| \phi_+(t) \|_{H^{s_0+\frac{1}{2}}_\sigma}\Big)\les A, \qquad \| \psi \|_{\mb{D}^{-\frac{1}{2}}_\sigma([t_0, t^*))}<\infty.$$
The finiteness of the dispersive norm $\| \cdot \|_{\mb{D}^{-\frac{1}{2}}_\sigma([t_0, t^*))}$ together with the Dominated Convergence Theorem, implies that for every $\delta>0$ we can find an interval $I=[t_1, t^*)$ with $t_1<t^*$ such that
        $$ \| \psi \|_{\mb{D}^{-\frac{1}{2}}_\sigma(I)} \les \delta.$$
In particular, choosing $\delta$ sufficiently small, depending only on $A$, an application of Theorem \ref{thm:main-det} implies that $(\psi, \phi_+) \in \mb{F}^{s_0, \sigma}_M(I)\times \mb{V}^{s_0+\frac{1}{2}, \sigma}_{+, m}(I)$. Therefore, by the existence of left limits in $V^2$, there exists $(\psi_0, \phi_0) \in H^{s_0}_\sigma \times H^{s_0+\frac{1}{2}}_\sigma$ such that
    $$ \lim_{t\rightarrow t^*} \Big( \big\|\psi(t) - \mc{U}_M(t) \psi_0 \big\|_{H^{s_0}_\sigma} + \big\| \phi(t) - \mc{U}^+_m(t) \phi_0 \big\|_{H^{s_0+\frac{1}{2}}_\sigma} \Big) = 0.$$
The local well-posedness theory in Theorem \ref{thm:lwp}, together with the definition of maximal $H^{s_0}_\sigma$-solution, then implies that we must have $t^*=\infty$. Consequently the solution $(\psi, \phi_+)$ exists globally in time and scatters as $t\rightarrow \infty$.

\subsection{Proof of Corollary \ref{cor:rad}} In view of Theorem \ref{thm:cond-dkg}, and the fact that the spinor remains in $\mc{H}$, and the wave component $\phi$ remains radial, it is enough to show that for $\psi_0 \in \mc{H}$, $1<p<\infty$, and $\sigma \g 0$ we have the bound 
        \begin{equation}\label{eqn:proof of cor:angular sum bound}  \sum_{N \in 2^\NN} N^\sigma \| H_N  \psi_0 \|_{L^p_x} \les \| \psi_0 \|_{L^p_x}.\end{equation}
However this follows directly from the definition of the the angular frequency localisation operators $H_N$, since the orthogonality of the spherical harmonics $y_{\ell, n}$ implies that for $\ell \g 2$ and $j=1, 2, 3$ we have
    $$\lr{1, y_{l, n} }_{L^2(\sph^2)} = \lr{\omega_j, y_{\ell, n}}_{L^2(\sph^2)} = 0.$$
Therefore $H_N \psi_0 = 0$ for $N \gg 1$ and hence \eqref{eqn:proof of cor:angular sum bound} follows.

    \section{Further Notation and Preliminary Results}\label{sec:prelim}

In this section we recall a number of results which will be used in the proof of the bilinear estimates in Theorem \ref{thm:main bilinear}. The setup and notation follows closely our previous paper \cite{Candy2016}, in particular, we refer to \cite{Candy2016} for further details and references. \\

We start by recalling the key fact that estimating a Duhamel term in $V^2_{\pm, m}$, can be reduced to estimating a bilinear integral. More precisely, suppose that $F \in L^\infty_t L^2_x$, and
      \[\sup_{ \| P_\lambda H_N v \|_{V^2_{\pm,m}} \lesa 1} \Big| \int_\RR \lr{ P_\lambda H_N v(t), F(t) }_{L^2_x} dt\Big| < \infty.\]
If $u \in C(\RR, L^2_x)$ satisfies $- i \p_t u \pm \lr{\nabla}_m u = F$, then $P_\lambda H_N u \in V^2_{\pm,m }$ and we have
      \begin{equation}\label{eqn:energy ineq for V2}\| P_\lambda H_N u \|_{V^2_{\pm ,m}} \lesa \|
        P_\lambda H_N u(0) \|_{L^2} + \sup_{ \| P_\lambda H_N v
          \|_{V^2_{\pm ,m}} \lesa 1} \int_\RR \lr{ P_\lambda H_N v(t),
          F(t) }_{L^2_x} dt .\end{equation}
An analogous bound holds without the angular frequency multiplier $H_N$. See, for instance, \cite[Lemma 7.3]{Candy2016} for a proof of this inequality.\\

        Let $\mc{Q}_\mu$ be a  collection of cubes of diameter proportional to $\mu$
        covering $\RR^3$ with uniformly finite overlap, and let $(\rho_q)_{q \in \mc{Q}_\mu}$ be a
        subordinate partition of unity.
For $q \in Q$ let
$P_q = \rho_q(|-i \nabla|)$. Given $ \alpha \les 1$ and a collection  $\mc{C}_\alpha$ of spherical caps $\kappa$ of diameter $\alpha$ with uniformly finite overlap, we let    $(\rho_\kappa)_{\kappa \in \mc{C}_\alpha}$ be a smooth partition
    of unity subordinate to the conic sectors spanned by $\kappa$, and define the
    angular Fourier localisation multipliers as
		$ R_\kappa = \rho_\kappa( - i \nabla)$. In certain regimes, we need to use the fact that the modulation localisation operators are uniformly disposable.
\begin{lemma}\label{lem:disposable} Let $1\leq q,r\leq \infty$, and $m>0$.
For any $0<\alpha\leq 1$, $\lambda \in 2^{\NN} $, $\kappa \in \mathcal{C}_\alpha$, $q\in \mc{Q}_{\alpha\lambda^2}$, and $d\in 2^{\ZZ}$ with $d\gtrsim \alpha^2 \lambda$, we have
                \begin{equation}\label{eqn:dispose}
                  \|C^{\pm, m}_{d} R_\kappa P_\lambda P_q u\|_{L^q_t L^r_x}+\|C^{\pm, m}_{\les d} R_\kappa P_\lambda P_q u\|_{L^q_t L^r_x}\lesa  \| R_\kappa  P_\lambda P_q u\|_{L^q_t L^r_x}.
                \end{equation}
Here, if $\alpha\gtrsim \lambda^{-1}$, the operator $P_q$ can be dropped, and if $\lambda\approx 1$, the operator $R_\kappa$ can be dropped.
Further, for
                every $d \in 2^\ZZ$
                \begin{equation}\label{eqn:dispose2}
                 \| C_{d}^{\pm, m} u \|_{V^2_{\pm,m}}  + \| C_{\les d}^{\pm, m} u \|_{V^2_{\pm,m}} \lesa \| u \|_{V^2_{\pm,m}}.
                \end{equation}
              \end{lemma}
\begin{proof}
First, it is enough to consider
 $0<\alpha\leq 1$, $\lambda \in 2^{\NN} $, $r_0\geq 0$ satisfying $\lr{r_0}\approx \lambda$ and $d\in 2^{\ZZ}$ with $d\gtrsim \alpha^2 \lambda$, and any function $u$ satisfying
\[\supp(\widehat{u(t)})\subset \{\xi \in \RR^{3}\mid ||\xi|-r_0|\lesa \alpha\lambda^2, \lr{\xi}\approx \lambda, r_0^{\frac12} (|\xi|-\xi_1)^{\frac12} \lesa \lambda\alpha\},\]
and to prove
the estimate
\begin{equation}
\|C^{\pm, m}_{\les d} u\|_{L^q_tL^r_x}\lesa \|u\|_{L^q_tL^r_x}.
\end{equation}
The support assumptions imply that we can write
\[
C^{\pm, m}_{\les d} u=\omega \ast  u
\]
for $\widetilde{\omega}$ a smooth bump function adapted to the set
\[\big\{(\tau,\xi) \in \RR^{1+3}\mid |\tau\pm\lr{\xi}|\lesa d, ||\xi|-r_0|\lesa \alpha\lambda^2, \lr{\xi}\approx \lambda, r_0^{\frac12} (|\xi|-\xi_1)^{\frac12} \lesa \lambda\alpha\big\}.\]
Multiple integration by parts yields
\[
|\omega(t,x)|\lesa_N d\beta\lambda (\alpha\lambda)^2\Big(1+d|t|+\beta\lambda\big|x_1\pm t\frac{r_0}{\lr{r_0}}\big|+ (\alpha\lambda)|(x_2,x_3)|\Big)^{-N},
\]
where $\beta=\min\{1,\alpha\lambda\}$, and $N\in\NN$.
Clearly, this implies \eqref{eqn:dispose}.

Second, from $C_{\les d}^{\pm,m}=\mc{U}^{\pm}_m P_{\les d} \mc{U}^{\pm}_m(-\cdot)$ and
\[
\|\sigma \ast v\|_{V^2}\lesa \|\sigma\|_{L^1}\|v\|_{V^2}.
\]
the estimate \eqref{eqn:dispose2} follows.
\end{proof}

In a similar vein, for all $1\leq p \leq \infty$,  we note that
\[\sup_{N\g 1} \|H_N\|_{L^p(\RR^3)\to L^p(\RR^3)}<+\infty,\]
and $H_N$ commutes with any radial Fourier
  multiplier such as $P_\lambda$. While $H_N$ also commutes with $C_d$, it does not commute with the cube and cap localisation operators $R_\kappa$ and $P_q$.
We also note an angular concentration bound \cite[Lemma 5.2]{Sterbenz2007}, which corresponds to \cite[Lemma 8.5]{Candy2016}. Let $2\les p < \infty$ and $0 \les s <\frac{2}{p}$. For all
        $\lambda, N \g 1$, $\alpha \gtrsim \lambda^{-1}$, and
        $\kappa \in \mc{C}_\alpha$ we have
        \begin{equation} \label{eqn:ang-con}
\| R_\kappa P_\lambda H_N f \|_{L^p_x(\RR^3)} \lesa  \alpha^s N^s \| P_\lambda H_N f \|_{L^p_x(\RR^3)}.
\end{equation}

    The proof of Theorem \ref{thm:main bilinear} requires carefully exploiting the structure of the product $\overline{\psi}\psi$. To this end, we recall a number of null form bounds that have been used frequently in the literature, see for instance \cite{Candy2016} and \cite{Bejenaru2015} for further details. We start be recalling that the multipliers $\Pi_{\pm}$ satisfy
        \begin{equation}\label{eqn:nullformbound1}
\big\| \big( \Pi_{\pm_1} - \Pi_{\pm_1}(\lambda
          \omega(\kappa)) \big) R_\kappa P_\lambda f \big\|_{L^r_x}
        \lesa \alpha \| R_\kappa P_\lambda u\|_{L^r_x},
        \end{equation}
provided that $\lambda \g 1$, $\alpha \gtrsim \lambda^{-1}$, and
        $\kappa \in \mc{C}_\alpha$. Similarly, in the Klein-Gordon regime we have
        \begin{equation}\label{eqn:nullformbound2}\big\| \big( \Pi_{\pm_1} - \Pi_{\pm_1}(\xi_0)
          \big) R_\kappa P_q P_\lambda f \big\|_{L^r_x}  \lesa \alpha
                                                          \| R_\kappa P_q P_\lambda u\|_{L^r_x},
                                                        \end{equation}
provided $\lambda\g 1$, $0<\alpha\lesssim \lambda^{-1}$,    $\kappa \in \mc{C}_\alpha$, $q\in Q_{\lambda^2\alpha}$ with        centre $\xi_0$.
Consequently, the identity
    \begin{align*}
      \big[\Pi_{\pm_1} f\big]^\dagger \gamma^0 \Pi_{\pm_2} g = \big[(\Pi_{\pm_1} &- \Pi_{\pm_1}(x))  f\big]^\dagger \gamma^0 \Pi_{\pm_2} g\\
                                                                                 &+ \big[\Pi_{\pm_1}(x) f\big]^\dagger \gamma^0  (\Pi_{\pm_2} - \Pi_{\pm_2}(y) ) g + f^\dagger \Pi_{\pm_1}(x) \gamma^0  \Pi_{\pm_2}(y)  g.
    \end{align*}
together with the pointwise null form type bound
    \begin{equation}\label{eqn:nullsymbolbound}|\Pi_{\pm_1}(x)
      \gamma^0 \Pi_{\pm_2}(y)| \lesa \theta(\pm_1 x, \pm_2 y) +
      \frac{\big| \pm_1 |x| \pm_2 |y| \big|}{\lr{x}
        \lr{y}}\end{equation}
then immediately implies, for instance, that
      \begin{equation}\label{eqn:null form bound I} \big\|
        \overline{\Pi_{\pm_1} R_{\kappa} \psi_{\lambda_1}} \Pi_{\pm_2} R_{\kappa'}\varphi_{\lambda_2} \big\|_{L^r_x} \lesa
        \alpha \| \psi \|_{L^a_x} \|\varphi
        \|_{L^b_x} \end{equation}
      where $\frac{1}{r} = \frac{1}{a}  +\frac{1}{b}$, $1<r, a, b < \infty$, and the caps $\kappa, \kappa'\in \mc{C}_\alpha$ satisfy $|\pm_1\kappa - \pm_2\kappa'| \lesa \alpha$, with $\alpha \gtrsim (\min\{\lambda_1, \lambda_2\})^{-1}$. Similarly, if we have the $V^2$ bound
	$$ \big\| \psi^\dagger \varphi \big\|_{L^r_x} \les \mb{C} \| \psi \|_{V^2_{\pm_1, M}} \| \varphi \|_{V^2_{\pm_2, M}} $$
        under some condition on $\supp \widehat{\psi}$ and
        $\supp \widehat{\varphi}$, then we also have, under the same
        conditions on $\supp \widehat{\psi}$ and
        $\supp \widehat{\varphi}$, the null form bound
        \begin{equation}\label{eqn:null form bound II} \big\|
          \overline{\Pi_{\pm_1}  R_{\kappa} \psi}
          \Pi_{\pm_2} R_{\kappa'} \varphi \big\|_{L^r_x} \lesa \alpha\mb{C} \|
          \psi \|_{V^2_{\pm_1, M}} \| \varphi \|_{V^2_{\pm_2,
              M}}. \end{equation}
These estimates are used frequently in the proof of Theorem \ref{thm:main bilinear}.\\

    Next, we recall the Strichartz estimates for the
    wave equation, which corresponds to \cite[Lemma 8.2]{Candy2016} and relies on \cite[Lemma 3.1]{Bejenaru2015} and     \cite[Appendix]{Sterbenz2007}.

\begin{lemma}[Wave Strichartz]\label{lem:wave strichartz}
  Let $m\g0$ and $2<q \les \infty$. If $0< \mu \les \lambda$, $N \g1$,
  and $\frac{1}{r} = \frac{1}{2} - \frac{1}{q}$ then for every
  $q \in Q_\mu$ we have
            \[ \| e^{\mp i t \lr{\nabla}_m} P_q P_\lambda f\|_{L^q_t L^r_x} \lesa \mu^{\frac{1}{2} - \frac{1}{r}} \lambda^{\frac{1}{2} - \frac{1}{r}} \| P_q P_\lambda f\|_{L^2_x}.\]
Moreover, if $\frac{1}{q} + \frac{2}{r} <1$ and $\epsilon>0$, we have
        \[ \| e^{\mp i t \lr{\nabla}_m}  P_\lambda H_N f \|_{L^q_t L^r_x} \lesa \lambda^{3(\frac{1}{2}-\frac{1}{r})- \frac{1}{q}} N^{\frac12+\epsilon} \| P_\lambda H_N f \|_{L^2_x}.\]
      \end{lemma}
We also require Strichartz
      estimates in the Klein-Gordon regime. The first bound can be found in, for instance, \cite[Lemma 3]{Machihara2003}, while the second is a special case of
        \cite[Theorem 1.1]{Cho2013} and corresponds to \cite[Lemma 8.3]{Candy2016}.

\begin{lemma}[Klein-Gordon Strichartz]\label{lem:KG strichartz}
  Let $m>0$. Then for $\frac{1}{4} \les \frac{1}{r} \les \frac{3}{10}$
  we have
        \[ \| e^{ \mp i t \lr{\nabla}_m} P_\lambda f \|_{L^r_{t,x}} \lesa \lambda^\frac{1}{2} \| P_\lambda f \|_{L^2_x}.\]
        On the other hand, if
        $\frac{3}{10}<\frac{1}{r} < \frac{5}{14}$ and $\epsilon>0$ we
        have
	\[ \| e^{ \mp i t \lr{\nabla}_m} P_\lambda H_N f  \|_{L^{r}_{t, x}} \lesa \lambda^{2-\frac{5}{r}} N^{7(\frac{1}{r} - \frac{3}{10})+\epsilon} \|P_\lambda H_N f\|_{L^2_x}.\]
      \end{lemma}
Both Lemma \ref{lem:wave strichartz} and Lemma \ref{lem:KG strichartz} have analogues in $V^2_{\pm, m}$. This follows by decomposing into $U^2_{\pm, m}$ atoms, and applying the estimate for free solutions, see for instance the arguments used in \cite[Proposition 2.19 and Corollary 2.21]{Hadac2009}. For instance, under the assumptions in Lemma \ref{lem:wave strichartz}, the first inequality in Lemma \ref{lem:wave strichartz} implies that
     \[ \|  P_q P_\lambda u \|_{L^q_t L^r_x} \lesa \mu^{\frac{1}{2} - \frac{1}{r}} \lambda^{\frac{1}{2} - \frac{1}{r}} \| P_q P_\lambda u\|_{V^2_{\pm, m}}.\]
Similarly, the remaining bounds in Lemma \ref{lem:wave strichartz} and Lemma \ref{lem:KG strichartz} imply corresponding versions in $V^2_{\pm, m}$. \\

   We now turn to the bilinear estimates that we require in the proof of Theorem \ref{thm:main bilinear}. To simplify the gain of the $L^4$ type norm in a particular high modulation interaction, we use the following ``cheap'' bilinear $L^2_{t,x}$ estimate in $V^2$.

    \begin{theorem}\label{thm:cheap bilinear}
      Let $\mu \ll \lambda_1 \approx \lambda_2$. Then for any
      $0<\gamma \les 1$ we have
      \begin{align*} \big\| P_{\mu} \big(\overline{ \Pi_{\pm_1} \psi_{\lambda_1} }  &\Pi_{\pm_2} \varphi_{\lambda_2} \big) \big\|_{L^2_{t,x}} \\
        &\lesa \mu \Big( \frac{\mu}{\lambda_1}\Big)^{ - 2\gamma} \Big(
          \lambda_1^{-\frac{1}{2}} \| \psi_{\lambda_1} \|_{L^4_{t,x}} \lambda_2^{-\frac{1}{2}} \| \psi_{\lambda_2} \|_{L^4_{t,x}} \Big)^\gamma \Big(\| \psi_{\lambda_1} \|_{V^2_{\pm_1, M}}
          \|\varphi_{\lambda_2} \|_{V^2_{\pm_2, M}}\Big)^{1-\gamma}.
 \end{align*}
\end{theorem}

\begin{proof}
  We consider separately the cases $\pm_1=\pm_2$ and
  $\pm_1 = - \pm_2$. In the former case, the fact that the output
  frequencies are restricted to be of size $\mu$, implies that we can
  directly exploit the null structure together with the standard
  $L^4_{t,x}$ Strichartz bound to deduce that for every $\epsilon>0$
  we have
  \begin{align*}
    \big\| P_{\mu} \big(\overline{ \Pi_{\pm_1} \psi_{\lambda_1} }  \Pi_{\pm_2} \varphi_{\lambda_2} \big)\big\|_{L^2_{t,x}} &\lesa \frac{\mu}{\lambda_1} \sum_{\substack{q, q' \in Q_\mu \\ |q-q'| \approx \mu}}  \| P_q \psi_{\lambda_1} \|_{L^4_{t,x}} \| P_{q'} \varphi_{\lambda_2} \|_{L^4_{t,x}} \\
                                                                                                                                             &\lesa \mu \Big( \frac{\mu}{\lambda_1} \Big)^{\frac{1}{2}-\epsilon} \| \psi_{\lambda_1} \|_{V^2_{\pm_1, M}} \| \varphi_{\lambda_2} \|_{V^2_{\pm_2, M}}.
  \end{align*}	
  On the other hand, in the $\pm_1 = - \pm_2$ case, the frequency
  restriction implies that the free waves
  $ (-i \p_t \pm_1 \lr{\nabla}_M )\psi = (-i \p_t \pm_2 \lr{\nabla} )
  \varphi = 0$
  are fully transverse. In particular, for free waves, we have the
  bilinear estimate
  \begin{align*}
    \big\| P_{\mu} \big(\overline{ \Pi_{\pm_1} \psi_{\lambda_1} } \Pi_{\pm_2} \varphi_{\lambda_2} \big) \big\|_{L^2_{t,x}} \lesa \mu \| \psi_{\lambda_1}(0) \|_{L^2_x} \| \varphi_{\lambda_2}(0) \|_{L^2_x}
  \end{align*}
  see for instance \cite[Lemma 2.6]{Candy2016}. Arguing as in \cite[Proposition 2.19 and Corollary 2.21]{Hadac2009}, by interpolating this with the standard $L^4_{t,x}$ bound, we deduce that for every $\epsilon>0$ we have the $V^2_{\pm, M}$ bound
	$$ \big\| P_{\mu} \big[ \big( \Pi_{\pm_1} \psi_{\lambda_1} \big)^\dagger \gamma^0  \Pi_{\pm_2} \varphi_{\lambda_2} \big] \big\|_{L^2_{t,x}} \lesa \mu \Big( \frac{\mu}{\lambda_1} \Big)^{-\epsilon} \| \psi_{\lambda_1} \|_{V^2_{\pm_1, M}} \| \varphi_{\lambda_2} \|_{V^2_{\pm_2, M}}.$$
        Thus in either case of the $\pm$ cases, by $L^p$ interpolation, we
        obtain
	\begin{align*}
          \big\| P_{\mu} \big[ \big( \Pi_{\pm_1} \psi_{\lambda_1} \big)^\dagger \gamma^0  &\Pi_{\pm_2} \varphi_{\lambda_2} \big] \big\|_{L^2_{t,x}} \\
                                                                                          &\lesa
                                                                                            \mu
                                                                                            \Big(
                                                                                            \frac{\mu}{\lambda_1}\Big)^{-\epsilon(1-\gamma)
                                                                                            -
                                                                                            \gamma}
                                                                                            \Big(
                                                                                            \lambda_1^{-\frac{1}{2}}
                                                                                            \|
                                                                                            \psi_{\lambda_1}
                                                                                            \|_{L^4_{t,x}}
                                                                                            \lambda_2^{-\frac{1}{2}}
                                                                                            \|
                                                                                            \psi_{\lambda_2}
                                                                                            \|_{L^4_{t,x}}
                                                                                            \Big)^\gamma
                                                                                            \Big(\|
                                                                                            \psi_{\lambda_1}
                                                                                            \|_{V^2_{\pm_1,
                                                                                            M}}
                                                                                            \|\varphi_{\lambda_2}\|_{V^2_{\pm_2,
                                                                                            M}}\Big)^{1-\gamma}.
        \end{align*}
        Therefore, the result follows by choosing $\epsilon$ sufficiently
        small.
      \end{proof}

        The main bilinear input in the critical case is the following
        bilinear restriction type bound from \cite{Candy2017}, which
        extends the corresponding bound in \cite{Candy2016}. The key point in the following is that the estimate holds for functions in $V^2_{\pm, m}$, in the full bilinear range. This is a highly non-trivial observation, which, in contrast to the linear and bilinear estimates mentioned above, \emph{does not} follow from the same bounds for free homogeneous solutions. Instead it requires a more involved direct argument, see \cite{Candy2016, Candy2017} for further details.

\begin{theorem}\label{thm:bilinear small scale KG}
Let $\epsilon>0$, $1\les q, r \les2$, $\frac{1}{q} + \frac{2}{r}<2$. For all $m_1, m_2 \g 0$, $0< \alpha \les 1$, and  $\xi_0, \eta_0 \in \RR^3$ such that $\lr{\xi_0}_{m_1} \approx \lambda_1$, $\lr{\eta_0}_{m_2} \approx \lambda_2$, and
    \[ \frac{ \big| m_2 |\xi_0| - m_1 |\eta_0| \big|}{\lambda_1 \lambda_2} + \Big( \frac{|\xi_0| |\eta_0| \mp \xi_0 \cdot \eta_0}{\lambda_1 \lambda_2} \Big)^\frac{1}{2} \approx \alpha, \] and for all $u,v$ satisfying
    \[ \supp \widehat{u} \subset \big\{ \big| |\xi| - |\xi_0| \big| \ll  \beta \lambda_1 , \,\,  \big( |\xi| |\xi_0| - \xi \cdot \xi_0 \big)^\frac{1}{2} \ll \alpha  \lambda_1 \big\}, \quad \supp \widehat{v} \subset \big\{ \big| |\xi| - |\eta_0| \big| \ll  \beta  \lambda_2, \,\,  \big( |\xi| |\eta_0| - \xi \cdot \eta_0 \big)^\frac{1}{2} \ll \alpha  \lambda_2 \big\}\]
then we have the bilinear estimate
   \[ \| uv \|_{L^q_t L^r_x} \lesa \alpha^{2 - \frac{2}{r}  - \frac{2}{q}} \beta^{1-\frac{1}{r}}  \lambda_{min}^{3-\frac{3}{r} - \frac{1}{q}} \Big( \frac{\lambda_{max}}{\lambda_{min}}\Big)^{\frac{1}{q} - \frac{1}{2}+\epsilon} \| u  \|_{V^2_{\pm_1,m_1}} \| v \|_{V^2_{\pm_2, m_2}} \]
 where $\lambda_{min} = \min\{\lambda_1, \lambda_2\}$,
    $\lambda_{max} = \max\{\lambda_1, \lambda_2\}$, and $\beta = ( \frac{m_1}{\alpha \lambda_1} + \frac{m_2}{\alpha
          \lambda_2} + 1)^{-1}$.
\end{theorem}

\section{Multilinear Estimates in the subcritical case}\label{sec:multi-sub}
Our aim of this section is to establish the following result, which applies to the non-resonant regime. Here, after rescaling, we have $m=1$ and $M>\frac12$ now.
        \begin{theorem}\label{thm:duhamel-sub}
          Let $M>\frac{1}{2}$. If $s_0>0$ is sufficiently small, there exists $0<\theta<1$ and $1<a<2$ and $C>1$, such that if $b=4(\frac{1}{a}-\frac{1}{2})$, for all $s\g s_0 >0$       \begin{equation}\label{eqn:thm duhamel sub:V psi}
                \begin{split}
                    \big\|\Pi_{\pm_1}\mathcal{I}_M^{\pm_1}[\phi \gamma^0\Pi_{\pm_2}\varphi]\big\|_{\mb{V}^{s}_{\pm_1,M}}
                        &\lesa  \sup_{\mu, \lambda_2 \ge 1}\Big(\|\phi_\mu\|_{L^4_{t,x}}\lambda_2^{s-s_0-\frac12}\|\varphi_{\lambda_2}\|_{L^4_{t,x}}\Big)^\theta
                        \Big(\|\phi\|_{\mb{V}^{\frac12+s_0}_{+,1}} \|\varphi\|_{\mb{F}_{\pm_2, M}^s}\Big)^{1-\theta} \\
                        &\qquad + \sup_{\mu,\lambda_2\ge1}\Big(\mu^{s-s_0}\|\phi_\mu\|_{L^4_{t,x}}\lambda_2^{-\frac12}
                        \|\varphi_{\lambda_2}\|_{L^4_{t,x}}\Big)^\theta
                        \Big(\|\phi\|_{\mb{V}^{\frac12 +s}_{+,1}} \|\varphi\|_{\mb{F}_{\pm_2, M}^{s_0}}\Big)^{1-\theta},
          \end{split}
        \end{equation}
          and
           \begin{equation}\label{eqn:thm duhamel sub:Y psi}
\begin{split}
            \big\|\Pi_{\pm_1}\mathcal{I}_M^{\pm_1}[\phi \gamma^0
            \Pi_{\pm_2}\varphi]\big\|_{\mb{Y}^{s}_{\pm_1,M}}
                &\lesa  \sup_{\mu, \lambda_2 \ge 1}\Big(\|\phi_\mu\|_{L^4_{t,x}}\lambda_2^{s-s_0-\frac12}\|\varphi_{\lambda_2}\|_{L^4_{t,x}}\Big)^\theta
            \Big(\|\phi\|_{\mb{V}^{\frac12+s_0}_{+,1}} \|\varphi\|_{\mb{V}_{\pm_2, M}^s}\Big)^{1-\theta} \\
&\qquad +\sup_{\mu, \lambda_2 \ge 1}\big(\mu^{s-s_0}\|\phi_\mu\|_{L^4_{t,x}}\lambda_2^{-\frac12}\|\varphi_{\lambda_2}\|_{L^4_{t,x}}\big)^\theta
            \big(\|\phi\|_{\mb{V}^{\frac12 +s}_{+,1}} \|\varphi\|_{\mb{V}_{\pm_2, M}^{s_0}}\big)^{1-\theta}.
          \end{split}
        \end{equation}
Similarly,
          \begin{equation}\label{eqn:thm duhamel sub:V phi}
\begin{split}
            \big\|\langle \nabla\rangle^{-1}\mathcal{I}_m^{+}[(\Pi_{\pm_1}\psi)^\dagger &\gamma^0 \Pi_{\pm_2}\varphi]\big\|_{\mb{V}^{\frac{1}{2}+s}_{+,1}} \\
              &\lesa  \sup_{\lambda_1,\lambda_2 \ge 1}\Big(\lambda_1^{-\frac{1}{2}} \|\psi_{\lambda_1} \|_{L^4_{t,x}}\lambda_2^{s-s_0-\frac{1}{2}} \|\varphi_{\lambda_2} \|_{L^4_{t,x}}\Big)^\theta \Big(\| \psi \|_{\mb{F}_{\pm_1, M}^{s_0}} \|\varphi\|_{\mb{F}_{\pm_2, M}^s}\Big)^{1-\theta}\\
              &\qquad + \sup_{\lambda_1,\lambda_2 \ge 1}\Big(\lambda_1^{s-s_0-\frac{1}{2}} \|\psi_{\lambda_1} \|_{L^4_{t,x}}\lambda_2^{-\frac{1}{2}} \|\varphi_{\lambda_2} \|_{L^4_{t,x}}\Big)^\theta \Big(\| \psi \|_{\mb{F}_{\pm_1, M}^s} \|\varphi\|_{\mb{F}_{\pm_2, M}^{s_0}}\Big)^{1-\theta}.
            \end{split}
          \end{equation}
        \end{theorem}
        The proof, which we postpone to Subsection
        \ref{subsec:proof-duhamel-sub} below, relies on trilinear
        estimates.

        \subsection{A subcritical trilinear
          estimate}\label{subsec:tri-sub}
Here, we consider frequency localised estimates and use the shorthand notation
$f_\lambda=P_\lambda f$.
\begin{theorem}\label{thm:trilinear freq loc subcrit}
  Let $M>\frac{1}{2}$, $0<\varrho \ll 1$, $\frac{5}{3}<a<2$ and $0<b<\frac{\varrho}{4}$.  Let $\phi:\RR^{1+3} \rightarrow \CC$, and $\varphi, \psi: \RR^{1+3} \rightarrow \CC^4$ such that $\Pi_{\pm_1} \psi = \psi$ and $\Pi_{\pm_2} \varphi = \varphi$. Define
          $$ \mb{A}= \| \phi_\mu\|_{L^4_{t,x}} \lambda_1^{-\frac{1}{2}} \| \psi_{\lambda_1} \|_{L^4_{t,x}} \lambda_2^{-\frac{1}{2}} \| \varphi_{\lambda_2} \|_{L^4_{t,x}}. $$
  There exists $\theta_0\in (0,1)$ such that
   \begin{equation}\label{eqn:tri freq loc subcrit:V2 main} \Big| \int_{\RR^{1+3}} \phi_\mu \overline{\psi_{\lambda_1}} \varphi_{\lambda_2} dx dt \Big| \lesa  \mu^\varrho \Big( \frac{\mu}{\max\{\lambda_1, \lambda_2\}}\Big)^{\frac{1}{10}}  \mb{A}^{\theta_0} \Big( \mu^{\frac{1}{2}} \| \phi_{\mu} \|_{V^2_{+, 1}} \| \psi_{\lambda_1} \|_{V^2_{\pm_1,M}}\| \varphi_{\lambda_2} \|_{V^2_{\pm_2, M}}\Big)^{1-\theta_0} .
  \end{equation}
  If $\lambda_1 \gg \lambda_2$, we can improve this to
   \begin{equation}\label{eqn:tri freq loc subcrit:V2-A0}
        \begin{split}
          \Big| \int_{\RR^{1+3}} \phi_\mu \overline{\psi_{\lambda_1}} \varphi_{\lambda_2} - \sum_{d \lesa \lambda_2} C_{\les d}&\phi_\mu \overline{\mc{C}_{\les d}^{\pm_1}\psi_{\lambda_1}} \mc{C}_d^{\pm_2} \varphi_{\lambda_2} dx dt \Big| \\
                &\lesa \lambda_2^{\varrho} \Big( \frac{\lambda_2}{\lambda_1}\Big)^{^{\frac{1}{10}}}  \mb{A}^{\theta_0} \Big( \mu^{\frac{1}{2}} \| \phi_{\mu} \|_{V^2_{+, 1}} \| \psi_{\lambda_1} \|_{V^2_{\pm_1,M}}\| \varphi_{\lambda_2} \|_{V^2_{\pm_2, M}}\Big)^{1-\theta_0} .
        \end{split}
   \end{equation}
   and
     \begin{equation}\label{eqn:tri freq loc subcrit:Y}
       \Big| \int_{\RR^{1+3}} \sum_{d \lesa \lambda_2} C_{\les d}\phi_\mu \overline{\mc{C}_{\les d}^{\pm_1}\psi_{\lambda_1}} \mc{C}_d^{\pm_2} \varphi_{\lambda_2} dx dt \Big|
                \lesa \lambda_2^{\varrho} \Big( \frac{\lambda_2}{\lambda_1}\Big)^{\frac{1}{4}(\frac{1}{a}-\frac{1}{2})}   \mb{A}^{\theta_0} \Big( \mu^{\frac{1}{2}} \| \phi_{\mu} \|_{V^2_{+, 1}} \| \psi_{\lambda_1} \|_{V^2_{\pm_1,M}}\| \varphi \|_{Y^{\pm_2, M}_{\lambda_2}}\Big)^{1-\theta_0}
     \end{equation}
  Similarly, if $\lambda_1 \ll \lambda_2$, we have
   \begin{equation}\label{eqn:tri freq loc subcrit:V2-A0 alt}
        \begin{split}
          \Big| \int_{\RR^{1+3}} \phi_\mu \overline{\psi_{\lambda_1}} \varphi_{\lambda_2} - \sum_{d \lesa \lambda_1} C_{\les d}&\phi_\mu \overline{\mc{C}_{d}^{\pm_1}\psi_{\lambda_1}} \mc{C}_{\les d}^{\pm_2} \varphi_{\lambda_2} dx dt \Big| \\
                &\lesa \lambda_1^{\varrho} \Big( \frac{\lambda_1}{\lambda_2}\Big)^{\frac{1}{10}}   \mb{A}^{\theta_0} \Big( \mu^{\frac{1}{2}} \| \phi_{\mu} \|_{V^2_{+, 1}} \| \psi_{\lambda_1} \|_{V^2_{\pm_1,M}}\| \varphi_{\lambda_2} \|_{V^2_{\pm_2, M}}\Big)^{1-\theta_0} .
        \end{split}
   \end{equation}
   and
     \begin{equation}\label{eqn:tri freq loc subcrit:Y alt}
       \Big| \int_{\RR^{1+3}} \sum_{d \lesa \lambda_1} C_{\les d}\phi_\mu \overline{\mc{C}_{d}^{\pm_1}\psi_{\lambda_1}} \mc{C}_{\les  d}^{\pm_2} \varphi_{\lambda_2} dx dt \Big|
                \lesa \lambda_1^{\varrho} \Big( \frac{\lambda_1}{\lambda_2}\Big)^{\frac{1}{4}(\frac{1}{a}-\frac{1}{2})}  \mb{A}^{\theta_0} \Big( \mu^{\frac{1}{2}} \| \phi_{\mu} \|_{V^2_{+, 1}} \| \psi \|_{Y^{\pm_1,M}_{\lambda_1}} \| \varphi_{\lambda_2} \|_{V^2_{\pm_2, M}}\Big)^{1-\theta_0}.
     \end{equation}
\end{theorem}
\begin{proof}
The first step is to decompose the trilinear product into the modulation localised terms
  \begin{align*} \phi_{\mu} \big( \Pi_{\pm_1}& \psi_{\lambda_1}\big)^\dagger \gamma^0 \Pi_{\pm_2} \varphi_{\lambda_2} \\
                                             &=\sum_{d} C_d \phi_{\mu} \big( \mc{C}_{\ll d}^{\pm_1} \psi_{\lambda_1}\big)^\dagger \gamma^0 \mc{C}_{\ll d}^{\pm_2} \varphi_{\lambda_2} + C_{\lesa d} \phi_{\mu} \big( \mc{C}_{ d}^{\pm_1} \psi_{\lambda_1}\big)^\dagger \gamma^0 \mc{C}_{\lesa d}^{\pm_2} \varphi_{\lambda_2}+C_{\lesa d} \phi_{\mu} \big( \mc{C}_{\lesa d}^{\pm_1} \psi_{\lambda_1}\big)^\dagger \gamma^0 \mc{C}_{d }^{\pm_2} \varphi_{\lambda_2}\\
                                             &=\sum_d A_0 + A_1 + A_2.
  \end{align*}
  We now consider separately the small modulation cases
	$$ \mu \lesa \lambda_1 \approx \lambda_2 \text{ and } d \lesa \mu, \qquad \qquad \mu  \gg \min\{ \lambda_1, \lambda_2\} \text{ and } d \lesa \min\{\lambda_1, \lambda_2\} $$
        and the high modulation cases
	$$ \mu \lesa \lambda_1 \approx \lambda_2 \text{ and } d \gg \mu, \qquad \qquad \mu \gg \min\{ \lambda_1, \lambda_2\} \text{ and } d \gg \min\{\lambda_1, \lambda_2\}. $$
Due to the $L^4_{t,x}$ Strichartz inequality we have the obvious bound
        $$\mb{A} \lesa  \mu^\frac{1}{2}\| \phi_{\mu} \|_{V^2_{+, 1}} \| \psi_{\lambda_1} \|_{V^2_{\pm_1,M}}\| \varphi_{\lambda_2} \|_{V^2_{\pm_2, M}}.$$
Therefore, if we have the bounds \eqref{eqn:tri freq loc subcrit:V2 main}, \eqref{eqn:tri freq loc subcrit:V2-A0}, and \eqref{eqn:tri freq loc subcrit:V2-A0 alt} for some $\theta_0$, we can always replace $\theta_0$ with a smaller factor. In particular, in each of the various cases considered below, it is enough to get some (potentially very small) power of $\mb{A}$, the factor $\theta_0$ can then be taken to the minimum of the powers obtained. In the remaining estimates involving the $\|\cdot\|_{Y_{\lambda_j}^{\pm_j,M}}$-norms, as this norm \emph{does not} give control the $L^4_{t,x}$, we directly verify the fact that it is possible to make the exponent of $\mb{A}$ smaller if needed.\\

Before we start the case by case analysis, we recall some facts from the modulation analysis for the Dirac-Klein-Gordon system \cite{D'Ancona2007b,Bejenaru2015,Candy2016}. As in \cite[Subsection 8.2]{Candy2016} we define the modulation function
\[\mathfrak{M}_{\pm_1,\pm_2}(\xi,\eta)=|\lr{\xi-\eta}\mp_1\lr{\xi}\pm_2\lr{\eta}|\]
where we take $\xi \in \supp  \widehat{\psi}$, $\eta \in \supp \widehat{\varphi}$, and $\xi - \eta \in \supp \widehat{\phi}$. A computation shows that, for any $j=0,1,2$,
\begin{equation}\label{eq:m-b}
\int_{\RR^{1+3}}A_j dxdt\ne 0 \quad \text{ implies }  \quad \mathfrak{M}_{\pm_1,\pm_2}(\xi,\eta)\lesa d \text{ within the domain of integration.}
\end{equation}
Furthermore, since $m=1$ and $M>\frac12$ we are in the non-resonant case where
\begin{equation}\label{eq:global-nr}
\mathfrak{M}_{\pm_1,\pm_2}\gtrsim (\min\{\mu,\lambda_1,\lambda_2\})^{-1}.
\end{equation}
In particular, in the case by case analysis below, we may assume that $d \gtrsim (\min\{\mu,\lambda_1,\lambda_2\})^{-1}$. \\

\textbf{Case 1: $\mu \lesa \lambda_1 \approx \lambda_2$ and $d \lesa \mu$.} We first observe that the sum over the modulation is restricted to the region $\mu^{-1} \lesa d \lesa \mu$. Moreover, the resonance identities in \cite[Lemma 8.7]{Candy2016} together with \eqref{eq:m-b} imply that
\begin{equation}\label{eq:res-id1}
 d \gtrsim \mathfrak{M}_{\pm_1,\pm_2}(\xi,\eta) \gtrsim \frac{\lambda_1^2}{\mu}\theta^2(\pm_1\xi,\pm_2\eta)+\mu\theta^2(\xi-\eta,\pm_1\xi)+\mu\theta^2(\xi-\eta,\pm_2\eta).
\end{equation}
 Consequently, the angle between the Fourier supports of $\psi$ and $\varphi$
        must be of size $\alpha = (\frac{d\mu}{\lambda_1^2})^\frac{1}{2}$.
        In particular, decomposing $\psi$ and $\varphi$ into caps of
        size $\alpha$, and cubes of size $\mu$, applying H\"older's
        inequality, and using the null form bound (\ref{eqn:null form
          bound I}) we deduce that for every $0<\delta<1$
	\begin{align}
         & \Big| \int_{\RR^{1+3}} A_0 dx dt \Big| \lesa \alpha \sum_{\substack{q, q' \in \mc{Q}_\mu \\ |q-q'| \approx \mu}} \sum_{\substack{\kappa, \kappa'\in \mc{C}_\alpha \\ |\pm_1 \kappa - \pm_2  \kappa'| \approx \alpha }} \| C_d \phi_\mu \|_{L^2_{t,x}} \| P_q R_\kappa \psi_{\lambda_1} \|_{L^4_{t,x}} \| P_{q'} R_{\kappa'} \varphi_{\lambda_2} \|_{L^4_{t,x}} \notag \\
                                                 \lesa{}& \Big( \frac{d}{\mu}\Big)^{-2\delta} \Big( \frac{\mu}{\lambda_1} \Big)^{\frac{1}{2} -8\delta}  \mu^{\frac{1}{2}} \|\phi_\mu \|_{V^2_{+,1}}\Big( \lambda_1^{-\frac{1}{2}}\lambda_2^{-\frac{1}{2}} \| \psi_{\lambda_1} \|_{L^4_{t,x}} \| \varphi_{\lambda_2} \|_{L^4_{t,x}}\Big)^{\delta} \Big(\| \psi_{\lambda_1} \|_{V^2_{\pm_1, M}}  \| \varphi_{\lambda_2} \|_{V^2_{\pm_2, M}}\Big)^{1-\delta} \label{eqn:thm trilinear subcrit:case 1 A_0}
	\end{align}
        where we absorbed the square over caps and cubes using
        \cite[Lemma 8.6]{Candy2016}, which gives, in particular
  \begin{equation}\label{eqn:thm trilinear subcrit:example of square sum control}\Big( \sum_{q \in \mc{Q}_{\mu} }\sum_{\kappa \in \mc{C}_{\alpha}} \| P_q R_{\kappa} \psi_{\lambda_1} \|_{L^4_{t,x}}^2 \Big)^{\frac{1}{2}} \lesa \alpha^{-2\delta} \Big(\frac{\mu}{\lambda_1}\Big)^{-2\delta}(\mu \lambda_1 )^{\frac{1}{4}}\Big( \lambda_1^{-\frac{1}{2}} \| \psi_{\lambda_1} \|_{L^4_{t,x}}\Big)^{\delta} \| \psi_{\lambda_1} \|_{V^2_{\pm_1, M}}^{1-\delta} .\end{equation}
 On the other hand, applying the Klein-Gordon Strichartz estimate and \eqref{eq:res-id1}, we deduce that for
 $\frac{10}{3} \les r <4$ and
 $\beta = ( \frac{d}{\mu} )^{\frac{1}{2}}$
 \begin{align*}
   \Big| \int_{\RR^{1+3}} A_0 dx dt \Big| &\lesa \alpha  \sum_{\substack{\kappa, \kappa'\in \mc{C}_\alpha \\ |\pm_1 \kappa - \pm_2  \kappa'| \approx \alpha }}  \sum_{\substack{\kappa''\in \mc{C}_\beta \\ | \kappa'' - \pm_2  \kappa'| \lesa \beta }} \| C_d R_{\kappa''} \phi_\mu \|_{L^\frac{r}{r-2}_{t,x}} \| R_\kappa \mc{C}_{\ll d}^{\pm_1} \psi_{\lambda_1} \|_{L^r_{t,x}} \| R_{\kappa'} \mc{C}_{\ll d}^{\pm_2} \varphi_{\lambda_2} \|_{L^r_{t,x}} \\
                                          &\lesa \Big( \frac{d}{\mu} \Big)^{\frac{4}{r}-1 - 2\delta} \Big( \frac{\mu}{\lambda_1} \Big)^{-6\delta} \mu^{\frac{8}{r} - 2}  \| \phi_\mu \|_{L^4_{t,x}}^{\frac{8}{r}-2} \Big( \mu^\frac{1}{2} \| \phi_\mu \|_{V^2_{+,1}} \Big)^{3-\frac{8}{r}} \| \psi_{\lambda_1} \|_{V^2_{\pm_1, M}} \| \varphi_{\lambda_2}\|_{V^2_{\pm_2, M}}.
 \end{align*}
 Combining these bounds with $r = \frac{10}{3}$ then gives, for every
 $0<\delta, \theta<\frac{2}{5}$,
	$$ \Big| \int_{\RR^{1+3}} A_0 dx dt \Big| \lesa \Big(\frac{d}{\mu}\Big)^{ - 2 \delta} \Big( \frac{\mu}{\lambda_1}\Big)^{\frac{1}{2} - 6 (\theta + \delta)} \mu^{ \frac{2}{5} \theta}  \mb{A}^{\delta \theta} \Big( \mu^{\frac{1}{2}} \| \phi_\mu\|_{V^2_{+,1}} \| \psi_{\lambda_1 } \|_{V^2_{\pm_1, M}} \| \varphi_{\lambda_2} \|_{V^2_{\pm_2, M}}\Big)^{1-\delta \theta}.  $$
        Summing up over $\mu^{-1} \lesa d \lesa \mu$, and choosing
        $\delta, \theta>0$ sufficiently small, then gives the required
        estimate for the $A_0$ term.

        To estimate the $A_1$ term, we again put the high modulation
        term into $L^2_{t,x}$, which gives
	\begin{align}
          \Big| \int_{\RR^{1+3}} &A_1 dx dt \Big| \notag \\
          &\lesa \alpha  \sum_{\substack{\kappa, \kappa'\in \mc{C}_\alpha \\ |\pm_1 \kappa - \pm_2  \kappa'| \lesa \alpha }}  \sum_{\substack{\kappa''\in \mc{C}_\beta \\ | \kappa'' - \pm_2  \kappa'| \lesa \beta }} \| R_{\kappa''} C_{\lesa d} \phi_\mu \|_{L^4_{t,x}} \| R_{\kappa} \mc{C}_d^{\pm_1} \psi_{\lambda_1} \|_{L^2_{t,x}} \| R_{\kappa'} \mc{C}_{\lesa d}^{\pm_2}\varphi_{\lambda_2} \|_{L^4_{t,x}} \notag \\
                                                 &\lesa \Big( \frac{d}{\mu}\Big)^{-2\delta} \Big( \frac{\mu}{\lambda_1}\Big)^{\frac{1}{2} - 6\delta} \Big(  \| \phi_\mu \|_{L^4_{t,x}} \lambda_2^{-\frac{1}{2}} \| \varphi_{\lambda_2} \|_{L^4_{t,x}}\Big)^{\delta} \Big( \mu^\frac{1}{2} \| \phi_\mu \|_{V^2_{+, 1}}\| \varphi_{\lambda_2} \|_{V^2_{\pm_2, M}}\Big)^{1-\delta}\| \psi_{\lambda_1} \|_{V^2_{\pm_1, M}}.
                                                 \label{eqn:thm trilinear subcrit:case 1 A_1}
	\end{align}
        To gain an $L^4_{t,x}$ norm of $\psi_{\lambda_1}$, we essentially repeat the argument used in the $A_0$ case. More precisely, we observe that using the Klein-Gordon Strichartz estimate for $\phi_\mu$ and $\varphi_{\lambda_2}$ gives for every $\frac{10}{3}\les r < 4$ and $\delta>0$
            \begin{align*}
              \Big| \int_{\RR^{1+3}} &A_1 dx dt \Big|\\
                &\lesa \alpha \sum_{\substack{\kappa, \kappa'\in \mc{C}_\alpha \\ |\pm_1 \kappa - \pm_2  \kappa'| \lesa \alpha }} \| C_{\lesa d} \phi_\mu \|_{L^r_{t,x}} \| C_d R_\kappa \psi_{\lambda_1} \|_{L^2_{t,x}}^{3-\frac{8}{r}} \| C_d R_\kappa \psi_{\lambda_1} \|_{L^4_{t,x}}^{\frac{8}{r} - 2} \| C_{\lesa d} R_{\kappa'} \varphi_{\lambda_2} \|_{L^r_{t,x}} \\
                &\lesa \Big( \frac{d}{\mu} \Big)^{\frac{4}{r} -1 -4\delta} \Big( \frac{\mu}{\lambda_1} \Big)^{\frac{3}{2} -\frac{4}{r} -  4 \delta} \mu^{\frac{4}{r} - 1} \mu^\frac{1}{2} \| \phi_\mu \|_{V^2_{+, 1}}\Big( \lambda_1^{-\frac{1}{2}} \| \psi_{\lambda_1} \|_{L^4_{t,x}}\Big)^{\frac{8}{r} -2} \| \psi_{\lambda_1} \|_{V^2_{\pm_1, M}}^{ 3 - \frac{8}{r}}\| \varphi_{\lambda_2} \|_{V^2_{\pm_2, M}}
            \end{align*}
        where the square sums over caps were again controlled by using \cite[Lemma 8.6]{Candy2016}. Together with (\ref{eqn:thm trilinear subcrit:case 1 A_1}), this completes the proof for the $A_1$ component. An identical computation gives an acceptable bound for the $A_2$ term. Hence, by choosing $\delta$ sufficiently small and summing up over $\mu^{-1} \lesa d \lesa \mu$, we get the required bound in Case 1. \\

        \textbf{Case 2: $\mu \gg \min\{\lambda_1, \lambda_2\}$ and $d \lesa \min\{\lambda_1, \lambda_2\}$.} We only consider the case
        $\lambda_1 \g \lambda_2$, as the remaining case is
        identical. As previously, we first estimate $A_0$ by placing
        $\phi_\mu \in L^2$ and $\psi_{\lambda_1}, \varphi_{\lambda_2} \in L^4$.
From the resonance identities in \cite[Lemma 8.7]{Candy2016} together with \eqref{eq:m-b} we obtain
\begin{equation}\label{eq:res-id2}
 d \gtrsim \mathfrak{M}_{\pm_1,\pm_2}(\xi,\eta) \gtrsim \frac{\mu^2}{\lambda_2}\theta^2(\xi-\eta,\pm_1\xi)+\lambda_2\theta^2(\pm_1\xi,\pm_2\eta)+\lambda_2\theta^2(\xi-\eta,\pm_2\eta).
\end{equation}
Hence,  if we let $\beta = (\frac{d}{\lambda_2})^\frac{1}{2}$ we obtain
	\begin{align}
          \Big| &\int_{\RR^{1+3}} A_0 dx dt \Big|\lesa \beta  \sum_{\substack{q, q'' \in Q_{\lambda_2} \\ |q-q''| \approx \lambda_2}} \sum_{\substack{ \kappa, \kappa' \in \mc{C}_\beta \\ |\pm_1 \kappa - \pm_2\kappa'|\lesa \beta}} \| P_{q''} C_d \phi_\mu \|_{L^2_{t,x}} \| R_\kappa P_q \psi_{\lambda_1} \|_{L^4_{t,x}} \| R_{\kappa'}  \varphi_{\lambda_2}\|_{L^4_{t,x}} \notag \\
                                                 &\lesa \Big( \frac{d}{\lambda_2}\Big)^{-4\delta} \Big( \frac{\lambda_2}{\mu} \Big)^{\frac{1}{4} - 2 \delta}  \mu^\frac{1}{2} \| \phi_\mu \|_{V^2_{+,1}}  \Big( \lambda_1^{-\frac{1}{2}} \| \psi_{\lambda_1} \|_{L^4_{t,x}} \lambda_2^{-\frac{1}{2}} \| \varphi_{\lambda_2} \|_{L^4_{t,x}} \Big)^\delta \Big( \| \psi_{\lambda_1 } \|_{V^2_{\pm_1, M}} \| \varphi_{\lambda_2} \|_{V^2_{\pm_2, M}}\Big)^{1-\delta} \label{eqn:thm trilinear subcrit:case 2 A0}
	\end{align} for every
        $0<\delta<1$.
        Thus we have a high-low gain provided we place the functions
        $\phi_{\mu}$, $\psi_{\lambda_1}$, and $\varphi_{\lambda_2}$
        into the relevant $V^2$ space. On the other hand, letting
        $\alpha = (\frac{d \lambda_2}{\mu^2})^\frac{1}{2}$ and  applying
        the Klein-Gordon Strichartz estimate gives for
        $\frac{10}{3} \les r < 4$
        \begin{align*}
          \Big| \int_{\RR^{1+3}} A_0 dx dt \Big| &\lesa \beta\sum_{ \substack{ \kappa, \kappa' \in \mc{C}_\beta \\ |\pm_1 \kappa - \pm_2\kappa'| \lesa \beta}} \sum_{\substack{ \kappa'' \in \mc{C}_\alpha \\ |\pm_1 \kappa - \kappa''| \lesa \alpha }}  \| C_d R_{\kappa''} \phi_\mu \|_{L^{\frac{r}{r-2}}_{t,x}} \| R_{\kappa} \mc{C}_d^{\pm_1} \psi_{\lambda_1} \|_{L^r_{t,x}} \| R_{\kappa'} \mc{C}_d^{\pm_2} \varphi_{\lambda_2}\|_{L^r_{t,x}} \\
                                                 &\lesa \Big( \frac{d}{\lambda_2} \Big)^{\frac{4}{r}  -1 - 2\delta} \Big( \frac{\mu}{\lambda_2} \Big)^{1+\frac{4}{r}} \lambda_2^{\frac{8}{r} - 2} \big( \mu^{\frac{1}{2}} \| \phi_\mu \|_{V^2_{+, 1}} \big)^{3 - \frac{8}{r}} \| \phi_\mu \|_{L^4_{t,x}}^{\frac{8}{r} - 2} \| \psi_{\lambda_1} \|_{V^2_{\pm_1, M}} \| \varphi_{\lambda_2} \|_{V^2_{\pm_2, M}},
        \end{align*}
where we used that $\sup_{\kappa \in \mc{C}_\beta} \{ \kappa'' \in \mc{C}_\alpha : |\pm_1 \kappa - \kappa''| \lesa \alpha\}\lesa (\frac{\mu}{\lambda_2})^2$.
        Provided we choose $\delta>0$ sufficiently small, the above
        estimates and summation with respect to $\lambda^{-1}_2\lesa d\lesa \mu$ give an acceptable bound for the $A_0$ term.

The  argument to control the $A_1$ term is similar, we just reverse
        the roles of $\phi_\mu$ and $\psi_{\lambda_1}$, and note that,
        with $\alpha = (\frac{ d \lambda_2}{\mu^2} )^\frac{1}{2}$,
        \begin{align}
          \Big| \int_{\RR^{1+3}}& A_1 dx dt \Big| \notag \\
                                &\lesa\beta  \sum_{\substack{q, q'' \in Q_{\lambda_2} \\ |q-q''| \approx \lambda_2}} \sum_{\substack{\kappa, \kappa'' \in \mc{C}_\alpha \\ |\pm_1 \kappa - \kappa''| \lesa \alpha}} \sum_{\substack{ \kappa' \in \mc{C}_\beta \\ |\pm_1 \kappa - \pm_2\kappa'|\lesa \beta}}  \| P_{q''} R_{\kappa''} \phi_\mu \|_{L^4_{t,x}} \| R_\kappa P_q \mc{C}^{\pm_1}_d \psi_{\lambda_1} \|_{L^2_{t,x}} \| R_{\kappa'}  \varphi_{\lambda_2}\|_{L^4_{t,x}} \notag \\
                                &\lesa \Big( \frac{d}{\lambda_2}\Big)^{-4\delta} \Big( \frac{\lambda_2}{\mu} \Big)^{\frac{1}{4} - 4\delta}  \| \psi_{\lambda_1 } \|_{V^2_{\pm_1, M}} \Big( \|\phi_\mu \|_{L^4_{t,x}} \lambda_2^{-\frac{1}{2}} \| \varphi_{\lambda_2} \|_{L^4_{t,x}} \Big)^\delta \Big(\mu^\frac{1}{2} \| \phi_\mu \|_{V^2_{+,1}}\| \varphi_{\lambda_2} \|_{V^2_{\pm_2, M}}\Big)^{1-\delta}. \label{eqn:thm trilinear subcrit:case 2 A1}
	\end{align}
        As in the $A_0$ case, we can apply the Klein-Gordon Strichartz
        estimate to gain a positive factor of
        $\frac{d}{\lambda_2}$ as well as an $L^4_{t,x}$ factor of $\psi_{\lambda_1}$. Hence summing up over
        $\lambda_2^{-1} \lesa d \ll \lambda_2$ gives the claimed
        bound for the $A_1$ term.

        Finally to bound the $A_2$ component, we can either lose
        $\epsilon$ high derivatives, or avoid this loss by exploiting
        the $Y^{\pm, M}_{\lambda}$ type norms. More precisely, using
        the wave Strichartz pair $(\frac{2r}{r-1}, 2r)$ with $\frac{1}{r} = \frac{1}{a} - \frac{1}{2}(\frac{1}{a} - \frac{1}{2})$ and $a$ is as in the definition of the $Y^{\pm, M}_{\lambda}$ norm, we see that for any $\delta>0$ sufficiently small
	\begin{align}
          \Big| &\int_{\RR^{1+3}} A_2 dx dt \Big| \notag \\
		&\lesa  \beta\sum_{\substack{q, q''\in Q_{\lambda_2}\\ |q-q''| \approx \lambda_2}} \sum_{\substack{\kappa, \kappa'' \in \mc{C}_{\alpha} \\ |\pm_1\kappa - \kappa''| \lesa \alpha}} \sum_{\substack{\kappa' \in \mc{C}_{\beta} \\ |\pm_1 \kappa - \pm_2\kappa'| \lesa \beta}}\|  P_{q''} R_{\kappa''} C^+_{\lesa d} \phi_\mu \|_{L^{\frac{2r}{r-1}}_t L^{2r}_x} \| R_{\kappa} P_{q}\mc{C}_{\lesa d}^{\pm_1} \psi_{\lambda_1} \|_{L^\frac{2r}{r-1}_t L^{2r}_x} \| R_{\kappa'} \mc{C}_d^{\pm_2} \varphi_{\lambda_2} \|_{L^r_t L^{\frac{r}{r-1}}_x} \notag \\
		&\lesa \Big( \frac{d}{\lambda_2}\Big)^{ - b - 4\delta} \Big( \frac{\lambda_2}{\mu} \Big)^{\frac{1}{2a} - \frac{1}{4} - 4\delta}  \mb{A}^\delta  \Big( \mu^{\frac{1}{2}} \| \phi_\mu \|_{V^2_{+, 1}} \| \psi_{\lambda_1} \|_{V^2_{\pm_1, M}} \| \varphi \|_{Y^{\pm_2, M}_{\lambda_2}} \Big)^{1-\delta}.
                  \label{eqn:thm trilinear subcrit:case 2 A2}
	\end{align}
Here, we have used two estimates which require further explanation. First,
$L^r$ interpolation together with Bernstein's inequality implies

\[\| R_{\kappa'} \mc{C}_d^{\pm_2} \varphi_{\lambda_2} \|_{L^r_t L^{\frac{r}{r-1}}_x} \lesa  \Big( \frac{d}{\lambda_2} \Big)^{-\frac{1}{2} } \lambda_2^{\frac{2}{r}-\frac{3}{2}}\big(\lambda_2^{-\frac12}\| R_{\kappa'} \mc{C}_d^{\pm_2} \varphi_{\lambda_2} \|_{L^4_{t,x}}\big)^{\delta_0}\big(d^{-\frac{1}{a}} \| R_{\kappa'} \mc{C}_d^{\pm_2} \varphi_{\lambda_2} \|_{L^a_t L^{2}_x}\big)^{1-\delta_0}\]
with $\delta_0\in (0,1)$ satisfying $\frac{1}{r}=\frac{\delta_0}{4}+\frac{1-\delta_0}{a}$. Then, by writing $C_d^{\pm_2} \varphi_{\lambda_2}$ as a superposition of free waves, applying the Strichartz estimate and H\"older we obtain
\[
\lambda_2^{-\frac12}\| R_{\kappa'} \mc{C}_d^{\pm_2} \varphi_{\lambda_2} \|_{L^4_{t,x}}\lesa d^{\frac{1}{a}}\| R_{\kappa'} \mc{C}_d^{\pm_2} \varphi_{\lambda_2} \|_{L^a_t L^{2}_x}.
\]
We conclude that for all $0<\delta<\delta_0$
        $$\| R_{\kappa'} \mc{C}_d^{\pm_2}\varphi_{\lambda_2} \|_{L^r_t L^{\frac{r}{r-1}}_x} \lesa \Big( \frac{d}{\lambda_2} \Big)^{-\frac{1}{2} - b} \lambda_2^{\frac{1}{a} - 1} \Big( \lambda_2^{-\frac{1}{2}} \| \varphi_{\lambda_2} \|_{L^4_{t,x}} \Big)^\delta \| \varphi \|_{Y^{\pm_2, M}_{\lambda_2}}^{1-\delta},$$
since $R_{\kappa'}\mc{C}_d^{\pm_2}$ is uniformly disposable here.
The second estimate used above is the following: Since $\frac{5}{3}<a<2$ we have via
        \cite[Lemma 8.6]{Candy2016} for every $0<\delta <\frac{1}{2}$
	\begin{align*}
          \Big( \sum_{q \in Q_{\lambda_2}} \sum_{\kappa \in \mc{C}_{\alpha}} \| R_{\kappa} P_{q} \mc{C}_{\lesa d}^{\pm_1} \psi_{\lambda_1} \|_{L^{\frac{2r}{r-1}}_t L^{2r}_x}^2 \Big)^{\frac{1}{2}}&\lesa (\mu \lambda_2)^{\frac{1}{2} - \frac{1}{2r}} \Big( \sum_{q \in Q_{\lambda_2}} \sum_{\kappa \in \mc{C}_{\beta}} \big( \lambda_1^{-\frac{1}{2}} \| \psi_{\lambda_1} \|_{L^4_{t,x}}\big)^{2\delta} \| \psi_{\lambda_1} \|_{V^2_{\pm_1, M}}^{2-2\delta} \Big)^{\frac{1}{2}} \\
          &\lesa \Big( \frac{\lambda_2}{\mu} \Big)^{-2 \delta} \Big( \frac{d}{\lambda_2} \Big)^{-2\delta} ( \mu \lambda_2)^{\frac{3}{8}-\frac{1}{4a}}
          \Big( \lambda_1^{-\frac{1}{2}} \| \psi_{\lambda_1} \|_{L^4_{t,x}} \Big)^\delta \| \psi_{\lambda_1} \|_{V^2_{\pm_1, M}}^{1-\delta}
        \end{align*}
        together with a similar bound for the $\phi_{\mu}$ term. Thus
        summing up over $\lambda_2^{-1} \lesa d \lesa \lambda_2$ and
        choosing $\delta>0$ sufficiently small (depending on both $a$
        and $\varrho$), we get an acceptable bound for the $A_2$ term.

        We also require a bound for the $A_2$ component without using the $Y^{\pm, M}_{\lambda}$ norm. To this end, we note that for every
        $\delta>0$ we have the weaker bound
	\begin{align}
          \Big| \int_{\RR^{1+3}} & A_2 dx dt \Big| \notag\\
                                 &\lesa \beta\sum_{\substack{q, q''\in \mc{C}_{\lambda_2}\\ |q-q''| \approx \lambda_2}} \sum_{\substack{\kappa, \kappa'' \in \mc{C}_{\alpha} \\ |\pm_1\kappa - \kappa''| \lesa \alpha}} \sum_{\substack{\kappa' \in \mc{C}_{\beta} \\ |\pm_1 \kappa - \pm_2\kappa'| \lesa \beta}} \|   R_{\kappa''} P_{q''}C^+_{\lesa d} \phi_\mu \|_{L^4_{t,x}} \| R_{\kappa} P_{q}\mc{C}_{\lesa d}^{\pm_1} \psi_{\lambda_1} \|_{L^4_{t,x}} \| R_{\kappa'} \mc{C}_d^{\pm_2} \varphi_{\lambda_2} \|_{L^2_{t,x}} \notag\\
                                 &\lesa \Big( \frac{d}{\lambda_2} \Big)^{-4\delta} \Big( \frac{\lambda_2}{\mu}\Big)^{-4\delta}\Big( \| \phi_\mu \|_{L^4_{t,x}} \lambda_1^{-\frac{1}{2}} \|\psi_{\lambda_1} \|_{L^4_{t,x}} \Big)^\delta \Big( \mu^\frac{1}{2} \| \phi_\mu \|_{V^2_{+, 1}} \| \psi_{\lambda_1} \|_{V^2_{\pm_1, M}} \Big)^{1-\delta} \| \varphi_{\lambda_2} \|_{V^2_{\pm_2, M}}. \label{eqn:thm trilinear subcrit:case 2 A2 without Y}
	\end{align}
        To gain a power of the $L^4_{t,x}$ norm of $\varphi_{\lambda_2}$, we exploit the Klein-Gordon Strichartz estimates as previously which gives for $\frac{10}{3}\les r < 4$ and $\delta>0$
            \begin{align*}
              \Big| \int_{\RR^{1+3}} &A_2 dx dt \Big|\\
                &\lesa \beta \sum_{\substack{\kappa, \kappa'\in \mc{C}_\alpha \\ |\pm_1 \kappa - \pm_2  \kappa'| \lesa \alpha }} \sum_{\substack{\kappa' \in \mc{C}_{\beta} \\ |\pm_1 \kappa - \pm_2\kappa'| \lesa \beta}} \| C_{\lesa d} R_{\kappa''} \phi_\mu \|_{L^r_{t,x}} \| C_{\lesa d} R_{\kappa} \psi_{\lambda_1} \|_{L^r_{t,x}} \| C_d R_{\kappa'} \varphi_{\lambda_2} \|_{L^2_{t,x}}^{3-\frac{8}{r}} \| C_d R_{\kappa'} \varphi_{\lambda_2} \|_{L^4_{t,x}}^{\frac{8}{r} - 2}  \\
                &\lesa \Big( \frac{d}{\lambda_2} \Big)^{\frac{4}{r} -1 +\delta} \Big( \frac{\lambda_2}{\mu} \Big)^{-\frac{1}{2} - 2 \delta} \lambda_2^{\frac{8}{r} - 2} \mu^\frac{1}{2} \| \phi_\mu \|_{V^2_{+, 1}}\| \psi_{\lambda_1} \|_{V^2_{\pm_1, M}} \Big( \lambda_1^{-\frac{1}{2}} \| \varphi_{\lambda_2} \|_{L^4_{t,x}}\Big)^{\frac{8}{r} -2} \| \varphi_{\lambda_2} \|_{V^2_{\pm_2, M}}^{ 3 - \frac{8}{r}}.  \end{align*}
        Combining these bounds and summing up with respect to $\lambda_2^{-1}\lesa d\lesa \lambda_2$, we obtain an acceptable contribution for the $A_2$ term. \\

        \textbf{Case 3: $\mu \ll \lambda_1 \approx \lambda_2$ and
          $d \gg \mu$.} To bound the $A_0$ component, we first observe that $\int_{\RR^{1+3}} A_0 \not = 0$ implies $d \approx \mathfrak{M}_{\pm_1, \pm_2}$. In particular, since either $\mathfrak{M}_{\pm_1, \pm_2} \lesa \mu$ or $\mathfrak{M}_{\pm_1, \pm_2} \approx \lambda_1$,  the sum over the modulation is restricted to $d \gtrsim \lambda_1$. If now apply Theorem \ref{thm:cheap bilinear} with $\gamma =\frac{1}{8}$ we see that
        \begin{align*}
          \Big| \int_{\RR^{1+3}} A_0 dx dt \Big|&\lesa \| C^+_d \phi_\mu \|_{L^2_{t,x}} \big\| P_{\mu} \big[ \big(  \mc{C}_{\ll d}^{\pm_1} \psi_{\lambda_1} \big)^\dagger \gamma^0   \mc{C}_{\ll d}^{\pm_2} \varphi_{\lambda_2} \big] \big\|_{L^2_{t,x}}\\
                                                &\lesa \Big( \frac{d}{\lambda_1} \Big)^{-\frac{1}{2}} \Big( \frac{\mu}{\lambda_1} \Big)^{\frac{1}{4}} \mu^\frac{1}{2} \|\phi_\mu\|_{V^2_{+,1}} \Big( \lambda_1^{-\frac{1}{2}} \| \psi_{\lambda_1} \|_{L^4_{t,x}} \lambda_1^{-\frac{1}{2}} \| \varphi_{\lambda_2} \|_{L^4_{t,x}} \Big)^{\frac{1}{8}} \Big( \| \psi_{\lambda_1} \|_{V^2_{\pm_1, M}} \| \varphi_{\lambda_2} \|_{V^2_{\pm_2, M}}\Big)^{\frac{7}{8}}.
        \end{align*}
        On the other hand, using the Klein-Gordon Strichartz
        estimates, and noting that the $C_{\ll d}$ multipliers are now
        disposable, we have for $\frac{10}{3}\les r < 4$ the estimate
        \begin{align}
          \Big| \int_{\RR^{1+3}} A_0 dx dt \Big| &\les \| C_d \phi_\mu \|_{L^\frac{r}{r-2}_{t,x}} \| \psi_{\lambda_1} \|_{L^r_{t,x}} \|  \varphi_{\lambda_2} \|_{L^r_{t,x}} \\
                                                 &\lesa \Big( \frac{d}{\mu} \Big)^{\frac{4}{r}-1} \mu^{\frac{8}{r} - 2}  \| \phi_\mu \|_{L^4_{t,x}}^{\frac{8}{r}-2} \Big( \mu^\frac{1}{2} \| \phi_\mu \|_{V^2_{+,1}} \Big)^{3-\frac{8}{r}} \| \psi_{\lambda_1} \|_{V^2_{\pm_1, M}} \| \varphi_{\lambda_2}\|_{V^2_{\pm_2, M}}.\label{eqn:proof of thm tri sub:c3 kg bound A0}
	\end{align}
        In particular, interpolating between these bounds and summing
        up over $d \gtrsim \lambda_1$ gives an acceptable contribution
        for the $A_0$ term.

        The bound for the $A_1$ term is slightly different as we no longer have $\mathfrak{M}_{\pm_1, \pm_2} \approx d$, and thus have to consider the full region $d \gg \mu$. We first deal with the region $d \gtrsim \lambda_1$. Here we argue as usual by controlling the integral by $L^4 \times L^2 \times L^4$, which gives for every $0< \delta\les 1$
	\begin{align}
          \Big| \int_{\RR^{1+3}} A_1dx dt \Big| &\lesa \| \phi_\mu \|_{L^4_{t,x}} \sum_{\substack{q, q' \in Q_\mu \\ |q - q'| \lesa \mu}} \| P_q C_d \psi_{\lambda_1} \|_{L^2_{t,x}} \| P_{q'} \varphi_{\lambda_2} \|_{L^4_{t,x}} \notag \\
                                           &\lesa \Big( \frac{d}{\lambda_1} \Big)^{-\frac{1}{2}} \Big( \frac{\mu}{\lambda_1} \Big)^{\frac{1}{4} - 2\delta} \| \phi_\mu \|_{L^4_{t,x}} \| \psi_{\lambda_2} \|_{V^2_{\pm_1, M}} \Big( \lambda_1^{-\frac{1}{2}} \| \varphi_{\lambda_2} \|_{L^4_{t,x}} \Big)^\delta \| \varphi_{\lambda_2} \|_{V^2_{\pm_2, M}}^\delta \label{eqn:proof of thm tri sub:case 3 pm}
	\end{align}
        where we controlled the square sum as previously via an estimate analogous to
        (\ref{eqn:thm trilinear subcrit:example of square sum control}). To gain an $L^4_{t,x}$ norm of $\psi$, we again exploit the Klein-Gordon Strichartz estimate and observe that for $\frac{10}{3} \les r <4$ we have
    \begin{align}
              \Big| \int_{\RR^{1+3}} A_1 dx dt \Big|
                &\lesa\| \phi_\mu \|_{L^r_{t,x}} \| C_d \psi_{\lambda_1} \|_{L^2_{t,x}}^{3-\frac{8}{r}} \| C_d R_\kappa \psi_{\lambda_1} \|_{L^4_{t,x}}^{\frac{8}{r} - 2} \| \varphi_{\lambda_2} \|_{L^r_{t,x}} \notag \\
                &\lesa \Big( \frac{d}{\lambda_1} \Big)^{\frac{4}{r} -\frac{3}{2}}  \lambda_1^{\frac{8}{r}-2} \mu^\frac{1}{2} \| \phi_\mu \|_{V^2_{+, 1}}\| \varphi_{\lambda_2} \|_{V^2_{\pm_2, M}} \Big( \lambda_1^{-\frac{1}{2}} \| \psi_{\lambda_1} \|_{L^4_{t,x}}\Big)^{\frac{8}{r} -2} \| \psi_{\lambda_1} \|_{V^2_{\pm_1, M}}^{ 3 - \frac{8}{r}}.
                \label{eqn:proof of thm tri sub:c3 kg A1}
            \end{align}
    Combining the bounds (\ref{eqn:proof of thm tri sub:case 3 pm}) and \eqref{eqn:proof of thm tri sub:c3 kg A1}, and summing up over modulation, we deduce the required bound for the $A_1$ in the region $d \gtrsim \lambda_1$.

    We now consider the case $\mu \ll d \ll \lambda_1$. This implies that $\mathfrak{M}_{\pm_1, \pm_2} \ll \lambda_1$, which is only possible if $\pm_1 = \pm_2$. The key point is that we may now exploit the null structure in the product of the spinors $\psi$ and $\varphi$, since
    we gain $\theta(\xi, \eta)$, and the angle between the supports of $\widehat{\psi}$ and $\widehat{\phi}$ is less than $\frac{\mu}{\lambda_1}$. In particular, exploiting the standard null structure bound implies that we may improve   \eqref{eqn:proof of thm tri sub:case 3 pm} to
            \begin{align}
          \Big| \int_{\RR^{1+3}} A_1dx dt \Big|
                                                     &\lesa \Big( \frac{d}{\lambda_1} \Big)^{-\frac{1}{2}} \Big( \frac{\mu}{\lambda_1} \Big)^{\frac{5}{4} - 2\delta} \| \phi_\mu \|_{L^4_{t,x}} \| \psi_{\lambda_2} \|_{V^2_{\pm_1, M}} \Big( \lambda_1^{-\frac{1}{2}} \| \varphi_{\lambda_2} \|_{L^4_{t,x}} \Big)^\delta \| \varphi_{\lambda_2} \|_{V^2_{\pm_2, M}}^\delta
                                                     \label{eqn:proof of thm tri sub:c3 A1 med mod}
	        \end{align}
    Again combining this bound with \eqref{eqn:proof of thm tri sub:c3 kg A1}, and summing up over $\mu \ll d \ll \lambda_1$, the required bound for the $A_1$ term follows. The $A_2$ term follows from an identical argument. \\

        \textbf{Case 4: $\mu \gtrsim \min\{\lambda_1, \lambda_2\}$ and $d \gtrsim \min\{\lambda_1, \lambda_2\}$.} It is enough to consider the case        $\lambda_1 \g \lambda_2$. To estimate the $A_0$ term, we first observe that as in the previous case, we may restrict the sum over modulation to $d \gtrsim  \mu$. If we now  observe that
	\begin{align}
          \Big| \int_{\RR^{1+3}} A_0 dx dt \Big| &\lesa \| C_d^+ \phi_\mu \|_{L^2_{t,x}} \|\psi_{\lambda_1}\|_{L^4_{t,x}} \| \varphi_{\lambda_2} \|_{L^4_{t,x}} \notag \\
                                                 &\lesa \Big( \frac{d}{\mu} \Big)^{-\frac{1}{2}} \Big( \frac{\lambda_2}{\mu} \Big)^{\frac{1}{2}} \mu^{\frac{1}{2}} \| \phi_\mu \|_{V^2_{+,1}} \lambda_1^{-\frac{1}{2}} \| \psi_{\lambda_1} \|_{L^4_{t,x}} \lambda_2^{-\frac{1}{2}} \| \varphi_{\lambda_2} \|_{L^4_{t,x}}
                                                 \label{eqn:proof of thm tri sub:c4 A0}
	\end{align}
    then, together with \eqref{eqn:proof of thm tri sub:c3 kg bound A0}, summing up over $d\gtrsim \mu$ gives the required bound for the $A_0$ term.

    If we suppose that $d \gtrsim \mu$, then a similar argument handles  the $A_1$ term. Again supposing that $d \gtrsim \mu$, to bound the $A_2$ term, we decompose into cubes of diameter $\mu$ to obtain
	\begin{align}
          \Big| \int_{\RR^{1+3}}& A_2 dx dt \Big| \notag \\&\lesa \sum_{\substack{ q, q' \in Q_{\lambda_2} \\ |q-q''| \approx \mu}} \| P_{q''} \phi_\mu \|_{L^4_{t,x}} \| P_q \psi_{\lambda_1} \|_{L^4_{t,x}} \| \mc{C}_d^{\pm_2} \varphi_{\lambda_2} \|_{L^2_{t,x}} \notag \\
                                                 &\lesa \Big( \frac{d}{\mu} \Big)^{-\frac{1}{2}} \Big( \frac{\lambda_2} {\mu}\Big)^{\frac{1}{2} - 4 \delta } \Big( \| \phi_\mu \|_{L^4_{t,x}} \lambda_1^{-\frac{1}{2}} \|\psi_{\lambda_1} \|_{L^4_{t,x}} \Big)^\delta \Big( \mu^{\frac{1}{2}} \|\phi_\mu\|_{V^2_{+,1}} \|\psi_{\lambda_1} \|_{V^2_{\pm_1, M}} \Big)^{1-\delta} \| \varphi_{\lambda_2} \|_{V^2_{\pm_2, M}}.
                                                 \label{eqn:proof of thm tri sub:c4 A2}
	\end{align}
    An analogous estimate to \eqref{eqn:proof of thm tri sub:c3 kg A1} then gives an additional $L^4_{t,x}$ norm of $\varphi_{\lambda_2}$. This gives an acceptable bound when $d \gtrsim \mu$.

    It remains to bound the $A_1$ and $A_2$ terms when $\lambda_2 \ll d \ll \mu$. We first recall that either $\mathfrak{M}_{\pm_1, \pm_2}\approx \lambda_1$ or $\mathfrak{M}_{\pm_1, \pm_2} \lesa \lambda_2$. Hence the restriction $\lambda_2 \ll d \ll \mu$ implies that
    $\mathfrak{M}_{\pm_1, \pm_2} \ll d$ and consequently, a short computation shows that at least \emph{two} of the functions $\phi$, $\psi$, and $\varphi$ must have large modulation. More precisely, we have the decomposition
         \begin{equation}\label{eqn:proof of thm tri sub:c4 A1 mod decomp}
            \int_{\RR^{1+3}} A_1 dx dt
            = \int_{\RR^{1+3}} C_{\approx d} \phi_{\mu} \big( \mc{C}_{ d}^{\pm_1} \psi_{\lambda_1}\big)^\dagger \gamma^0 \mc{C}_{\lesa d}^{\pm_2} \varphi_{\lambda_2} dx dt + \int_{\RR^{1+3}} C_{\ll d} \phi_{\mu} \big( \mc{C}_{ d}^{\pm_1} \psi_{\lambda_1}\big)^\dagger \gamma^0 \mc{C}_{\approx d}^{\pm_2} \varphi_{\lambda_2} dx dt.\end{equation}
    To bound the first term in \eqref{eqn:proof of thm tri sub:c4 A1 mod decomp}, we observe that
        \begin{align*}
          \Big| \int_{\RR^{1+3}} C_{\approx d} \phi_{\mu} \big( \mc{C}_{ d}^{\pm_1} \psi_{\lambda_1}\big)^\dagger \gamma^0 \mc{C}_{\lesa d}^{\pm_2} \varphi_{\lambda_2} dx dt\Big|
                &\lesa \| C_{\approx d} \phi_{\mu}\|_{L^2_{t,x}} \| \mc{C}_{ d}^{\pm_1} \psi_{\lambda_1} \|_{L^2_{t,x}} \| \mc{C}_{\lesa d}^{\pm_2} \varphi_{\lambda_2}\|_{L^\infty_{t,x}} \\
                &\lesa \Big( \frac{d}{\lambda_2}\Big)^{-1}  \Big( \frac{\lambda_2}{\mu} \Big)^\frac{1}{2} \mu^\frac{1}{2} \| \phi_\mu \|_{V^2_{+, 1}} \| \psi_{\lambda_1} \|_{V^2_{\pm_1, M}} \| \varphi_{\lambda_2} \|_{V^2_{\pm_2, M}}.
        \end{align*}
    On the other hand, to bound the second term in \eqref{eqn:proof of thm tri sub:c4 A1 mod decomp}, we decompose into cubes of size $\lambda_2$ and apply Bernstein's inequality which gives for every $\epsilon>0$
        \begin{align*}
          \Big| \int_{\RR^{1+3}} C_{\ll d} \phi_{\mu} \big( \mc{C}_{ d}^{\pm_1} \psi_{\lambda_1}\big)^\dagger \gamma^0 \mc{C}_{\lesa d}^{\pm_2} \varphi_{\lambda_2} dx dt\Big|
                &\lesa \sum_{\substack{q, q' \in \mc{Q}_{\lambda_2} \\ |q-q'| \lesa \lambda_2}} \| C_{\ll d} P_{q'} \phi_{\mu}\|_{L^\infty_{t,x}} \| P_q \mc{C}_{ d}^{\pm_1} \psi_{\lambda_1} \|_{L^2_{t,x}} \| \mc{C}_{\approx d}^{\pm_2} \varphi_{\lambda_2}\|_{L^2_{t,x}} \\
                &\lesa \Big( \frac{d}{\lambda_2}\Big)^{-1}  \Big( \frac{\lambda_2}{\mu} \Big)^{\frac{1}{2}-\epsilon} \mu^\frac{1}{2} \| \phi_\mu \|_{V^2_{+, 1}} \| \psi_{\lambda_1} \|_{V^2_{\pm_1, M}} \| \varphi_{\lambda_2} \|_{V^2_{\pm_2, M}}.
        \end{align*}
    Thus we deduce that for $\lambda_2 \ll d \ll \mu$ and $\mu \approx \lambda_1 \gtrsim \lambda_2$ we have
        \begin{equation}\label{eqn:proof of thm tri sub:c4 A1 restricted d}
            \Big| \int_{\RR^{1+3}} A_1 dx dt \Big| \lesa \Big( \frac{d}{\lambda_2}\Big)^{-1}  \Big( \frac{\lambda_2}{\mu} \Big)^{\frac{1}{4}} \mu^\frac{1}{2} \| \phi_\mu \|_{V^2_{+, 1}} \| \psi_{\lambda_1} \|_{V^2_{\pm_1, M}} \| \varphi_{\lambda_2} \|_{V^2_{\pm_2, M}}.
        \end{equation}
     To gain powers of the $L^4_{t,x}$ norm, we simply use an analogous bound to \eqref{eqn:proof of thm tri sub:c4 A0} and \eqref{eqn:proof of thm tri sub:c3 kg bound A0}. Thus summing up over $\lambda_2 \ll d \ll \mu$ then gives the required bound for the $A_1$ term.

     We now consider the $A_2$ term in the region $\lambda_2 \ll d \ll \mu$. As in the argument for the $A_1$ term, we have the decomposition
        \begin{equation}\label{eqn:proof of thm tri sub:c4 A2 mod decomp}
            \int_{\RR^{1+3}} A_2 dx dt
            = \int_{\RR^{1+3}} C_{\approx d} \phi_{\mu} \big( \mc{C}_{ \lesa d}^{\pm_1} \psi_{\lambda_1}\big)^\dagger \gamma^0 \mc{C}_{d}^{\pm_2} \varphi_{\lambda_2} dx dt + \int_{\RR^{1+3}} C_{\ll d} \phi_{\mu} \big( \mc{C}_{\approx  d}^{\pm_1} \psi_{\lambda_1}\big)^\dagger \gamma^0 \mc{C}_{ d}^{\pm_2} \varphi_{\lambda_2} dx dt.\end{equation}
     The first term in \eqref{eqn:proof of thm tri sub:c4 A2 mod decomp} can be handled in an analogous manner to the second term in \eqref{eqn:proof of thm tri sub:c4 A1 mod decomp}. Namely, decomposing into cubes and applying Bernstein's inequality gives for every $\epsilon>0$
        \begin{align*}
          \Big| \int_{\RR^{1+3}} C_{\approx d} \phi_{\mu} \big( \mc{C}_{ \lesa d}^{\pm_1} \psi_{\lambda_1}\big)^\dagger \gamma^0 \mc{C}_{ d}^{\pm_2} \varphi_{\lambda_2} dx dt\Big|
                &\lesa \sum_{\substack{q, q' \in \mc{Q}_{\lambda_2} \\ |q-q'| \lesa \lambda_2}} \| C_{\approx  d} P_{q'} \phi_{\mu}\|_{L^2_{t,x}} \| P_q \mc{C}_{ d}^{\pm_1} \psi_{\lambda_1} \|_{L^\infty_{t,x}} \| \mc{C}_{\approx d}^{\pm_2} \varphi_{\lambda_2}\|_{L^2_{t,x}} \\
                &\lesa \Big( \frac{d}{\lambda_2}\Big)^{-1}  \Big( \frac{\lambda_2}{\mu} \Big)^{\frac{1}{2}-\epsilon} \mu^\frac{1}{2} \| \phi_\mu \|_{V^2_{+, 1}} \| \psi_{\lambda_1} \|_{V^2_{\pm_1, M}} \| \varphi_{\lambda_2} \|_{V^2_{\pm_2, M}}.
        \end{align*}
     Applying an identical argument to the second term in \eqref{eqn:proof of thm tri sub:c4 A2 mod decomp}, we deduce that
        \begin{equation}\label{eqn:proof of thm tri sub:c4 A2 restricted d}
            \Big| \int_{\RR^{1+3}} A_2 dx dt \Big| \lesa \Big( \frac{d}{\lambda_2}\Big)^{-1}  \Big( \frac{\lambda_2}{\mu} \Big)^{\frac{1}{4}} \mu^\frac{1}{2} \| \phi_\mu \|_{V^2_{+, 1}} \| \psi_{\lambda_1} \|_{V^2_{\pm_1, M}} \| \varphi_{\lambda_2} \|_{V^2_{\pm_2, M}}.
        \end{equation}
     Together with the standard bound \eqref{eqn:proof of thm tri sub:c4 A2}, and the $A_2$ version of \eqref{eqn:proof of thm tri sub:c3 kg A1}, after summing up over $\lambda_2 \ll d \ll \mu$ we deduce the final bound required for the $A_2$ component.
\end{proof}

\subsection{Proof of Theorem \ref{thm:duhamel-sub}}\label{subsec:proof-duhamel-sub}
The first step is to obtain  frequency localised versions of the
required bounds. Namely, if $\varrho>0$ is sufficiently small and we take $\frac{1}{a}=\frac{1}{2} + \frac{\varrho}{32}$ and $b=4(\frac{1}{a}-\frac{1}{2})$, our aim is to show
there exists $0<\theta_1<\frac{1}{4}$ such that for all $0\les \theta \les \theta_1$ we have for the Dirac Duhamel term, the bounds
\begin{equation}\label{eqn:proof of thm duhamel subcrit:psi bd V2}\begin{split}
    \big\|  \Pi_{\pm_1} P_{\lambda_1} &\mc{I}^{\pm_1, M}\big( \phi_{\mu} \gamma^0 \Pi_{\pm_2} \varphi_{\lambda_2} \big) \big\|_{V^2_{\pm_1, M}} \\
        &\lesa  (\min\{\mu, \lambda_2\})^{\varrho} \Big( \frac{\min\{\mu, \lambda_2\}}{\max\{\mu, \lambda_2\}}\Big)^{\frac{\varrho}{100}} \Big( \|
        \phi_\mu\|_{L^4_{t,x}} \lambda_2^{-\frac{1}{2}}\| \varphi_{\lambda_2} \|_{L^4_{t,x}} \Big)^{\theta} \Big(  \mu^{\frac{1}{2}} \|
        \phi_\mu \|_{V^2_{+,1}} \| \varphi \|_{F^{\pm_2,M}_{\lambda_2}} \Big)^{1-\theta}
  \end{split}
\end{equation}
and
    \begin{equation}\label{eqn:proof of thm duhamel subcrit:psi bd Y}
        \begin{split}
          \big\|  \Pi_{\pm_1}  &\mc{I}^{\pm_1, M}\big( \phi_{\mu} \gamma^0 \Pi_{\pm_2} \varphi_{\lambda_2} \big) \big\|_{Y^{\pm_1, M}_{\lambda_1}} \\
    &\lesa (\min\{\mu,\lambda_2\})^{\varrho} \Big( \frac{\min\{
      \mu, \lambda_2\}}{\max\{\mu, \lambda_2\}}\Big)^{\frac{\varrho}{100}} \Big( \|
    \phi_\mu\|_{L^4_{t,x}} \lambda_2^{-\frac{1}{2}}\| \varphi_{\lambda_2} \|_{L^4_{t,x}} \Big)^{\theta} \Big(  \mu^{\frac{1}{2}} \|
    \phi_\mu \|_{V^2_{+,1}} \| \varphi_{\lambda_2} \|_{V^2_{\pm_2,
        M}} \Big)^{1-\theta},
        \end{split}
    \end{equation}
while for the wave Duhamel term, we have
\begin{equation}\label{eqn:proof of thm duhamel subcrit:phi bd V2}\begin{split}
    &\mu^{-\frac{1}{2}} \big\|  P_{\mu} \mc{I}^{+,1}\big( \overline{\Pi_{\pm_1} \psi_{\lambda_1}} \Pi_{\pm_2}\varphi_{\lambda_2}\big) \big\|_{V^2_{+,1}} \\
    &\lesa  (\min\{\mu, \lambda_1, \lambda_2\})^{\varrho} \Big(
    \frac{\min\{ \mu, \lambda_1, \lambda_2\}}{\max\{\mu, \lambda_1,
      \lambda_2\}}\Big)^{\frac{\varrho}{100}} \Big( \lambda_1^{-\frac{1}{2}} \|
    \psi_{\lambda_1} \|_{L^4_{t,x}} \lambda_2^{-\frac{1}{2}} \| \varphi_{\lambda_2} \|_{L^4_{t,x}} \Big)^{\theta} \Big(  \| \psi
    \|_{F^{\pm_1, M}_{\lambda_1}} \| \varphi \|_{F^{\pm_2,
        M}_{\lambda_2}} \Big)^{1-\theta}.
  \end{split}
\end{equation}
Assuming the bounds \eqref{eqn:proof of thm duhamel subcrit:psi bd V2}, \eqref{eqn:proof of thm duhamel subcrit:psi bd Y}, and \eqref{eqn:proof of thm duhamel subcrit:psi bd Y} for the moment, the estimates in Theorem \ref{thm:duhamel-sub} are a consequence of a straightforward summation argument. More precisely, fix $s_0>0$ sufficiently small. We have
        \begin{align*}
         & \big\|  \Pi_{\pm_1} \mc{I}^{\pm_1, M}\big( \phi \gamma^0 \Pi_{\pm_2} \varphi_{\lambda_2} \big) \big\|_{\mb{V}_{\pm_1, M}^{s_0}}\\
                           \lesa{} & \sum_{\lambda_1} \lambda_1^{s_0} \bigg(\sum_{\substack{\mu, \lambda_2 \\  \mu \approx \lambda_2 \gg \lambda_1}} \big\|  P_{\lambda_1}\Pi_{\pm_1} \mc{I}^{\pm_1, M}\big( \phi_\mu \gamma^0 \Pi_{\pm_2} \varphi_{\lambda_2} \big) \big\|_{V^2_{\pm_1, M}} \\
&\qquad + \sum_{\substack{\mu, \lambda_2 \\  \lambda_1 \approx \max\{\mu, \lambda_2\}}}\big\|  P_{\lambda_1}\Pi_{\pm_1} \mc{I}^{\pm_1, M}\big( \phi_\mu \gamma^0 \Pi_{\pm_2} \varphi_{\lambda_2} \big) \big\|_{{V^2_{\pm_1, M}}}\bigg)
        \end{align*}
 An application of \eqref{eqn:proof of thm duhamel subcrit:psi bd V2} with $\varrho = \frac{100}{101} s_0$ gives $0<\theta_0<\frac{1}{2}$ such that for all $0<\theta<\min\{\theta_0, \frac{1}{202}\}$ we have
\begin{align*}
&\big\|  \Pi_{\pm_1} \mc{I}^{\pm_1, M}\big( \phi \gamma^0 \Pi_{\pm_2} \varphi_{\lambda_2} \big) \big\|_{\mb{V}_{\pm_1, M}^{s_0}}\\
&\lesa\bigg( \sup_{\mu, \lambda_2} \Big(\|\phi_\mu \|_{L^4_{t,x}} \lambda_2^{-\frac{1}{2}} \| \varphi_{\lambda_2}\|_{L^4_{t,x}}\Big)^\theta \Big( \mu^{\frac{1}{2}+s_0} \| \phi_\mu\|_{V^2_{+,1}} \lambda_2^{s_0} \| \varphi \|_{F^{\pm_2, M}_{\lambda_2}} \Big)^{(1-\theta)}\bigg) \\
                             & \qquad \cdot
                                                          \bigg( \sum_{\lambda_1} \sum_{\mu \gg \lambda_1} \lambda_1^{-(\frac{1}{101} - 2\theta)s_0} \Big(\frac{\mu}{\lambda_1}\Big)^{-(\frac{102}{101}-2\theta)s_0} +\sum_{\lambda_1} \sum_{\mu \lesa \lambda_1} \lambda_1^{-(\frac{1}{101} - 2\theta)s_0} \Big( \frac{\mu}{\lambda_1} \Big)^{\theta s_0} \bigg)\\
                             &\lesa \sup_{\mu, \lambda_2} \Big(\|\phi_\mu \|_{L^4_{t,x}} \lambda_2^{-\frac{1}{2}} \| \varphi_{\lambda_2}\|_{L^4_{t,x}}\Big)^\theta \Big( \mu^{\frac{1}{2}+s_0} \| \phi_\mu\|_{V^2_{+,1}} \lambda_2^{s_0} \| \varphi \|_{F^{\pm_2, M}_{\lambda_2}} \Big)^{(1-\theta)}.
        \end{align*}
Thus we obtain \eqref{eqn:thm duhamel sub:V psi} in the case $s=s_0$. The general case $s>s_0$ follows by using the fact that $\lambda_1^s \lesa \lambda_1^{s_0} (\max\{ \mu, \lambda_2 \})^{s-s_0}$. An identical argument using (\ref{eqn:proof of thm duhamel subcrit:psi bd Y}) gives the $\mb{Y}^s_{\pm_1, M}$ bound (\ref{eqn:thm duhamel sub:Y psi}). Similarly the bound \eqref{eqn:thm duhamel sub:V phi} follows from (\ref{eqn:proof of thm duhamel subcrit:phi bd V2}).

We now turn to the proof of the estimates \eqref{eqn:proof of thm duhamel subcrit:psi bd V2}, \eqref{eqn:proof of thm duhamel subcrit:psi bd Y}, and \eqref{eqn:proof of thm duhamel subcrit:phi bd V2}. It is enough to consider the case $\theta=\theta_1$, as the $L^4_{t,x}$ terms are dominated by the corresponding $V^2$ norms.  The bounds (\ref{eqn:proof of thm duhamel subcrit:psi bd V2}) and (\ref{eqn:proof of thm duhamel subcrit:phi bd V2}) follow directly from Theorem \ref{thm:trilinear freq loc subcrit} together with \eqref{eqn:energy ineq for V2}. On the other hand, the argument used to obtain (\ref{eqn:proof of thm duhamel subcrit:psi bd Y}) is slightly more involved. We start by considering the case $\mu \lesa \lambda_2$. An application of the Klein-Gordon Strichartz estimate gives for every $\epsilon>0$
    \begin{align*}
     d^{\frac{3}{5}} \Big( \frac{\min\{d,\lambda_1\}}{\lambda_1} \Big)^{1-\frac{3}{5}} \big\| P_{\lambda_1} \mc{C}^{\pm_1, M}_d \mc{I}^{\pm_1, M}\big( \phi_{\mu} \gamma^0 \Pi_{\pm_2} \varphi_{\lambda_2} \big) \big\|_{L^\frac{5}{3}_t L^2_x}
                                                                          &\lesa  \lambda_1^{-\frac{2}{5}} \Big\| \Big( \sum_{q \in Q_{\mu} }\|\phi_\mu P_q\varphi_{\lambda_2} \|_{L^2_x}^2\Big)^\frac{1}{2}\Big\|_{L^\frac{5}{3}_t} \notag \\
                                                                          &\lesa \lambda_1^{-\frac{2}{5}}  \| \phi_\mu \|_{L^{\frac{10}{3}}_{t,x}} \Big( \sum_{q\in Q_{\mu}} \| P_q \varphi_{\lambda_2} \|_{L^{\frac{10}{3}}_t L^5_x}^2 \Big)^\frac{1}{2}  \notag \\
                                                                          &\lesa \lambda_1^{-\frac{2}{5}}  ( \mu \lambda_2 )^{\frac{3}{10}} \Big( \frac{\mu}{\lambda_2} \Big)^{-\epsilon} \mu^{\frac{1}{2}} \| \phi_\mu \|_{V^2_{+, 1}} \| \varphi_{\lambda_2} \|_{V^2_{\pm_2, M}}.
    \end{align*}
As we may assume that $\max\{\mu, \lambda_1\} \approx \lambda_2$, by choosing $\epsilon>0$ small, we deduce that
    \begin{equation}\label{eqn-proof of F control subcrit-wave low Lp}
        d^{\frac{3}{5}} \Big( \frac{\min\{d,\lambda_1\}}{\lambda_1} \Big)^{1-\frac{3}{5}}  \big\| P_{\lambda_1} \mc{C}^{\pm_1, M}_d \mc{I}^{\pm_1, M}\big( \phi_{\mu} \gamma^0 \Pi_{\pm_2} \varphi_{\lambda_2} \big) \big\|_{L^\frac{5}{3}_t L^2_x}
                \lesa \mu^{\frac{1}{5}}  \Big( \frac{\mu}{\lambda_2} \Big)^{\frac{1}{11}} \mu^{\frac{1}{2}} \| \phi_\mu \|_{V^2_{+, 1}} \| \varphi_{\lambda_2} \|_{V^2_{\pm_2, M}}.
    \end{equation}
On the other hand, an application of \eqref{eqn:tri freq loc subcrit:V2 main} in Theorem \ref{thm:trilinear freq loc subcrit} gives
\begin{align}
  d^{\frac{1}{2}} \big\| P_{\lambda_1} \mc{C}^{\pm_1, M}_d &\mc{I}^{\pm_1, M}\big( \phi_{\mu} \gamma^0 \Pi_{\pm_2} \varphi_{\lambda_2} \big) \big\|_{L^2_{t,x}} \notag \\
                                                           &\lesa \big\| P_{\lambda_1} \Pi_{\pm_1} \mc{I}^{\pm_1, M}\big( \phi_{\mu} \gamma^0 \Pi_{\pm_2} \varphi_{\lambda_2} \big) \big\|_{V^2_{\pm_1, M}} \notag \\
                                                           &\lesa{}\mu^{\frac{\varrho}{2}} \Big(
                                                             \frac{\mu}{\lambda_2}\Big)^{\frac{1}{10}} \Big( \| \phi_\mu\|_{L^4_{t,x}} \lambda_2^{-\frac{1}{2}} \| \varphi_{\lambda_2} \|_{L^4_{t,x}} \Big)^{\theta_0}
                                                             \Big( \mu^{\frac{1}{2}} \| \phi_\mu \|_{V^2_{+,1}} \| \varphi_{\lambda_2} \|_{V^2_{\pm_2, M}}
                                                             \Big)^{1-\theta_0}.\label{eqn-proof of F control subcrit-L2 bound1}
\end{align}
Hence (\ref{eqn:proof of thm duhamel subcrit:psi bd Y}) in the region $\mu \lesa \lambda_2$ follows by interpolating between (\ref{eqn-proof of F control subcrit-wave low Lp}) and (\ref{eqn-proof of F control subcrit-L2 bound1}) and using the condition $\frac{1}{a} = \frac{1}{2} + \frac{\varrho}{32}$.

We now consider the case $\mu \gg \lambda_2$. For this frequency interaction, Theorem \ref{thm:trilinear freq loc subcrit} requires a $Y^{\pm, M}_{\lambda_2}$ norm on the righthand side. Thus, as our goal is to obtain a bound only using the $V^2_{\pm, M}$ norms, we have to work a little harder. We start by writing the product as
     \begin{equation}\label{eqn-proof of F control subcrit-Lp bound wave high}\phi_\mu \gamma^0 \Pi_{\pm_2} \varphi_{\lambda_2} = \Big(\phi_\mu \gamma^0 \Pi_{\pm_2} \varphi_{\lambda_2} - \sum_{d' \lesa \lambda_2} C_{\les d'}^{\pm_1, M} \big( C_{\les d'} \phi_{\mu} \gamma^0 \mc{C}^{\pm_2}_{d'} \varphi_{\lambda_2} \big) \Big)+ \sum_{d' \lesa \lambda_2} C_{\les d'}^{\pm_1, M} \big( C_{\les d'} \phi_{\mu} \gamma^0 \mc{C}^{\pm_2}_{d'} \varphi_{\lambda_2} \big).\end{equation}
The first term can be bounded by adapting the argument used in the case $\mu \lesa \lambda_2$ as here (\ref{eqn:tri freq loc subcrit:V2-A0}) in  Theorem \ref{thm:trilinear freq loc subcrit} gives a bound without using the $Y^{\pm, M}_{\lambda_2} $ norm. More precisely, letting $\beta = (\frac{d'}{\lambda_2})^\frac{1}{2}$ and exploiting null structure, we have
 \begin{align*}
     d^{\frac{3}{5}} \Big( \frac{d}{\lambda_1} \Big)^{1-\frac{3}{5}} \Big\| P_{\lambda_1} \mc{C}^{\pm_1}_d \mc{I}^{\pm_1, M}\Big( \sum_{d' \lesa \lambda_2} C_{\les d'}^{\pm_1, M} &\big( C_{\les d'} \phi_{\mu} \gamma^0 \mc{C}^{\pm_2}_{d'} \varphi_{\lambda_2} \big)\Big) \Big\|_{L^\frac{5}{3}_t L^2_x}\\
                                                                          &\lesa  \mu^{-\frac{2}{5}} \sum_{d'\lesa \lambda_2}  \Big\| \Big( \sum_{\substack{\kappa, \kappa'  \in \mc{C}_{\beta}\\|\kappa - \kappa'|\lesa \beta} }\big\| R_{\kappa}\Pi_{\pm_1} \big(C_{\les d'}\phi_\mu \gamma^0 R_{\kappa'} \mc{C}^{\pm_2}_{d'} \varphi_{\lambda_2}\big) \big\|_{L^2_x}^2\Big)^\frac{1}{2}
                                                                          \Big\|_{L^\frac{5}{3}_t} \notag \\
                                                                          &\lesa \mu^{-\frac{2}{5}} \sum_{d' \lesa \lambda_2} \beta \|  C_{\les d'} \phi_\mu \|_{L^{\frac{10}{3}}_{t,x}} \Big( \sum_{\kappa'  \in \mc{C}_{\beta} } \| R_{\kappa'} \mc{C}_{d'}^{\pm_2} \varphi_{\lambda_2} \|_{L^{\frac{10}{3}}_t L^5_x}^2 \Big)^\frac{1}{2}  \notag \\
                                                                          &\lesa \lambda_2^{\frac{1}{5}}  \Big( \frac{\lambda_2}{\mu} \Big)^{\frac{2}{5}} \mu^{\frac{1}{2}} \| \phi_\mu \|_{V^2_{+, 1}} \| \varphi_{\lambda_2} \|_{V^2_{\pm_2, M}}.
    \end{align*}
 Consequently, applying a similar argument to the $\phi_\mu \gamma^0  \Pi_{\pm_2 }\varphi_{\lambda_2}$ component, we deduce that
    \begin{align*} d^{\frac{3}{5}} \Big( \frac{d}{\lambda_1} \Big)^{1-\frac{3}{5}} \Big\|P_{\lambda_1} \mc{C}^{\pm_1}_d \mc{I}^{\pm_1, M}\Big(\phi_\mu \gamma^0 \Pi_{\pm_2 }\varphi_{\lambda_2} -  \sum_{d' \lesa \lambda_2} C_{\les d'}^{\pm_1, M} &\big( C_{\les d'} \phi_{\mu} \gamma^0 \mc{C}^{\pm_2}_{d'} \varphi_{\lambda_2} \big)\Big) \Big\|_{L^\frac{5}{3}_t L^2_x}\\
            &\lesa \lambda_2^{\frac{1}{5}}  \Big( \frac{\lambda_2}{\mu} \Big)^{\frac{2}{5}} \mu^{\frac{1}{2}} \| \phi_\mu \|_{V^2_{+, 1}} \| \varphi_{\lambda_2} \|_{V^2_{\pm_2, M}}.
    \end{align*}
 On the other hand, an application of Theorem \ref{thm:trilinear freq loc subcrit} gives
    \begin{align*}
  d^{\frac{1}{2}} \Big\| P_{\lambda_1} & \mc{C}^{\pm_1}_d \mc{I}^{\pm_1, M}\Big(\phi_\mu \gamma^0 \Pi_{\pm_2 }\varphi_{\lambda_2} -  \sum_{d' \lesa \lambda_2} C_{\les d'}^{\pm_1, M} \big( C_{\les d'} \phi_{\mu} \gamma^0 \mc{C}^{\pm_2}_{d'} \varphi_{\lambda_2} \big)\Big) \Big\|_{L^2_{t,x}}\\
                                                           &\lesa \big\| P_{\lambda_1} \Pi_{\pm_1} \mc{I}^{\pm_1, M}\Big(\phi_\mu \gamma^0 \Pi_{\pm_2 }\varphi_{\lambda_2} -  \sum_{d' \lesa \lambda_2} C_{\les d'}^{\pm_1, M} \big( C_{\les d'} \phi_{\mu} \gamma^0 \mc{C}^{\pm_2}_{d'} \varphi_{\lambda_2} \big)\Big)   \big\|_{V^2_{\pm_1, M}} \\
                                                           &\lesa{}\lambda_2^{\frac{\varrho}{2}} \Big(
                                                             \frac{\lambda_2}{\mu}\Big)^{\frac{1}{10}} \Big( \| \phi_\mu\|_{L^4_{t,x}} \lambda_2^{-\frac{1}{2}} \| \varphi_{\lambda_2} \|_{L^4_{t,x}} \Big)^{\theta_0}
                                                             \Big( \mu^{\frac{1}{2}} \| \phi_\mu \|_{V^2_{+,1}} \| \varphi_{\lambda_2} \|_{V^2_{\pm_2, M}}
                                                             \Big)^{1-\theta_0}
\end{align*}
and therefore interpolating as before gives the required bound for the first term in the decomposition (\ref{eqn-proof of F control subcrit-Lp bound wave high}). It remains to bound the second term in (\ref{eqn-proof of F control subcrit-Lp bound wave high}). Let $ 1< r<a$. Exploiting null structure and decomposing $\phi_{\mu}$ into caps of size $\alpha = (\frac{d' \lambda_2}{\mu^2})^\frac{1}{2}$, and $\varphi_{\lambda_2}$ into caps of size $\beta = ( \frac{d'}{\lambda_2})^\frac{1}{2}$,  we deduce that for all $0<\theta < \frac{1}{4}$ and $\max\{d, \lambda_2^{-1}\}\lesa d' \lesa \lambda_2$ we have
\begin{align}
     \big\| &P_{\lambda_1} \mc{C}^{\pm_1}_{d} C_{\les d'}^{\pm_1,M}\big( C_{\les d'} \phi_{\mu} \gamma^0 \mc{C}^{\pm_2}_{d'} \varphi_{\lambda_2} \big) \big\|_{L^a_t L^2_x}\notag\\
            &\lesa  d^{\frac{1}{r}-\frac{1}{a}} \Big\| \Big( \sum_{q \in Q_{\lambda_2}}\sum_{\substack{ \kappa, \kappa'' \in \mc{C}_\alpha \\ |\kappa - \kappa''| \lesa \alpha}} \sum_{\substack{\kappa' \in \mc{C}_{\beta}\\ |\kappa - \kappa'| \lesa \beta}}   \big\| P_{\lambda_1} R_{\kappa} \Pi_{\pm_1}\big( C_{\les d'} R_{\kappa''}P_q \phi_{\mu} \gamma^0 \mc{C}^{\pm_2}_{d'} R_{\kappa'} \varphi_{\lambda_2} \big) \big\|_{L^2_x}^2 \Big)^\frac{1}{2} \Big\|_{L^r_t} \notag\\
            &\lesa d^{\frac{1}{r}-\frac{1}{a}} \beta \Big( \sum_{q \in Q_{\lambda_2}} \sum_{\substack{ \kappa'' \in \mc{C}_\alpha}}  \| C_{\les d'} R_{\kappa''} P_q \phi_{\mu} \|_{L^{2r}_t L^\frac{2r}{r-1}_x}^2\Big)^\frac{1}{2}  \Big( \sum_{\kappa' \in \mc{C}_{\beta}}   \| \mc{C}^{\pm_2}_{d'} R_{\kappa'} \varphi_{\lambda_2} \|_{L^{2r}_{t,x}}^2 \Big)^\frac{1}{2} \notag\\
            &\lesa d^{\frac{1}{r}-\frac{1}{a}}\mu^{1-\frac{1}{r}} \Big( \frac{\lambda_2}{\mu} \Big)^{\frac{3}{2} - \frac{3}{2r} - 6\theta}  \Big( \frac{d'}{\lambda_2}\Big)^{1-\frac{1}{r} - 2\theta}  \Big( \| \phi_\mu \|_{L^4_{t,x}} \lambda_2^{-\frac{1}{2}} \| \varphi_{\lambda_2} \|_{L^4_{t,x}} \Big)^\theta \Big( \mu^\frac{1}{2} \| \phi_\mu \|_{V^2_{+, m}} \| \varphi_{\lambda_2} \|_{V^2_{\pm_2, M}}\Big)^{1-\theta} \label{eqn:proof of thm duhamel sub:Y bound fixed mod}
    \end{align}
where we used the bounds
$$ \Big( \sum_{\substack{ \kappa'' \in \mc{C}_\alpha}}\sum_{q\in \mc{Q}_{\lambda_2}}  \| C_{\les d'} R_{\kappa''} P_q \phi_{\mu} \|_{L^{2r}_t L^\frac{2r}{r-1}_x}^2\Big)^\frac{1}{2} \lesa \lambda_2^{\frac{1}{r}-\frac{1}{2}} \alpha^{-2\theta} \Big( \frac{\lambda_2}{\mu}\Big)^{\frac{1}{2} - \frac{1}{2r} - 4\theta} \| \phi_\mu \|_{L^4_{t,x}}^\theta \Big( \mu^\frac{1}{2} \| \phi_\mu \|_{V^2_{+, m}}\Big)^{1-\theta}$$
and
$$  \Big( \sum_{\kappa' \in \mc{C}_{\beta}}   \| \mc{C}^{\pm_2}_{d'} R_{\kappa'} \varphi_{\lambda_2} \|_{L^{2r}_{t,x}}^2 \Big)^\frac{1}{2}
            \lesa \beta^{-\theta} \lambda_2^{1-\frac{1}{r}} (d')^{\frac{1}{2}-\frac{1}{r}} \Big( \lambda_2^{-\frac{1}{2}} \| \varphi_{\lambda_2} \|_{L^4_{t,x}}\Big)^\theta \| \varphi_{\lambda_2} \|_{V^2_{\pm_2, M}}^{1-\theta}$$
which, similar to (\ref{eqn:thm trilinear subcrit:example of square sum control}), hold for all sufficiently small $\theta>0$ and follow from $L^p$ interpolation, Lemma \ref{lem:wave strichartz}, an application of H\"older's inequality, and the square sum bound for $V^2$. An application of \eqref{eq:global-nr} implies that after restricting the output to modulation $d$, the sum over the modulation is only over the region $\max\{ d, \lambda_2^{-1}\} \lesa d' \lesa \lambda_2$. Consequently, summing up \eqref{eqn:proof of thm duhamel sub:Y bound fixed mod} we deduce that for $1<r<a$ and $\theta>0$ sufficiently small
    \begin{align}
     \Big\| &P_{\lambda_1} \mc{C}^{\pm_1}_d \mc{I}^{\pm_1, M}\Big( \sum_{d' \lesa \lambda_2} C_{\les d'}^{\pm_1, M} \big( C_{\les d'} \phi_{\mu} \gamma^0 \mc{C}^{\pm_2}_{d'} \varphi_{\lambda_2} \big) \Big)\Big\|_{L^a_t L^2_x}\notag \\
            &\lesa d^{-1}  \sum_{\max\{ d, \lambda_2^{-1}\} \lesa d' \lesa \lambda_2}  \big\| P_{\lambda_1} \mc{C}^{\pm_1}_{d} \big( C_{\les d'} \phi_{\mu} \gamma^0 \mc{C}^{\pm_2}_{d'} \varphi_{\lambda_2} \big) \big\|_{L^a_t L^2_x}\notag \\
            &\lesa  d^{\frac{1}{r}-\frac{1}{a}-1} \mu^{1-\frac{1}{r}} \Big( \frac{\lambda_2}{\mu} \Big)^{\frac{3}{2} - \frac{3}{2r} - 6\theta}  \sum_{\max\{ d, \lambda_2^{-1}\} \lesa d' \lesa \lambda_2} \Big( \frac{d'}{\lambda_2}\Big)^{1-\frac{1}{r} - 2\theta}  \Big( \| \phi_\mu \|_{L^4_{t,x}} \lambda_2^{-\frac{1}{2}} \| \varphi_{\lambda_2} \|_{L^4_{t,x}} \Big)^\theta \Big( \mu^\frac{1}{2} \| \phi_\mu \|_{V^2_{+, m}} \| \varphi_{\lambda_2} \|_{V^2_{\pm_2, M}}\Big)^{1-\theta} \notag\\
            &\lesa d^{-\frac{1}{a}} \Big( \frac{d}{\mu} \Big)^{\frac{1}{r}-1} \Big( \frac{\lambda_2}{\mu} \Big)^{1-\frac{1}{r}}\Big( \| \phi_\mu \|_{L^4_{t,x}} \lambda_2^{-\frac{1}{2}} \| \varphi_{\lambda_2} \|_{L^4_{t,x}} \Big)^\theta \Big( \mu^\frac{1}{2} \| \phi_\mu \|_{V^2_{+, m}} \| \varphi_{\lambda_2} \|_{V^2_{\pm_2, M}}\Big)^{1-\theta}. \label{eqn:proof of thm duhamel sub:Y bd without loss}
    \end{align}
Therefore, taking $\frac{1}{r}=1 - 4(\frac{1}{a}-\frac{1}{2})$ we obtain \eqref{eqn:proof of thm duhamel subcrit:psi bd Y}. This completes the proof of Theorem \ref{thm:duhamel-sub}.

        \section{Multilinear Estimates in the critical
          case}\label{sec:multi-crit}
        In this section we consider the scale-invariant regime with a small amount of angular regularity. Here, after rescaling, we have $m=1$ and $M>0$.

        \begin{theorem}\label{thm:duhamel-crit}
          Let $M>0$ and $\sigma>0$. There exists $0<\theta<1$, $1<a<2$, and $b>0$ such that for all $s\g 0$
                \begin{equation}\label{eqn:thm duhamel crit:V psi}
\begin{split}
                    \big\|\Pi_{\pm_1}\mathcal{I}^{\pm_1,M}\big(\phi \gamma^0\Pi_{\pm_2}\varphi\big)\big\|_{\mb{V}^{s, \sigma}_{\pm_1,M}}
                        \lesa{}& \Big( \sup_{\mu, \lambda_2 \ge 1}\|\phi_\mu\|_{\mb{D}^{s}_{\sigma}}\|\varphi_{\lambda_2}\|_{\mb{D}^{-\frac12}_{\sigma}}\Big)^\theta
                        \Big(\|\phi\|_{\mb{V}^{s+\frac12, \sigma}_{+,1}} \|\varphi\|_{\mb{F}_{\pm_2, M}^{0, \sigma}}\Big)^{1-\theta}\\
&+\Big( \sup_{\mu, \lambda_2 \ge 1}\|\phi_\mu\|_{\mb{D}^{0}_{\sigma}}\|\varphi_{\lambda_2}\|_{\mb{D}^{s-\frac12}_{\sigma}}\Big)^\theta
                        \Big(\|\phi\|_{\mb{V}^{\frac12, \sigma}_{+,1}} \|\varphi\|_{\mb{F}_{\pm_2, M}^{s, \sigma}}\Big)^{1-\theta}
                      \end{split}
                    \end{equation}
        and
               \begin{equation}\label{eqn:thm duhamel crit:Y psi}
\begin{split}
            \big\|\Pi_{\pm_1}\mathcal{I}^{\pm_1,M}\big(\phi \gamma^0\Pi_{\pm_2}\varphi\big)\big\|_{\mb{Y}^{s, \sigma}_{\pm_1,M}}
                \lesa{}&  \Big( \sup_{\mu, \lambda_2 \ge 1}\|\phi_\mu\|_{\mb{D}^{s}_{\sigma}}\|\varphi_{\lambda_2}\|_{\mb{D}^{-\frac12}_{\sigma}}\Big)^\theta
            \Big(\|\phi\|_{\mb{V}^{s+\frac12, \sigma}_{+,1}} \|\varphi\|_{\mb{V}_{\pm_2, M}^{0, \sigma}}\Big)^{1-\theta} \\
&+\Big( \sup_{\mu, \lambda_2 \ge 1}\|\phi_\mu\|_{\mb{D}^{0}_{\sigma}}\|\varphi_{\lambda_2}\|_{\mb{D}^{s-\frac12}_{\sigma}}\Big)^\theta
            \Big(\|\phi\|_{\mb{V}^{\frac12, \sigma}_{+,1}} \|\varphi\|_{\mb{V}_{\pm_2, M}^{s, \sigma}}\Big)^{1-\theta}.
          \end{split}
        \end{equation}
Similarly,
          \begin{equation}\label{eqn:thm duhamel crit:V phi}
            \begin{split}
            \big\|\langle \nabla\rangle^{-1}\mathcal{I}^{+,1}\big(\overline{\Pi_{\pm_1}\psi} \Pi_{\pm_2}\varphi\big)\big\|_{\mb{V}^{s+\frac{1}{2}, \sigma}_{+,1}}
             \lesa{}& \Big(\sup_{\lambda_1,\lambda_2 \ge 1} \|\psi_{\lambda_1} \|_{\mb{D}^{s-\frac12}_{\sigma}} \|\varphi_{\lambda_2} \|_{\mb{D}^{-\frac12}_{\sigma}}\Big)^\theta \Big(\| \psi \|_{\mb{F}_{\pm_1, M}^{s, \sigma}} \|\varphi\|_{\mb{F}_{\pm_2, M}^{0, \sigma}}\Big)^{1-\theta}\\
&+\Big(\sup_{\lambda_1,\lambda_2 \ge 1} \|\psi_{\lambda_1} \|_{\mb{D}^{-\frac12}_{\sigma}} \|\varphi_{\lambda_2} \|_{\mb{D}^{s-\frac12}_{\sigma}}\Big)^\theta \Big(\| \psi \|_{\mb{F}_{\pm_1, M}^{0, \sigma}} \|\varphi\|_{\mb{F}_{\pm_2, M}^{s, \sigma}}\Big)^{1-\theta}.
            \end{split}
          \end{equation}
        \end{theorem}

        The proof will be postponed to Subsection
        \ref{subsec:proof-duhamel-crit} below.

        \subsection{The Trilinear Estimate}\label{subsec:tri-crit}
To obtain the $L^4_{t,x}$ norms in Theorem \ref{thm:duhamel-crit}, we use the following consequence of Theorem \ref{thm:bilinear small scale KG} which gives the required $L^4_{t,x}$ at a cost of modulation and high-low factors. However, as we will have a little room in our estimates, we can always absorb a small power of this loss elsewhere, which is sufficient to obtain the required Strichartz norms in Theorem \ref{thm:duhamel-crit}.

\begin{lemma}\label{lem:tri L4 gain}
Let $d>0$ and for $j=1, 2, 3$, let $m_j > 0$, and $\lambda_j, N_j \g 1$. If $\sigma>0$ there exists $\theta>0$ such that
     \begin{align*} \Big| \int_{\RR^{1+3}} C^{\pm_1, m_1}_d u_{\lambda_1, N_1} v_{\lambda_2, N_2} w_{\lambda_3, N_3} dx dt \Big| &\lesa N_{min}^{\sigma} \Big(\frac{\min\{d, \lambda_1\}}{\lambda_1}\Big)^{-\frac{1}{2}-\sigma} \lambda_{min}^{-1}\Big(\| u_{\lambda_1, N_1} \|_{L^4_{t,x}} \| v_{\lambda_2, N_3} \|_{L^4_{t,x}} \| w_{\lambda_3, N_3} \|_{L^4_{t,x}}\Big)^\theta \\
       &\qquad \qquad \cdot \Big( (\lambda_1 \lambda_2 \lambda_3)^{\frac{1}{2}}\| u_{\lambda_1, N_1} \|_{V^2_{\pm_1, m_1}} \| v_{\lambda_2, N_2} \|_{V^2_{\pm_2, m_2}} \| w_{\lambda_3, N_3} \|_{V^2_{\pm_3, m_3}}\Big)^{1-\theta}
     \end{align*}
where $\lambda_{min} = \min\{ \lambda_1, \lambda_2, \lambda_3\}$ and $N_{min} = \min\{N_1, N_2, N_3\}$.
\begin{proof}
We first observe that an application of H\"{o}lder's inequality gives
    \begin{align}
   & \Big| \int_{\RR^{1+3}} C^{\pm_1, m_1}_d u_{\lambda_1, N_1} v_{\lambda_2, N_2} w_{\lambda_3, N_3} dx dt \Big| \notag \\
\les{} & \| C^{\pm_1, m_1}_d u_{\lambda_1, N_1} \|_{L^2_{t,x}} \| v_{\lambda_2, N_2} \|_{L^4_{t,x}} \| w_{\lambda_3, N_3} \|_{L^4_{t,x}} \notag \\
      \lesa{}& d^{-\frac{1}{2}}   \| u_{\lambda_1, N_1} \|_{V^2_{\pm_1, m_1}} \| v_{\lambda_2, N_2} \|_{L^4_{t,x}} \| w_{\lambda_3, N_3} \|_{L^4_{t,x}} \notag \\
      \lesa{}& \Big( \frac{\min\{d, \lambda_1\}}{\lambda_1}\Big)^{-\frac{1}{2}} \lambda_{min}^{-1}\lambda_1^{\frac{1}{2}}  \| u_{\lambda_1, N_1} \|_{V^2_{\pm_1, m_1}} \| v_{\lambda_2, N_2} \|_{L^4_{t,x}} \| w_{\lambda_3, N_3} \|_{L^4_{t,x}}.
      \label{eqn:proof of lem tri L4 gain:L4 norms of v w}\end{align}
Thus it only remains to show we can gain a power of the $L^4_{t,x}$ norm of $u_{\lambda_1, N_1}$. To this end, we first consider the case where $N_2 = N_{min}$. An application of the wave Strichartz estimate with angular regularity in Lemma \ref{lem:wave strichartz} gives for $3<p<4$
    \begin{align}
     & \Big| \int_{\RR^{1+3}} C^{\pm_1, m_1}_d u_{\lambda_1, N_1} v_{\lambda_2, N_2} w_{\lambda_3, N_3} dx dt \Big| \notag \\
        \les{}& \| C^{\pm_1, m_1}_d u_{\lambda_1, N_1} \|_{L^4_{t,x}}^{\frac{4}{p}-1} \| C^{\pm_1, m_1}_d u_{\lambda_1, N_1} \|_{L^2_{t,x}}^{2-\frac{4}{p}} \| v_{\lambda_2, N_2} \|_{L^p_{t,x}} \| w_{\lambda_3, N_3} \|_{L^4_{t,x}} \notag \\
      \lesa{}& \Big( \frac{ \min\{ d, \lambda_1\}}{\lambda_1} \Big)^{ - \frac{5}{2}(\frac{4}{p}-1)}  \| u_{\lambda_1, N_1} \|_{L^4_{t,x}}^{\frac{4}{p}-1} \Big( d^{-\frac{1}{2}} \| u_{\lambda_1, N_1} \|_{V^2_{\pm_1, m_1}}\Big)^{2-\frac{4}{p}} N_2 \lambda_2^{\frac{3}{2}-\frac{4}{p}} \| v_{\lambda_2, N_2} \|_{V^2_{\pm_2, m_2}}\lambda_3^{\frac{1}{2}}  \| w_{\lambda_3, N_3} \|_{V^2_{\pm_3, m_3}}\notag\\
      \lesa{}& \Big( \frac{ \min\{ d, \lambda_1\}}{\lambda_1} \Big)^{ -\frac{3}{2}} \frac{N_{min}}{ \lambda_{min}} \| u_{\lambda_1, N_1} \|_{L^4_{t,x}}^{\frac{4}{p}-1} \Big( \lambda_1^{\frac{1}{2}} \| u_{\lambda_1, N_1} \|_{V^2_{\pm_1, m_1}}\Big)^{2-\frac{4}{p}} \lambda_2^{\frac{1}{2}} \| v_{\lambda_2, N_2} \|_{V^2_{\pm_2, m_2}} \lambda_3^{\frac{1}{2}} \| w_{\lambda_3, N_3} \|_{V^2_{\pm_3, m_3}}
      \label{eqn:proof of lem tri L4 gain:N2 small}
    \end{align}
where to remove the $C^{\pm_1, m_1}_d$ multiplier from the $L^4_{t,x}$ norm, we let $\alpha = (\frac{\min\{d, \lambda_1\}}{\lambda_1})^{\frac{1}{2}}$ and apply Lemma \ref{lem:disposable} to deduce that
    \begin{align}
      \| C^{\pm_1, m_1}_d u_{\lambda_1, N_1} \|_{L^4_{t,x}} \les{}& \sum_{\kappa \in \mc{C}_\alpha} \sum_{q \in \mc{Q}_{\alpha \lambda^2} } \| C^{\pm_1, m_1}_d R_{\kappa} P_q u_{\lambda_1, N_1} \|_{L^4_{t,x}} \notag\\
            \lesa{}& \alpha^{-2} \big( (\alpha \lambda)^{-3}  + 1 \big) \| u_{\lambda_1, N_1} \|_{L^4_{t,x}} \lesa \Big( \frac{\min\{d, \lambda_1\}}{\lambda_1} \Big)^{-\frac{5}{2}} \| u_{\lambda_1, N_1} \|_{L^4_{t,x}}. \label{eqn:proof of lem tri L4 gain:Cd dis}
    \end{align}
If we combine \eqref{eqn:proof of lem tri L4 gain:L4 norms of v w} and \eqref{eqn:proof of lem tri L4 gain:N2 small} we deduce that
    \begin{align*}
      & \Big| \int_{\RR^{1+3}} C^{\pm_1, m_1}_d u_{\lambda_1, N_1} v_{\lambda_2, N_2} w_{\lambda_3, N_3} dx dt \Big|\\
       \lesa{}& \Big[ \Big( \frac{\min\{d, \lambda_1\}}{\lambda_1}\Big)^{-\frac{1}{2}} \lambda_{min}^{-1} \lambda_1^{\frac{1}{2}}  \| u_{\lambda_1, N_1} \|_{V^2_{\pm_1, m_1}} \| v_{\lambda_2, N_2} \|_{L^4_{t,x}} \| w_{\lambda_3, N_3} \|_{L^4_{t,x}}\Big]^{1-\sigma} \\
       &\cdot \Big[\Big( \frac{ \min\{ d, \lambda_1\}}{\lambda_1} \Big)^{ - \frac{3}{2}}  \lambda_{min}^{-1} N_{min} \| u_{\lambda_1, N_1} \|_{L^4_{t,x}}^{\frac{4}{p}-1} \Big( \lambda_1^{\frac{1}{2}} \| u_{\lambda_1, N_1} \|_{V^2_{\pm_1, m_1}}\Big)^{2-\frac{4}{p}} \lambda_2^{\frac{1}{2}} \| v_{\lambda_2, N_2} \|_{V^2_{\pm_2, m_2}} \lambda_3^{\frac{1}{2}} \| w_{\lambda_3, N_3} \|_{V^2_{\pm_3, m_3}} \Big]^{\sigma}\\
       \lesa{}& N_{min}^{\sigma} \Big(\frac{\min\{d, \lambda_1\}}{\lambda_1}\Big)^{-\frac{3}{2}} \lambda_{min}^{-1}\Big(\| u_{\lambda_1, N_1} \|_{L^4_{t,x}} \| v_{\lambda_2, N_3} \|_{L^4_{t,x}} \| w_{\lambda_3, N_3} \|_{L^4_{t,x}}\Big)^{\sigma(\frac{4}{p}-1)} \\
       &\qquad \cdot \Big( (\lambda_1 \lambda_2 \lambda_3)^{\frac{1}{2}}\| u_{\lambda_1, N_1} \|_{V^2_{\pm_1, m_1}} \| v_{\lambda_2, N_2} \|_{V^2_{\pm_2, m_2}} \| w_{\lambda_3, N_3} \|_{V^2_{\pm_3, m_3}}\Big)^{1- \sigma(\frac{4}{p}-1)}
    \end{align*}
as required. A similar argument gives the case $N_3 = N_{min}$. Thus
it remains to consider $N_1=N_{min}$. It also straightforward to deal with the case $\lambda_2 \approx 1$, since we may then apply the Klein-Gordon Stichartz estimate. More precisely, an application of Lemma \ref{lem:KG strichartz} to $v_{\lambda_2, N_2}$ together with $L^q_{t,x}$ interpolation, gives for $\frac{10}{3}\les p < 4$
    \begin{align}
      &\Big| \int_{\RR^{1+3}} C^{\pm_1, m_1}_d u_{\lambda_1, N_1} v_{\lambda_2, N_2} w_{\lambda_3, N_3} dx dt \Big| \notag \\
        \les{}& \| C^{\pm_1, m_1}_d u_{\lambda_1, N_1} \|_{L^4_{t,x}}^{\frac{4}{p}-1} \| C^{\pm_1, m_1}_d u_{\lambda_1, N_1} \|_{L^2_{t,x}}^{2-\frac{4}{p}} \| v_{\lambda_2, N_2} \|_{L^p_{t,x}} \| w_{\lambda_3, N_3} \|_{L^4_{t,x}} \notag \\
      \lesa{}& \Big( \frac{ \min\{ d, \lambda_1\}}{\lambda_1} \Big)^{ - \frac{5}{2}(\frac{4}{p}-1)}  \lambda_3^{\frac{1}{2}}  \| u_{\lambda_1, N_1} \|_{L^4_{t,x}}^{\frac{4}{p}-1} \Big( d^{-\frac{1}{2}} \| u_{\lambda_1, N_1} \|_{V^2_{\pm_1, m_1}}\Big)^{2-\frac{4}{p}} \| v_{\lambda_2, N_2} \|_{V^2_{\pm_2, m_2}} \| w_{\lambda_3, N_3} \|_{V^2_{\pm_3, m_3}}\notag\\
      \lesa{}& \Big( \frac{ \min\{ d, \lambda_1\}}{\lambda_1} \Big)^{ \frac{3}{2}- \frac{8}{p}}  \lambda_{min}^{-1}  \| u_{\lambda_1, N_1} \|_{L^4_{t,x}}^{\frac{4}{p}-1} \Big( \lambda_1^{\frac{1}{2}} \| u_{\lambda_1, N_1} \|_{V^2_{\pm_1, m_1}}\Big)^{2-\frac{4}{p}} \lambda_2^{\frac{1}{2}} \| v_{\lambda_2, N_2} \|_{V^2_{\pm_2, m_2}} \lambda_3^{\frac{1}{2}} \| w_{\lambda_3, N_3} \|_{V^2_{\pm_3, m_3}}
      \label{eqn:proof of lem tri L4 gain:lambda approx one}
    \end{align}
where we applied the bound \eqref{eqn:proof of lem tri L4 gain:Cd dis}. Combining \eqref{eqn:proof of lem tri L4 gain:L4 norms of v w} and \eqref{eqn:proof of lem tri L4 gain:lambda approx one} we obtain the required bound in the case where $\lambda_2 \approx 1$. A similar argument gives the case $\lambda_3 \approx 1$.

It remains to gain a power of the $L^4$ norm of $u_{\lambda_1, N_1}$, in the case where $N_{min} = N_1$ and $\lambda_2, \lambda_3 \gg 1$. Clearly we may also assume that $\lambda_2\g \lambda_3$. In this region we apply the bilinear restriction estimate to deduce the required bound. The first step is to observe that, as either $\lambda_1 \ll \lambda_2 \approx \lambda_3$, or $\lambda_1 \approx \max\{\lambda_2, \lambda_3\}$, an angular Whitney type decomposition gives the identity
    \begin{align} \int_{\RR^{1+3}} &C^{\pm_1, m_1}_d u_{\lambda_1, N_1} v_{\lambda_2, N_2} w_{\lambda_3, N_3} dx dt\notag \\
                &= \sum_{\frac{1}{ \lambda_3} \lesa \ell \lesa \max\{1,\frac{\lambda_1}{ \lambda_3}\}} \sum_{\substack{\kappa, \kappa' \in \mc{C}_{\ell} \\ |\pm_2 \kappa - \pm_3\kappa'| \approx \ell}} \int_{\RR^{1+3}} C^{\pm_1, m_1}_d u_{\lambda_1, N_1} R_{\kappa} v_{\lambda_2, N_2} R_{\kappa'} w_{\lambda_3, N_3} dx dt \notag\\
                &\qquad \qquad \qquad \qquad  \qquad \qquad  + \sum_{\substack{\kappa, \kappa' \in \mc{C}_{\frac{1}{ \lambda_3}} \\ |\pm_2 \kappa - \pm_3\kappa'| \lesa \frac{1}{\lambda_3}}} \int_{\RR^{1+3}} C^{\pm_1, m_1}_d u_{\lambda_1, N_1} R_{\kappa} v_{\lambda_2, N_2} R_{\kappa'} w_{\lambda_3, N_3} dx dt. \label{eqn:proof of lem tri L4 gain:initial ang decomp}
    \end{align}
To estimate the first term in \eqref{eqn:proof of lem tri L4 gain:initial ang decomp}, we note that as we may restrict the support of $\widehat{u}_{\lambda_1, N_1}$ to lie in the set
        $$\{ -(\xi + \eta) \mid \xi \in \kappa, \eta \in \kappa', |\xi| \approx \lambda_2, |\eta| \approx \lambda_3\} $$
we have  for $\kappa, \kappa' \in \mc{C}_\ell$ with $|\pm_2 \kappa - \pm_3 \kappa'| \approx \ell$
    \begin{align*}
    \int_{\RR^{1+3}} C^{\pm_1, m_1}_d u_{\lambda_1, N_1} R_{\kappa} v_{\lambda_2, N_2} R_{\kappa'} w_{\lambda_3, N_3} dx dt
    &=  \sum_{\kappa'' \in \mc{E}} \int_{\RR^{1+3}} C^{\pm_1, m_1}_d
      R_{\kappa''} u_{\lambda_1, N_1} R_{\kappa} v_{\lambda_2, N_2}
      R_{\kappa'} w_{\lambda_3, N_3} dx dt
    \end{align*}
where $
\mc{E}$ is the set of all $\kappa''\in \mc{C}_{\ell
  \frac{\lambda_3}{\lambda_1}}$ satisfying $\min\{|\kappa''-\kappa|, |\kappa'' + \kappa|\} \lesa \ell \frac{\lambda_3}{\lambda_1}\}$. Notice that we have $\# \mc{E}\lesa 1$.
Consequently, applying H\"older's inequality, the angular concentration bound \eqref{eqn:ang-con}, and Theorem \ref{thm:bilinear small scale KG}, we deduce that for  for $\kappa, \kappa' \in \mc{C}_\ell$ with $|\pm_2 \kappa - \pm_3 \kappa'| \approx \ell$ and $\frac{3}{2}<p<2$ we have
      \begin{align*}& \Big|\int_{\RR^{1+3}} C^{\pm_1, m_1}_d u_{\lambda_1, N_1} R_{\kappa} v_{\lambda_2, N_2} R_{\kappa'} w_{\lambda_3, N_3}  dx dt\Big|\\
                \lesa{}& \Big(\sup_{ \kappa'' \in \mc{C}_{\ell \frac{\lambda_3}{\lambda_1} }}
                    \| C^{\pm_1, m_1}_d R_{\kappa''} u_{\lambda_1, N_1}\|_{L^4_{t,x}}^{4(\frac{1}{p}-\frac{1}{2})} \| C^{\pm_1, m_1}_d R_{\kappa''} u_{\lambda_1, N_1}\|_{L^2_{t,x}}^{3-\frac{4}{p}} \Big) \| R_{\kappa} v_{\lambda_2, N_2} R_{\kappa'} w_{\lambda_3, N_3}\|_{L^p_{t,x}} \\
                \lesa{}& N_1^{\sigma} \Big(\ell \frac{\lambda_3}{\lambda_1}\Big)^{\sigma} \ell^{2-\frac{4}{p}} \lambda_3^{\frac{7}{2}-\frac{5}{p}-\epsilon} \lambda_2^{\frac{1}{q} - \frac{1}{2} + \epsilon}  \| C^{\pm_1, m_1}_d u_{\lambda_1, N_1}\|_{L^4_{t,x}}^{4(\frac{1}{p}-\frac{1}{2})} \Big( d^{-\frac{1}{2}} \| u_{\lambda_1, N_1}\|_{V^2_{\pm_1, m_1}}\Big)^{3-\frac{4}{p}}\\
                &\qquad \qquad \qquad \qquad \qquad \qquad \qquad \qquad \cdot \| R_{\kappa} v_{\lambda_2, N_2}\|_{V^2_{\pm_2, m_2}} \| R_{\kappa'} w_{\lambda_3, N_3}\|_{V^2_{\pm_3, m_3}}  \\
                \lesa{}& \ell^{\sigma + 2 - \frac{4}{p}} N_1^\sigma \Big( \frac{\min\{d, \lambda_1\}}{\lambda_1} \Big)^{\frac{7}{2}-\frac{8}{p}} \lambda_{min}^{-1} \| u_{\lambda_1, N_1}\|_{L^4_{t,x}}^{4(\frac{1}{p}-\frac{1}{2})} \Big( \lambda_1^{\frac{1}{2}} \| u_{\lambda_1, N_1}\|_{V^2_{\pm_1, m_1}}\Big)^{3-\frac{4}{p}} \\
                &\qquad \qquad \qquad \qquad \qquad \qquad \qquad \qquad \cdot \lambda_2^{\frac{1}{2}} \| R_{\kappa} v_{\lambda_2, N_2}\|_{V^2_{\pm_2, m_2}}  \lambda_3^{\frac{1}{2}} \| R_{\kappa'} w_{\lambda_3, N_3}\|_{V^2_{\pm_3, m_3}}
    \end{align*}
where we again applied the bound \eqref{eqn:proof of lem tri L4 gain:Cd dis} to dispose of the $C_d$ multiplier. Hence summing up over caps, letting $\frac{1}{p} = \frac{1}{2} + \frac{\sigma}{8}$, and applying the square sum bound in $V^2$, we deduce that
    \begin{align*}
     & \sum_{\frac{1}{ \lambda_3} \lesa \ell \lesa \max\{1,\frac{\lambda_1}{ \lambda_3}\}} \sum_{\substack{\kappa, \kappa' \in \mc{C}_{\ell} \\ |\pm_2 \kappa - \pm_3\kappa'| \approx \ell}} \Big|\int_{\RR^{1+3}} C^{\pm_1, m_1}_d u_{\lambda_1, N_1} R_{\kappa} v_{\lambda_2, N_2} R_{\kappa'} w_{\lambda_3, N_3} dx dt\Big|\\
        \lesa{}& N_1^\sigma \Big( \frac{\min\{d, \lambda_1\}}{\lambda_1} \Big)^{-\frac{1}{2}-\sigma} \lambda_{min}^{-1} \| u_{\lambda_1, N_1}\|_{L^4_{t,x}}^{\frac{\sigma}{2}} \Big( \lambda_1^{\frac{1}{2}} \| u_{\lambda_1, N_1}\|_{V^2_{\pm_1, m_1}}\Big)^{1-\frac{\sigma}{2}} \lambda_2^{\frac{1}{2}} \| R_{\kappa} v_{\lambda_2, N_2}\|_{V^2_{\pm_2, m_2}}  \lambda_3^{\frac{1}{2}} \| R_{\kappa'} w_{\lambda_3, N_3}\|_{V^2_{\pm_3, m_3}}.
    \end{align*}
Together with \eqref{eqn:proof of lem tri L4 gain:L4 norms of v w}, we obtain an acceptable bound for the first term in \eqref{eqn:proof of lem tri L4 gain:initial ang decomp}. To bound the second term in \eqref{eqn:proof of lem tri L4 gain:initial ang decomp}, we apply a similar argument together with the Klein-Gordon Strichartz estimate to deduce that for $\frac{5}{3}<p<2$ we have
    \begin{align*}
      &\sum_{\substack{\kappa, \kappa' \in \mc{C}_{\frac{1}{ \lambda_3}} \\ |\pm_2 \kappa - \pm_3\kappa'| \lesa \frac{1}{\lambda_3}}} \Big|\int_{\RR^{1+3}} C^{\pm_1, m_1}_d u_{\lambda_1, N_1} R_{\kappa} v_{\lambda_2, N_2} R_{\kappa'} w_{\lambda_3, N_3} dx dt\Big|\\
            \lesa{}& \sum_{\substack{\kappa, \kappa' \in \mc{C}_{\frac{1}{ \lambda_3}} \\ |\pm_2 \kappa - \pm_3\kappa'| \lesa \frac{1}{\lambda_3}}} \sup_{\kappa'' \in \mc{C}_{\ell \frac{\lambda_3}{\lambda_1}}} \| C^{\pm_1, m_1}_d R_{\kappa''} u_{\lambda_1, N_1}\|_{L^4_{t,x}}^{4(\frac{1}{p}-\frac{1}{2})} \| C^{\pm_1, m_1}_d R_{\kappa''} u_{\lambda_1, N_1}\|_{L^2_{t,x}}^{3-\frac{4}{p}} \| R_{\kappa} v_{\lambda_2, N_2} \|_{L^{2p}_{t,x}} \| R_{\kappa'} w_{\lambda_3, N_3} \|_{L^{2p}_{t,x}} \\
            \lesa{}& N_1^\sigma \Big( \frac{\min\{d,
              \lambda_1\}}{\lambda_1} \Big)^{\frac{7}{2} -
              \frac{8}{p}} \lambda_1^{8(\frac{1}{p} -\frac{1}{2}) -
              \sigma} \lambda_{min}^{-1}   \|u_{\lambda_1,
              N_1}\|_{L^4_{t,x}}^{4(\frac{1}{p}-\frac{1}{2})} \Big(
              \lambda_1^{-\frac{1}{2}}\| u_{\lambda_1,
              N_1}\|_{V^2_{\pm_1, m_1}}\Big)^{3-\frac{4}{p}} \\
&\qquad \cdot \lambda_2^{\frac{1}{2}} \| v_{\lambda_2, N_2} \|_{V^2_{\pm_2, m_2}} \lambda_3^{\frac{1}{2}} \| w_{\lambda_3, N_3} \|_{V^2_{\pm_3, N_3}}.
    \end{align*}
Choosing $\frac{1}{p} = \frac{1}{2} + \frac{\sigma}{8}$ as before, and combining the resulting bound with \eqref{eqn:proof of lem tri L4 gain:L4 norms of v w}, the lemma follows.
\end{proof}
\end{lemma}

We now give the proof of the main step in the proof of Theorem \ref{thm:duhamel-crit}.

        \begin{theorem}\label{thm:tri freq loc crit}
          Let $M>\frac{1}{2}$, $0<\varrho \ll 1$, $\frac{5}{3}<a<2$ and $0<b<\frac{\varrho}{4}$. Define
            $$ \mb{A} = \| \phi_{\mu, N} \|_{L^4_{t,x}} \lambda_1^{-\frac{1}{2}} \| \psi_{\lambda_1, N_1} \|_{L^4_{t,x}} \lambda_2^{-\frac{1}{2}} \| \varphi_{\lambda_2, N_2} \|_{L^4_{t, x}}.$$
          There exists $\theta_0\in (0,1)$ such that, if $\mu \lesa \lambda_1\sim \lambda_2$ we have
          \begin{equation}\label{eqn:tri freq loc crit V2:main} \begin{split} \Big|\int_{\RR^{3+1}} \phi_{\mu,N} &\overline{ \Pi_{\pm_1} \psi_{\lambda_1,N_1}} \Pi_{\pm_2} \varphi_{\lambda_2,N_2} dx dt\Big| \\
              &\lesa N_{min}^{\varrho} \Big( \frac{\mu}{\lambda_1}\Big)^{ \frac{1}{10} }   \mb{A}^{\theta_0} \big( \mu^\frac{1}{2} \| \phi_{\mu,N} \|_{V^2_{+, 1}} \| \psi_{\lambda_1,N_1} \|_{V^2_{\pm_1,M}} \| \varphi_{\lambda_2,N_2} \|_{V^2_{\pm_2, M}}\big)^{1-\theta_0}.
              \end{split}
          \end{equation}
          On the other hand, when $\mu \sim \lambda_1 \gg \lambda_2$ we have
          \begin{equation}\label{eqn:tri freq loc crit V2:main-A0} \begin{split}  \Big|\int_{\RR^{3+1}} &\phi_{\mu,N} \overline{ \Pi_{\pm_1} \psi_{\lambda_1,N_1}} \Pi_{\pm_2} \varphi_{\lambda_2,N_2} - \sum_{\lambda_2^{-1} \lesa d \lesa \lambda_2} C_{\les d} \phi_{\mu, N} \overline{\mc{C}^{\pm_1}_{\les d} \psi_{\lambda_1, N_1}} \mc{C}^{\pm_2}_d \varphi_{\lambda_2, N_2} dx dt \Big| \\
              &\lesa N_{min}^{\varrho} \Big( \frac{ \lambda_2}{\mu}\Big)^{\frac{1}{10}}  \mb{A}^{\theta_0} \big( \mu^\frac{1}{2} \| \phi_{\mu,N} \|_{V^2_{+, 1}} \| \psi_{\lambda_1, N_1} \|_{V^2_{\pm_1, M}} \| \varphi_{\lambda_2, N_2} \|_{V^2_{\pm_2, M}}\big)^{1-\theta_0}\end{split}
          \end{equation}
          and
          \begin{equation}\label{eqn:tri freq loc crit:d bound with V}
            \begin{split} \sum_{\lambda_2^{-1} \lesa d \lesa \lambda_2} \Big|\int_{\RR^{3+1}}& C_{\les d} \phi_{\mu, N} \overline{\mc{C}^{\pm_1}_{\les d} \psi_{\lambda_1, N_1} } \mc{C}^{\pm_2}_d \varphi_{\lambda_2, N_2} dx dt \Big| \\
              &\lesa (\min\{N, N_1\})^{\varrho} \Big( \frac{ \lambda_2}{\mu}\Big)^{\frac{1}{10}}  \mb{A}^{\theta_0} \big( \mu^\frac{1}{2} \| \phi_{\mu,N} \|_{V^2_{+, 1}} \| \psi_{\lambda_1, N_1} \|_{V^2_{\pm_1, M}} \| \varphi_{\lambda_2, N_2} \|_{V^2_{\pm_2, M}}\big)^{1-\theta_0}.\end{split}
          \end{equation}
          Moreover, if we also use the $Y^{\pm, M}$ norm, we have
          \begin{equation}\label{eqn:tri freq loc crit:d bound with Y}
            \begin{split} \sum_{\lambda_2^{-1} \lesa d \lesa \lambda_2} \Big|\int_{\RR^{3+1}}& C_{\les d} \phi_{\mu, N} \overline{\mc{C}^{\pm_1}_{\les d} \psi_{\lambda_1, N_1} } \mc{C}^{\pm_2}_d \varphi_{\lambda_2, N_2} dx dt \Big| \\
              &\lesa N_{min}^{\varrho} \Big( \frac{ \lambda_2}{\mu}\Big)^{\frac{1}{4}(\frac{1}{a}-\frac{1}{2})}  \mb{A}^{\theta_0} \big( \mu^\frac{1}{2} \| \phi_{\mu,N} \|_{V^2_{+, 1}} \| \psi_{\lambda_1, N_1} \|_{V^2_{\pm_1, M}} \| \varphi_{ N_2} \|_{Y^{\pm_2, M}_{\lambda_2}}\big)^{1-\theta_0}.\end{split}
          \end{equation}
          Analogous bounds hold in the case $\lambda_1 \ll \lambda_2$.
        \end{theorem}
        \begin{proof}
          As in the proof of Theorem \ref{thm:trilinear freq loc subcrit}, we begin by decomposing
          \begin{align*} \phi_{\mu} \big( \Pi_{\pm_1}& \psi_{\lambda_1}\big)^\dagger \gamma^0 \Pi_{\pm_2} \varphi_{\lambda_2} \\
                                                     &=\sum_{d} C_d \phi_{\mu} \big( \mc{C}_{\ll d}^{\pm_1} \psi_{\lambda_1}\big)^\dagger \gamma^0 \mc{C}_{\ll d}^{\pm_2} \varphi_{\lambda_2} + C_{\lesa d} \phi_{\mu} \big( \mc{C}_{ d}^{\pm_1} \psi_{\lambda_1}\big)^\dagger \gamma^0 \mc{C}_{\lesa d}^{\pm_2} \varphi_{\lambda_2}+C_{\lesa d} \phi_{\mu} \big( \mc{C}_{\lesa d}^{\pm_1} \psi_{\lambda_1}\big)^\dagger \gamma^0 \mc{C}_{ d }^{\pm_2} \varphi_{\lambda_2}\\
                                                     &=\sum_d A_0 +
                                                       A_1 + A_2.
          \end{align*}
          and consider seperately the small modulation cases
	$$ \mu \ll \lambda_1 \approx \lambda_2 \text{ and } d \lesa \mu, \qquad \qquad \mu  \gtrsim \min\{ \lambda_1, \lambda_2\} \text{ and } d \lesa \min\{\lambda_1, \lambda_2\} $$
        and the high modulation case
	$$ d \gg \min\{\mu, \lambda_1, \lambda_2\}$$

Note that the bound \eqref{eq:global-nr} does not hold in the case $M\leq \frac12$, which is admissible now.\\
	
        \textbf{Case 1: $\mu \ll \lambda_1 \approx \lambda_2$ and
          $d \lesa \mu$.}
From \cite[(8.16)]{Candy2016} we have the bound
        \[   \Big| \int A_0 dx dt \Big|  \lesa \Big( \frac{d}{\mu}\Big)^{\frac{1}{8} - \epsilon} \Big( \frac{\mu}{\lambda_1}\Big)^{-1} N_{min}^{\frac{1}{4}} \mu^{\frac{1}{2}} \| \phi_{\mu, N} \|_{V^2_{+, 1}} \| \psi_{\lambda_1, N_1} \|_{V^2_{\pm_1, M}} \| \varphi_{\lambda_2, N_2} \|_{V^2_{\pm_2, M}}. \]
        Combining this bound with Lemma \ref{lem:tri L4 gain}, choosing $\epsilon>0$ sufficiently small, and summing up over $d \lesa \mu$, then gives $\theta>0$ such that
            $$ \sum_{d \lesa \mu} \Big| \int_{\RR^{1+3}} A_0 dx dt \Big|
                        \lesa N_{min}^{\varrho} \Big( \frac{\mu}{\lambda_2} \Big)^{\frac{1}{4}} \mb{A}^{\theta}\Big( \mu^{\frac{1}{2}} \| \phi_{\mu, N} \|_{V^2_{+, 1}} \| \psi_{\lambda_1, N_1} \|_{V^2_{\pm_1, M}} \| \varphi_{\lambda_2, N_2} \|_{V^2_{\pm_2, M}}\Big)^{1-\theta}. $$
        We now turn to the bound for the $A_1$ term. From \cite[(8.18)]{Candy2016} we have for every $\epsilon>0$
        $$ \Big| \int A_1 dx dt \Big|  \lesa \Big( \frac{d}{\mu}\Big)^{\frac{1}{8} - \epsilon} \Big( \frac{\mu}{\lambda_1}\Big)^{-\frac{1}{2}} N_{min}^{\frac{1}{4}} \mu^{\frac{1}{2}} \| \phi_{\mu, N} \|_{V^2_{+, 1}} \| \psi_{\lambda_1, N_1} \|_{V^2_{\pm_1, M}} \| \varphi_{\lambda_2, N_2} \|_{V^2_{\pm_2, M}}.$$
    Again combining this bound with a small power of the estimate from Lemma \ref{lem:tri L4 gain}, we get an acceptable contribution for the $A_1$ term. The proof for the $A_2$ term is identical. \\

        \textbf{Case 2: $\mu \gtrsim \min\{\lambda_1, \lambda_2\}$ and
          $(\min\{\lambda_1, \lambda_2\})^{-1} \lesa d \lesa \min\{\lambda_1, \lambda_2\}$.} We
        may assume that $\lambda_1 \g \lambda_2$. From \cite[(8.23)]{Candy2016}, we have for
        every $\epsilon>0$
    $$ \Big| \int_{\RR^{1+3}} A_0 dx dt \Big| + \Big| \int_{\RR^{1+3}} A_1 dx dt \Big|\lesa \Big( \frac{d}{\lambda_2} \Big)^{\frac{1}{8}-\epsilon} N^{\frac{1}{4}}_{min} \mu^\frac{1}{2} \| \phi_{\mu, N} \|_{V^2_{+,1}} \| \psi_{\lambda_1, N_1} \|_{V^2_{\pm_1, M}} \| \varphi_{\lambda_2, N_2} \|_{V^2_{\pm_2, M}}. $$
     An application of Lemma \ref{lem:tri L4 gain}, summing up over $\lambda_2^{-1} \lesa d \lesa \lambda_2$, then gives an acceptable bound for both $A_0$ and $A_1$.

    It remains to bound the $A_2$ term.  From \cite[(8.25)]{Candy2016}
    we have
            $$ \Big| \int_{\RR^{1+3}} A_2 dx dt \Big| \lesa \Big( \frac{d}{\lambda_2} \Big)^{\frac{1}{4} - \epsilon} \Big( \frac{\lambda_2}{\mu}\Big)^{\frac{1}{4} - \epsilon}  \min\{ N, N_1\} \mu^\frac{1}{2} \| \phi_{\mu, N} \|_{V^2_{+, 1}} \| \psi_{\lambda_1, N_1} \|_{V^2_{\pm_1, M}} \| \varphi_{\lambda_2, N_2} \|_{V^2_{\pm_2, M}}. $$
    As previously, together with  with Lemma \ref{lem:tri L4 gain}, this is enough to deduce the required estimates.  The remaining case where $N_2=N_{min}$, requires the use of the $Y^{\pm_2, M}_{\lambda_2}$ norm. To this end, a similar argument to (\ref{eqn:thm trilinear subcrit:case 2 A2}), implies that if we let $\beta=(\frac{d}{\lambda_2})^\frac{1}{2}$, then for every $\epsilon>0$ by using the angular concentration bound on $\varphi_{\lambda_2, N_2}$  we have
              \begin{align*}
              \Big| &\int_{\RR^{1+3}} A_2 dx dt \Big|  \\
                    &\lesa \sum_{\substack{q, q''\in Q_{\lambda_2}\\ |q-q''| \approx \lambda_2}} \sum_{\substack{\kappa, \kappa' \kappa'' \in \mc{C}_{\beta} \\ |\pm_1\kappa - \kappa''| \lesa \beta, \\ |\pm_1\kappa - \pm_2 \kappa'| \lesa \beta}}  \beta\|  P_{q''} R_{\kappa''} C^+_{\lesa d} \phi_{\mu, N} \|_{L^{\frac{2a}{a-1}}_t L^{2a}_x} \| R_{\kappa} P_{q}\mc{C}_{\lesa d}^{\pm_1} \psi_{\lambda_1, N_1} \|_{L^\frac{2a}{a-1}_t L^{2a}_x} \| R_{\kappa'} \mc{C}_d^{\pm_2} \varphi_{\lambda_2, N_2} \|_{L^a_t L^{\frac{a}{a-1}}_x} \\
                    &\lesa \Big( \frac{d}{\lambda_2}\Big)^{\frac{\varrho}{2} - b - \epsilon} \Big( \frac{\lambda_2}{\mu} \Big)^{\frac{1}{2a} - \frac{1}{4} - \epsilon} N_2^\varrho  \mu^{\frac{1}{2}} \| \phi_{\mu, N} \|_{V^2_{+, 1}} \| \psi_{\lambda_1, N_1} \|_{V^2_{\pm_1, M}}  \| H_{N_2} \varphi \|_{Y^{\pm_2, M}_{\lambda_2}}.
            \end{align*}
            Therefore, since $\frac{1}{a}=\frac{1}{2} + \frac{\varrho}{16}$, provided we choose $\epsilon>0$ sufficiently small, by combining the above estimate with Lemma \ref{lem:tri L4 gain} we get the claimed bound for the $A_2$ component. \\

            \textbf{Case 3: $\mu \gtrsim \min\{\lambda_1, \lambda_2\}$
              and $d \ll (\min\{\lambda_1, \lambda_2\})^{-1}$.} We now
            turn our attention to the \emph{resonant} region where we
            have no longer have a lower bound on $d$. It is worth noting that in this region the proof deviates somewhat from the argument in \cite{Candy2016}, as here we obtain an improvement in the amount of angular regularity required.

If $\mu \gg \min\{\lambda_1, \lambda_2\}$, \cite[Lemma 8.7]{Candy2016} implies that
\begin{equation}\label{eq:lower-bound-m}
\mathfrak{M}_{\pm_1,\pm_2}\gtrsim (\min\{\lambda_1, \lambda_2\})^{-1}
\end{equation}
which is ruled out in Case 3, hence from now on we may assume $\mu \approx \lambda_1 \approx \lambda_2$. If $\pm_1=\pm_2$, or $(\pm_1,\pm_2)=(-,+)$, or $M>\frac12$
\cite[Lemma 8.7]{Candy2016}  implies \eqref{eq:lower-bound-m} again, so that it remains to consider $(\pm_1,\pm_1)= (+,-)$ and we are either in the weakly resonant regime $M=\frac{1}{2}$, or the strongly resonant case $0<M<\frac{1}{2}$. In order to treat this case, we use
\begin{equation}\label{eq:res-bound-pm}
\begin{split}
\mathfrak{M}_{+, -}(\xi, \eta) \approx&  \frac{1}{\lr{\xi} + \lr{\eta}} \bigg| M^2 \frac{ (|\xi| - |\eta|)^2}{\lr{\xi}_{M} \lr{\eta}_{M} + |\xi| |\eta| + M^2}  + |\xi| |\eta| + \xi \cdot \eta + \frac{ 4 M^2 - 1}{2} \bigg|\\
\mathfrak{M}_{+, -}(\xi, \eta) \gtrsim & \frac{1}{\lr{\eta}} \bigg| \frac{ ( |\xi| - M |\xi - \eta|)^2}{\lr{\xi}_{M} \lr{\xi - \eta} + |\xi | |\xi - \eta| + M}  + |\xi| |\xi - \eta| - \xi \cdot (\xi - \eta)  + \frac{ 2M - 1 }{2} \bigg|
\end{split}
\end{equation}
 from \cite[Lemma 8.7]{Candy2016}.
We start by considering the case $M=\frac{1}{2}$. The key
            observation (originally made in \cite{Candy2016}) is that
            the null structure now acts at \emph{all} scales
            $0<d<\min\{\lambda_1, \lambda_2\}$. More precisely, \eqref{eq:res-bound-pm} and \eqref{eqn:nullsymbolbound} imply that for $|\xi|\approx|\eta|\approx \mu$
\[
 |\Pi_+(\xi)\gamma^0\Pi_-(\eta)|^2\lesa  \Big| \frac{||\xi|-|\eta||}{\mu^2}+ \theta(\xi,-\eta)\Big|^2\lesa \mu^{-1} \mathfrak{M}_{+, -}(\xi, \eta),
\]
which we exploit via \eqref{eqn:nullformbound2}.
In particular, letting $\beta = (\frac{d}{\mu})^{\frac{1}{2}}$, using \eqref{eq:res-bound-pm} and the $L^4_{t,x}$ Strichartz bound in Lemma \ref{lem:wave strichartz} we obtain
                \begin{align*}
                  \Big| \int_{\RR^{1+3}} A_0 &dx \,dt\Big| \\
                  &\lesa \beta \sum_{\substack{\kappa, \kappa', \kappa'' \in \mc{C}_{\beta} \\ |\kappa + \kappa'|, |\kappa - \kappa''| \lesa \beta}} \sum_{\substack{q, q'\in \mc{Q}_{\mu^2 \beta}  \\ |q-q'| \lesa \mu^2 \beta}} \| R_{\kappa''}  C_d \phi_{\mu, N} \|_{L^2_{t,x}} \| R_{\kappa} P_q \psi_{\lambda_1, N_1} \|_{L^4_{t,x}} \| R_{\kappa'} P_{q'} \varphi_{\lambda_2, N_2} \|_{L^4_{t,x}} \\
                  &\lesa (\beta \mu)^{\frac{1}{4}} \mu^{\frac{1}{2}} \| \phi_{\mu, N}\|_{V^2_{+, 1}}  \| \psi_{\lambda_1, N_1}\|_{V^2_{\pm_1, M}} \| \varphi_{\lambda_2, N_2} \|_{V^2_{\pm_2, M}}.
                \end{align*}
            Together with Lemma \ref{lem:tri L4 gain}, and as the sum over modulation $0<d \lesa \mu^{-1}$ is bounded, this gives control over the trilinear product when $M=\frac{1}{2}$. The arguments for the $A_1$ and $A_2$ terms are essentially identical.

            It remains to consider the fully resonant case $0<M<\frac{1}{2}$. In this regime the null structure no longer gives any gain at modulation scales $d\ll \mu^{-1}$. Consequently if we followed the argument used in the $M=\frac{1}{2}$ case, we would not be able to sum up over modulation scales. Instead, our goal will be to simply estimate the remaining trilinear term
                    $$ \int_{\RR^{1+3}} C_{\ll \mu^{-1}} \phi_{\mu, N} \overline{\mc{C}_{\ll \mu^{-1}}^{+} \psi_{\lambda_1, N_1}} \mc{C}_{\ll \mu^{-1}}^{-} \varphi_{\lambda_2, N_2} dx dt $$
            directly. The key observation, which was exploited in \cite{Candy2016}, is that in this trilinear interaction the three waves are already transverse, and thus we can apply the bilinear restriction estimates contained in Theorem \ref{thm:bilinear small scale KG}. The argument is as follows. We first observe that, by Lemma \ref{lem:wave strichartz} and $L^p$ interpolation, for every $\frac{10}{3}<\frac{1}{r} < \frac{5}{14}$ there exists $\theta>0$ such that
                    $$ \| \phi_{\mu, N} \|_{L^r_{t,x}} \lesa \mu^{\frac{3}{14}} N^{\frac{2}{5}} \| \phi_{\mu, N} \|_{L^4_{t,x}}^\theta \|\phi_{\mu, N} \|_{V^2_{+, 1}}^{1-\theta}.$$
            Interpolating with the trivial $L^\infty_t L^2_x$ estimate, we conclude that for all $\frac{3}{2}(\frac{1}{2}-\frac{1}{r})<\frac{1}{q}<\frac{5}{2}(\frac{1}{2} - \frac{1}{r})$ and sufficiently small $\theta>0$ the bound
                     \begin{equation}\label{eqn:proof of thm trilinear crit:case 3 phi linear bound}\| \phi_{\mu, N} \|_{L^q_t L^r_x} \lesa \mu^{\frac{3}{7q} + \frac{3}{7}(\frac{1}{2} - \frac{1}{r})} N^{\frac{4}{5q} + \frac{4}{5}(\frac{1}{2}-\frac{1}{r})} \| \phi_{\mu, N} \|_{L^4_{t,x}}^\theta \|\phi_{\mu, N} \|_{V^2_{+, 1}}^{1-\theta}.  \end{equation}
           On the other hand, by interpolating Theorem \ref{thm:bilinear small scale KG} with H\"older's inequality and exploiting null structure, we have for $1 \les q,r < 2$, $\frac{1}{q} + \frac{2}{r}<2$, and all sufficiently small $\theta>0$
                \begin{align*} \sum_{\kappa, \kappa'\in \mc{C}_{\mu^{-1}}} \sum_{\substack{q, q' \in \mc{Q}_{\mu} \\ \frac{|q-q'|}{\mu^2} + |\kappa +\kappa'| \approx \mu^{-1}}}& \big\| \overline{ R_{\kappa} P_q \mc{C}_{\ll \mu^{-1}}^{+}\psi_{\lambda_1, N_1}} R_{\kappa'} P_{q'} \mc{C}_{\ll \mu^{-1}}^{-}\varphi_{\lambda_2, N_2} \big\|_{L^q_t L^r_x}\\
                    &\lesa \mu^{\frac{1}{q}-\frac{1}{r} + 6\theta} \Big( \lambda_1^{-\frac{1}{2}} \|\psi_{\lambda_1, N_1} \|_{L^4_{t,x}} \lambda_2^{-\frac{1}{2}} \| \varphi_{\lambda_2, N_2} \|_{L^4_{t,x}}\Big)^{\theta} \Big( \| \psi_{\lambda_1, N_1} \|_{V^2_{\pm_1, M}} \| \varphi_{\lambda_2, N_2} \|_{V^2_{\pm_2, M}} \Big)^{1-\theta}.
                \end{align*}
           Therefore, if we let $\frac{1}{r} = \frac{1}{2} +\epsilon$ and $\frac{1}{q} = 1 - \frac{9}{4}\epsilon$, then the above estimates, together with the orthogonality implied by \eqref{eq:res-bound-pm}, give
            \begin{align*}
              \Big| \int_{\RR^{1+3}} &C_{\ll \mu^{-1}} \phi_{\mu, N} \overline{  \mc{C}_{\ll \mu^{-1}}^{+}\psi_{\lambda_1, N_1}} \mc{C}_{\ll \mu^{-1}}^{-}\varphi_{\lambda_2, N_2} dx dt \Big| \\
                &\lesa \sum_{\kappa, \kappa'\in \mc{C}_{\mu^{-1}}} \sum_{\substack{q, q' \in \mc{Q}_{\mu} \\ \frac{|q-q'|}{\mu^2} + |\kappa + \kappa'| \approx \mu^{-1}}} \| \phi_{\mu, N} \|_{L^{q'}_t L^{r'}_x} \big\| \overline{ R_{\kappa} P_q \mc{C}_{\ll \mu^{-1}}^{+}\psi_{\lambda_1, N_1}} R_{\kappa'} P_{q'} \mc{C}_{\ll \mu^{-1}}^{-}\varphi_{\lambda_2, N_2} \big\|_{L^q_t L^r_x}\\
                &\lesa \mu^{7\theta - \frac{13}{4} \epsilon} N^{\frac{13}{5}\epsilon} \mb{A}^\theta \Big( \mu^{\frac{1}{2}} \| \phi_{\mu, N} \|_{V^2_{+, 1}} \| \psi_{\lambda_1, N_1} \|_{V^2_{\pm_1, M}} \| \varphi_{\lambda_2, N_2} \|_{V^2_{\pm_2, M}} \Big)^{1-\theta}.
            \end{align*}
            Choosing $\epsilon$ and $\theta$ sufficiently small, we obtain the required bound in the case $N_{min} = N$. The remaining cases $N_1=N_{min}$ and $N_2=N_{min}$ are similar, the only change is to use (\ref{eqn:proof of thm trilinear crit:case 3 phi linear bound}) on the term with smallest angular frequency, and control the remaining pair using the bilinear restriction estimate in Theorem \ref{thm:bilinear small scale KG}. \\

            \textbf{Case 4: $d \gg \min\{\mu, \lambda_1, \lambda_2\} $.} We start by estimating the $A_0$ component. As in the subcritical case, nontrivial contributions require $\mathfrak{M}_{\pm_1, \pm_2} \approx d$. From the definition of $\mathfrak{M}_{\pm_1,\pm_2}$ we see that either $\mathfrak{M}_{\pm_1, \pm_2} \lesa \min\{\mu, \lambda_1, \lambda_2\}$ or $\mathfrak{M}_{\pm_1, \pm_2} \approx \max\{\mu, \lambda_1, \lambda_2\}$. In conclusion, we must have $d \approx \max\{\mu, \lambda_1, \lambda_2\}$. If $\mu \lesa \lambda_1 \approx \lambda_2$, an application of Theorem \ref{thm:cheap bilinear} gives
                \begin{align*}
                  \Big| \int_{\RR^{1+3}} A_0 dx dt \Big| &\les \| C_d \phi_{\mu, N} \|_{L^2_{t,x}} \| \overline{\mc{C}_{\ll d}^{\pm_1}\psi_{\lambda_1, N_1}} \mc{C}_{\ll d}^{\pm_2}\varphi_{\lambda_2, N_2} \|_{L^2_{t,x}} \\
                  &\lesa \Big( \frac{d}{\lambda_2} \Big)^{-\frac{1}{2}} \Big( \frac{\mu}{\lambda_2} \Big)^{\frac{1}{2}-\epsilon} \mu^\frac{1}{2}  \| \phi_{\mu, N} \|_{V^2_{+, 1}} \| \psi_{\lambda_1, N_1} \|_{V^2_{\pm_1, m_1}} \| \varphi_{\lambda_2, N_2} \|_{V^2_{\pm_2, m_2}}.
                \end{align*}
            On the other hand, if $\mu \gg \min\{ \lambda_1, \lambda_2\}$ (thus $\mu$ is essentially the largest frequency), we simply have
                \begin{align*}
                  \Big| \int_{\RR^{1+3}} A_0 dx dt \Big| &\les \| C_d \phi_{\mu, N} \|_{L^2_{t,x}} \|\mc{C}_{\ll d}^{\pm_1} \psi_{\lambda_1, N_1}\|_{L^4_{t,x}} \| \mc{C}_{\ll d}^{\pm_2}\varphi_{\lambda_2, N_2} \|_{L^4_{t,x}} \\
                  &\lesa \Big( \frac{d}{\mu} \Big)^{-\frac{1}{2}} \Big( \frac{\min\{\lambda_1, \lambda_2\}}{\mu} \Big)^{\frac{1}{2}} \mu^\frac{1}{2}  \| \phi_{\mu, N} \|_{V^2_{+, 1}} \| \psi_{\lambda_1, N_1} \|_{V^2_{\pm_1, m_1}} \| \varphi_{\lambda_2, N_2} \|_{V^2_{\pm_2, m_2}}.
                \end{align*}
            Thus in either case we have a high-low gain, and consequently applying Lemma \ref{lem:tri L4 gain} to gain $L^4_{t,x}$ norms, and summing up over $d \approx \max\{\mu, \lambda_1, \lambda_2\}$ we obtain the required bound for the $A_0$ component.

            To estimate the $A_1$ component, as in the subcritical case, we consider separately the cases $\min\{\mu, \lambda_1, \lambda_2\} \ll d \ll \max\{\mu, \lambda_1, \lambda_2\}$ and $d \gtrsim \max\{\mu, \lambda_1, \lambda_2\}$. In the latter case, we only require H\"{o}lder's inequality together with the (refined) $L^4_{t,x}$ Strichartz estimate. In particular, decomposing into cubes of size $\lambda_{min} = \min\{ \mu, \lambda_1, \lambda_2\}$, we obtain
                \begin{align*}
                  \Big| \int_{\RR^{1+3}} A_1 dx dt \Big|
                        &\lesa \sum_{\substack{q, q', q'' \in \mc{Q}_{\lambda_{min}} \\ |q - q' + q''| \lesa \lambda_{min}}}
                                \| P_{q''} \phi_{\mu, N} \|_{L^4_{t,x}} \| P_q \mc{C}^{\pm_1}_d \psi_{\lambda_1, N_1} \|_{L^2_{t,x}} \| P_{q'} \varphi_{\lambda_2, N_2} \|_{L^4_{t,x}} \\
                        &\lesa \Big( \frac{d}{\lambda_2}\Big)^{-\frac{1}{2}} \Big( \frac{\lambda_{min}}{\lambda_2}\Big)^{\frac{1}{4}-\epsilon} \Big( \frac{\lambda_{min}}{\mu}\Big)^{\frac{1}{4}-\epsilon}   \mu^\frac{1}{2} \| \phi_{\mu, N} \|_{V^2_{+,1}} \| \psi_{\lambda_1, N_1} \|_{V^2_{\pm_1, M}} \| \varphi_{\lambda_2, N_1}\|_{V^2_{\pm_2, M}}.
                \end{align*}
            Again applying Lemma \ref{lem:tri L4 gain} and summing up over $d \gtrsim \max\{\mu, \lambda_1, \lambda_2\}$ controls the $A_1$ term.

            We now consider the region $\min\{\mu, \lambda_1, \lambda_2\} \ll d \ll \max\{\mu, \lambda_1, \lambda_2\}$. Here we can simply observe that the bounds in the subcritical case, namely \eqref{eqn:proof of thm tri sub:c3 A1 med mod} and \eqref{eqn:proof of thm tri sub:c4 A1 restricted d}, imply that we have
                $$\Big| \int_{\RR^{1+3}}A_1 dx dt \Big| \lesa \Big( \frac{d}{\min\{\mu, \lambda_1, \lambda_2\}}\Big)^{-\frac{1}{2}} \Big( \frac{\min\{\mu, \lambda_1, \lambda_2\}}{\max\{\mu, \lambda_1, \lambda_2\}}\Big)^{\frac{1}{4}}  \mu^\frac{1}{2} \| \phi_{\mu, N} \|_{V^2_{+,1}} \| \psi_{\lambda_1, N_1} \|_{V^2_{\pm_1, M}} \| \varphi_{\lambda_2, N_1}\|_{V^2_{\pm_2, M}}.$$
            Again applying Lemma \ref{lem:tri L4 gain} and summing up over $\min\{\mu, \lambda_1, \lambda_2\} \ll d \ll \max\{\mu, \lambda_1, \lambda_2\}$, we deduce the required bound for the $A_1$ term. An identical argument bounds the $A_2$ term.
          \end{proof}

          \subsection{Proof of Theorem \ref{thm:duhamel-crit}}\label{subsec:proof-duhamel-crit}
Similarly to Subsection \ref{subsec:proof-duhamel-sub}, the first step is to obtain frequency localised versions of the required bounds. Let $\lambda_{min}= \min\{\mu, \lambda_1, \lambda_2\}$, $\lambda_{max}=\max\{\mu, \lambda_1, \lambda_2\}$. Our aim is to show that, if $\varrho>0$ is sufficiently small, $\frac{1}{a}=\frac{1}{2} + \frac{\varrho}{16}$, and $b= 2 (\frac{1}{a}-\frac{1}{2})$, then
there exists $0<\theta_1<\frac{1}{4}$ such that for all $0\les \theta \les \theta_1$ we have for the Dirac Duhamel term, the bounds
\begin{equation}\label{eqn:proof of thm duhamel crit:psi bd V2}\begin{split}
    \big\|  & P_{\lambda_1} H_{N_1} \Pi_{\pm_1}\mc{I}^{\pm_1,M}\big( \phi_{\mu, N} \gamma^0 \Pi_{\pm_2} \varphi_{\lambda_2, N_2} \big) \big\|_{V^2_{\pm_1, M}} \\
        &\lesa  (\min\{N, N_2\})^{\varrho} \Big( \frac{\lambda_{min}}{\lambda_{max}}\Big)^{\frac{\varrho}{100}} \Big( \|
        \phi_{\mu, N}\|_{L^4_{t,x}} \lambda_2^{-\frac{1}{2}}\| \varphi_{\lambda_2, N_2} \|_{L^4_{t,x}} \Big)^{\theta} \Big(  \mu^{\frac{1}{2}} \|
        \phi_{\mu, N} \|_{V^2_{+,1}} \| \varphi_{N_2} \|_{F^{\pm_2,M}_{\lambda_2}} \Big)^{1-\theta}
  \end{split}
\end{equation}
and
    \begin{equation}\label{eqn:proof of thm duhamel crit:psi bd Y}
        \begin{split}
          \big\|  &H_{N_1} \Pi_{\pm_1}  \mc{I}^{\pm_1,M}\big( \phi_{\mu, N} \gamma^0 \Pi_{\pm_2} \varphi_{\lambda_2, N_2} \big) \big\|_{Y^{\pm_1, M}_{\lambda_1}} \\
    &\lesa (\min\{N, N_2\})^{\varrho} \Big( \frac{\lambda_{min}}{\lambda_{max}}\Big)^{\frac{\varrho}{100}}  \Big( \|
    \phi_{\mu, N}\|_{L^4_{t,x}} \lambda_2^{-\frac{1}{2}}\| \varphi_{\lambda_2, N_2} \|_{L^4_{t,x}} \Big)^{\theta} \Big(  \mu^{\frac{1}{2}} \|
    \phi_{\mu, N} \|_{V^2_{+,1}} \| \varphi_{\lambda_2, N_2} \|_{V^2_{\pm_2,
        M}} \Big)^{1-\theta},
        \end{split}
    \end{equation}
while for the wave Duhamel term, we have
\begin{equation}\label{eqn:proof of thm duhamel crit:phi bd V2}\begin{split}
    &\mu^{-\frac{1}{2}} \big\|  P_{\mu} H_N \mc{I}^{+,1}\big( \overline{\Pi_{\pm_1} \psi_{\lambda_1, N_1}} \Pi_{\pm_2}\varphi_{\lambda_2, N_2}\big) \big\|_{V^2_{+,1}} \\
    &\lesa  (\min\{N_1, N_2\})^{\varrho} \Big(
    \frac{\lambda_{min}}{\lambda_{max}}\Big)^{\frac{\varrho}{100}} \Big( \lambda_1^{-\frac{1}{2}} \|
    \psi_{\lambda_1, N_1} \|_{L^4_{t,x}} \lambda_2^{-\frac{1}{2}} \| \varphi_{\lambda_2, N_2} \|_{L^4_{t,x}} \Big)^{\theta} \Big(  \| \psi_{N_1}
    \|_{F^{\pm_1, M}_{\lambda_1}} \| \varphi_{N_2} \|_{F^{\pm_2,
        M}_{\lambda_2}} \Big)^{1-\theta}.
  \end{split}
\end{equation}
Assuming the bounds \eqref{eqn:proof of thm duhamel crit:psi bd V2}, \eqref{eqn:proof of thm duhamel crit:psi bd Y}, and \eqref{eqn:proof of thm duhamel crit:phi bd V2} for the moment, the estimates in Theorem \ref{thm:duhamel-crit} are a consequence of a straightforward summation argument. Fix $\sigma>0$.  As in the subcritical case, it is enough to consider the case $s=0$ by using that, due to the convolution constraint, we always have $\lambda_1^s \lesa  (\max\{ \mu, \lambda_2 \})^{s}$.
Summing up \eqref{eqn:proof of thm duhamel crit:psi bd V2} with $\varrho = \frac{\sigma}{2}$ over angular frequencies $N_1$ gives $0<\theta_0<\frac{1}{2}$ such that for all $0<\theta<\theta_0$ we have
\begin{equation}\label{eqn:proof of thm duhamel crit:ang summed up}
    \begin{split}
       \Big( \sum_{N_1 \in 2^{\NN}} N_1^{2\sigma} \big\|  P_{\lambda_1} H_{N_1} &\Pi_{\pm_1}\mc{I}^{\pm_1}_{M}\big( \phi_{\mu} \gamma^0 \Pi_{\pm_2} \varphi_{\lambda_2} \big) \big\|_{V^2_{\pm_1, M}}^2 \Big)^\frac{1}{2} \\
        &\lesa  \Big( \frac{\lambda_{min}}{\lambda_{max}}\Big)^{\frac{\sigma}{200}} \Big( \|
        \phi_{\mu}\|_{\mb{D}^{0,\sigma}} \| \varphi_{\lambda_2} \|_{\mb{D}^{-\frac12,\sigma}} \Big)^{\theta} \\
        &\qquad \qquad \cdot \Big[  \mu^{\frac{1}{2}} \Big( \sum_{N \in 2^\NN} N^{2\sigma} \|
        \phi_{\mu, N} \|_{V^2_{+,1}}^2 \Big)^\frac{1}{2} \Big( \sum_{N_2 \in 2^\NN} N_2^{2\sigma} \| \varphi_{N_2} \|_{F^{\pm_2,M}_{\lambda_2}}^2\Big)^\frac{1}{2} \Big]^{1-\theta}.
    \end{split}
\end{equation}
Note that for $\frac{1}{q} + \frac{1}{r} = \frac{1}{2}$ and $\epsilon>0$ we have the elementary inequality
\begin{align*}
& \bigg\| \bigg( \sum_{\substack{ \mu, \lambda_2 \in 2^{\NN} \\ \mu +\lambda_2 \approx \lambda_1}} \Big( \frac{\min\{\mu, \lambda_2\}}{\lambda_1}\Big)^{\epsilon} a_{\mu} b_{\lambda_2}\bigg)_{\lambda_1 \in 2^{\NN}} \bigg\|_{\ell^2}
            + \bigg\| \bigg( \sum_{\substack{ \mu, \lambda_2 \in 2^{\NN} \\ \mu \approx \lambda_2 \gtrsim \lambda_1}} \Big( \frac{\lambda_1}{\lambda_2}\Big)^{\epsilon} a_{\mu} b_{\lambda_2}\bigg)_{\lambda_1 \in 2^{\NN}} \bigg\|_{\ell^2}
            \\\lesa{}&  \| (a_{\mu})_{\mu \in 2^{\NN}} \|_{\ell^q} \| (b_{\lambda_2})_{\lambda_2 \in 2^{\NN} }\|_{\ell^r}.
\end{align*}
Thus summing up (\ref{eqn:proof of thm duhamel crit:ang summed up}) over spatial frequencies, and assuming that $0<\theta<\frac{1}{4}$, we deduce the bound \eqref{eqn:thm duhamel crit:V psi}.  An identical argument using (\ref{eqn:proof of thm duhamel crit:psi bd Y}) gives the $\mb{Y}^s_{\pm_1, M}$ bound (\ref{eqn:thm duhamel crit:Y psi}). Similarly the bound for $\phi$ follows from (\ref{eqn:proof of thm duhamel crit:phi bd V2}).

We now turn to the proof of the estimates \eqref{eqn:proof of thm duhamel crit:psi bd V2}, \eqref{eqn:proof of thm duhamel crit:psi bd Y}, and \eqref{eqn:proof of thm duhamel crit:phi bd V2}. It is enough to consider the case $\theta=\theta_1$, as the $L^4_{t,x}$ terms are dominated by the corresponding $V^2$ norms.  The bounds (\ref{eqn:proof of thm duhamel crit:psi bd V2}) and (\ref{eqn:proof of thm duhamel crit:phi bd V2}) follow directly from Theorem \ref{thm:tri freq loc crit} together with \eqref{eqn:energy ineq for V2}. On the other hand, the argument used to obtain (\ref{eqn:proof of thm duhamel crit:psi bd Y}) is slightly more involved. We first note that, from \cite[(8.38)]{Candy2016} and \cite[(8.39)]{Candy2016}, we have the bound
    \begin{equation}\label{eqn:proof of thm duhamel crit:old}
    \begin{split}
      d^{\frac{2}{3}} \Big( \frac{\min\{d,\lambda_1\}}{\lambda_1}\Big)^{1-\frac{2}{3}} \big\| P_{\lambda_1} H_{N_1} &\mc{C}_d^{\pm_1}\mc{I}^{\pm_1,M}\big( \phi_{\mu, N} \gamma^0 \Pi_{\pm_2} \varphi_{\lambda_2, N_2}\big) \big\|_{L^\frac{3}{2}_t L^2_x}\\
            &\lesa \min\{N, N_2\} \Big( \frac{\lambda_{max}}{\lambda_{min}}\Big)^{\frac{1}{3}} \mu^{\frac{1}{2}} \| \phi_{\mu, N} \|_{V^2_{+, 1}} \| \psi_{\lambda_1, N_1} \|_{V^2_{\pm_2, M}}.
    \end{split}
    \end{equation}
On the other hand, if $\mu \lesa \lambda_2$, then an application of \eqref{eqn:tri freq loc crit V2:main}, gives
\begin{align}
  d^{\frac{1}{2}} \big\| P_{\lambda_1} H_{N_1} &\mc{C}_d^{\pm_1}\mc{I}^{\pm_1,M}\big( \phi_{\mu, N} \gamma^0 \Pi_{\pm_2} \varphi_{\lambda_2, N_2}\big) \big\|_{L^2_{t,x}} \notag \\
                                                           &\lesa \big\| P_{\lambda_1} H_{N_1} \mc{C}_d^{\pm_1}\mc{I}^{\pm_1,M}\big( \phi_{\mu, N} \gamma^0 \Pi_{\pm_2} \varphi_{\lambda_2, N_2}\big) \big\|_{V^2_{\pm_1, M}} \notag \\
                                                           &\lesa (\min\{N, N_2\})^{\frac{\varrho}{2}} \Big(
                                                             \frac{\mu}{\lambda_2}\Big)^{\frac{1}{10}} \Big( \| \phi_\mu\|_{L^4_{t,x}} \lambda_2^{-\frac{1}{2}} \| \varphi_{\lambda_2} \|_{L^4_{t,x}} \Big)^{\theta_0}
                                                             \Big( \mu^{\frac{1}{2}} \| \phi_\mu \|_{V^2_{+,1}} \| \varphi_{\lambda_2} \|_{V^2_{\pm_2, M}}
                                                             \Big)^{1-\theta_0}.\label{eqn-proof of F control crit-L2 bound1}
\end{align}
Hence (\ref{eqn:proof of thm duhamel crit:psi bd Y}) in the region $\mu \lesa \lambda_2$ follows by interpolating between \eqref{eqn:proof of thm duhamel crit:old} and (\ref{eqn-proof of F control crit-L2 bound1}) and using the conditions $\frac{1}{a} = \frac{1}{2} + \frac{\varrho}{16}$ and $b=2(\frac{1}{a}-\frac{1}{2})$. The case $\mu \gg \lambda_2$ and $N \les N_2$, follows from a similar argument using \eqref{eqn:tri freq loc crit V2:main-A0} and \eqref{eqn:tri freq loc crit:d bound with V}.

It remains to consider the case $\mu \gg \lambda_2$ and $N_2 \les N$. For this frequency interaction, Theorem \ref{thm:tri freq loc crit} requires a $Y^{\pm, M}_{\lambda_2}$ norm on the righthand side. Thus, as our goal is to obtain a bound only using the $V^2_{\pm, M}$ norms, we have to work a little harder. We start by writing the product as
     \begin{equation}\label{eqn-proof of F control crit-Lp bound wave high}
     \begin{split}
     \phi_{\mu, N} \gamma^0 \Pi_{\pm_2} \varphi_{\lambda_2, N_2}
        = \Big(\phi_{\mu, N} \gamma^0 \Pi_{\pm_2} \varphi_{\lambda_2, N_2} - \sum_{d' \lesa \lambda_2} &C_{\les d'}^{\pm_1, M} \big( C_{\les d'} \phi_{\mu, N} \gamma^0 \mc{C}^{\pm_2}_{d'} \varphi_{\lambda_2, N_2} \big) \Big)\\
        &+ \sum_{d' \lesa \lambda_2} C_{\les d'}^{\pm_1, M} \big( C_{\les d'} \phi_{\mu, N} \gamma^0 \mc{C}^{\pm_2}_{d'} \varphi_{\lambda_2, N_2} \big).
     \end{split}
     \end{equation}
The first term can be bounded by adapting the argument used in the previous  cases as here \eqref{eqn:tri freq loc crit V2:main-A0} in  Theorem \ref{thm:tri freq loc crit} gives a bound without using the $Y^{\pm, M}_{\lambda_2} $ norm. More precisely, letting $\beta = (\frac{d'}{\lambda_2})^\frac{1}{2}$ and exploiting null structure together with the now familiar modulation bounds, we have
 \begin{align*}
     d^{\frac{2}{3}} \Big( \frac{d}{\lambda_1} \Big)^{1-\frac{2}{3}} \Big\| P_{\lambda_1} H_{N_1}& \mc{C}^{\pm_1}_d \mc{I}^{\pm_1, M}\Big( \sum_{d' \lesa \lambda_2} C_{\les d'}^{\pm_1, M} \big( C_{\les d'} \phi_{\mu, N} \gamma^0 \mc{C}^{\pm_2}_{d'} \varphi_{\lambda_2, N_2} \big)\Big) \Big\|_{L^\frac{3}{2}_t L^2_x}\\
                                                                          &\lesa  \mu^{-\frac{1}{3}} \sum_{d'\lesa \lambda_2}  \Big\| \Big( \sum_{\substack{\kappa, \kappa', \kappa''  \in \mc{C}_{\beta}\\|\pm_1 \kappa - \pm_2 \kappa'|, |\kappa - \kappa''|\lesa \beta} }\big\| R_{\kappa} \big(C_{\les d'}R_{\kappa''}\phi_{\mu, N} \gamma^0 R_{\kappa'} \mc{C}^{\pm_2}_{d'} \varphi_{\lambda_2, N_2}\big) \big\|_{L^2_x}^2\Big)^\frac{1}{2}
                                                                          \Big\|_{L^\frac{3}{2}_t} \notag \\
                                                                          &\lesa \mu^{-\frac{1}{3}} \sum_{d' \lesa \lambda_2} \beta \Big( \sum_{\kappa'' \in \mc{C}_\beta} \|  C_{\les d'} R_{\kappa''}\phi_{\mu, N} \|_{L^4_{t,x}}^2\Big)^\frac{1}{2} \| \varphi_{\lambda_2, N_2} \|_{L^\frac{12}{5}_t L^4_x}  \notag \\
                                                                          &\lesa N_2 \Big( \frac{\lambda_2}{\mu} \Big)^{\frac{1}{3}} \mu^{\frac{1}{2}} \| \phi_{\mu, N} \|_{V^2_{+, 1}} \| \varphi_{\lambda_2, N_2} \|_{V^2_{\pm_2, M}}.
    \end{align*}
 Together with \eqref{eqn:proof of thm duhamel crit:old}, we deduce that
    \begin{align*}  d^{\frac{2}{3}} \Big( \frac{\min\{d,\lambda_1\}}{\lambda_1} \Big)^{1-\frac{2}{3}} \Big\| P_{\lambda_1} H_{N_1} \mc{C}^{\pm_1}_d \mc{I}^{\pm_1, M}\Big( \phi_{\mu, N} \gamma^0  \varphi_{\lambda_2, N_2}-& \sum_{d' \lesa \lambda_2} C_{\les d'}^{\pm_1, M} \big( C_{\les d'} \phi_{\mu, N} \gamma^0 \mc{C}^{\pm_2}_{d'} \varphi_{\lambda_2, N_2} \big)\Big) \Big\|_{L^\frac{3}{2}_t L^2_x}\\
            &\lesa N \Big( \frac{\mu}{\lambda_2} \Big)^{\frac{1}{3}} \mu^{\frac{1}{2}} \| \phi_{\mu, N} \|_{V^2_{+, 1}} \| \varphi_{\lambda_2, N_2} \|_{V^2_{\pm_2, M}}.
    \end{align*}
 Applying $L^p$ interpolation together with \eqref{eqn:tri freq loc crit V2:main-A0} and arguing as previously, we deduce that
         \begin{align*}
          \Big\|  H_{N_1} \Pi_{\pm_1}  \mc{I}^{\pm_1, M}\Big(\phi_{\mu, N} &\gamma^0  \varphi_{\lambda_2, N_2}- \sum_{d' \lesa \lambda_2} C_{\les d'}^{\pm_1, M} \big( C_{\les d'} \phi_{\mu, N} \gamma^0 \mc{C}^{\pm_2}_{d'} \varphi_{\lambda_2, N_2} \big)\Big) \Big\|_{Y^{\pm_1, M}_{\lambda_1}} \\
    &\lesa N^{\varrho} \Big( \frac{\lambda_2}{\mu}\Big)^{\frac{\varrho}{100}}  \Big( \|
    \phi_{\mu, N}\|_{L^4_{t,x}} \lambda_2^{-\frac{1}{2}}\| \varphi_{\lambda_2, N_2} \|_{L^4_{t,x}} \Big)^{\theta} \Big(  \mu^{\frac{1}{2}} \|
    \phi_{\mu, N} \|_{V^2_{+,1}} \| \varphi_{\lambda_2, N_2} \|_{V^2_{\pm_2,
        M}} \Big)^{1-\theta}.
        \end{align*}
Therefore it only remains to bound the second term in (\ref{eqn-proof of F control crit-Lp bound wave high}), but this follows by taking $\frac{1}{r}=1-2(\frac{1}{a}-\frac{1}{2})$ in \eqref{eqn:proof of thm duhamel sub:Y bd without loss}.

          \bibliographystyle{amsplain} \bibliography{remark}
        \end{document}